\documentclass [11pt,reqno]{amsart}
\usepackage{amsmath,amssymb,verbatim,geometry,color}
\usepackage[all]{xy}
\usepackage{mathrsfs}
\usepackage[backref,pagebackref,pdftex,hyperindex]{hyperref}
\usepackage[pdftex]{graphicx}
\def\XXint#1#2#3{{\setbox0=\hbox{$#1{#2#3}{\int}$ }
\vcenter{\hbox{$#2#3$ }}\kern-.6\wd0}}

\newcommand{\C}{\mathbb{C}}

\newcommand{\N}{\mathbb{N}}

 \newcommand{\Q}{\mathbb{Q}}
 \newcommand{\R}{\mathbb{R}}
 \newcommand{\Z}{\mathbb{Z}}

\newcommand{\ft}{\mathfrak{t}}

\newcommand{\hnu}{\widehat{\nu}}
\newcommand{\hro}{\widehat{\rho}}

\newcommand{\regu}{\mathrm{reg}}

\newcommand{\tf}{\widetilde{\varphi}}

\newcommand{\cE}{\mathcal{E}}
\newcommand{\cF}{\mathcal{F}}

\newcommand{\cH}{\mathcal{H}}

\newcommand{\cL}{\mathcal{L}}
\newcommand{\cM}{\mathcal{M}}

\newcommand{\cO}{\mathcal{O}}

\newcommand{\cR}{\mathcal{R}}

\newcommand{\cU}{\mathcal{U}}

\newcommand{\cZ}{\mathcal{Z}}

\renewcommand{\a}{\alpha}
\renewcommand{\b}{\beta}
\renewcommand{\d}{\delta}
\newcommand{\e}{\varepsilon}
\newcommand{\f}{\varphi}

\newcommand{\om}{\omega}
\newcommand{\Om}{\Omega}

\newcommand{\p}{\psi}

\newcommand{\cf}{{\rm cf.\ }} 
\newcommand{\eg}{{\rm e.g.\ }} 
\newcommand{\ie}{{\rm i.e.\ }}

\newcommand{\winter}{\wedge\dots\wedge}

\renewcommand{\DH}{\mathrm{DH}}

\newcommand{\cET}{\cE^{1,T}}

\newcommand{\enR}{\mathrm{R}}
\newcommand{\henR}{\widehat{\mathrm{R}}}

\DeclareMathOperator{\en}{E}

\DeclareMathOperator{\ent}{H}
\DeclareMathOperator{\hent}{\widehat{H}}
\DeclareMathOperator{\mab}{M}

\DeclareMathOperator{\ii}{I}
\DeclareMathOperator{\jj}{J}
\DeclareMathOperator{\env}{P}

\DeclareMathOperator{\Aut}{Aut}

\DeclareMathOperator{\Ent}{Ent}

\DeclareMathOperator{\Fut}{Fut}

\DeclareMathOperator{\Lah}{Lah}

\DeclareMathOperator{\MA}{MA}

\DeclareMathOperator{\id}{id}

\DeclareMathOperator{\PSH}{PSH}

\DeclareMathOperator{\tr}{tr}

\newcommand{\reg}{\mathrm{reg}}
\DeclareMathOperator{\Ric}{Ric}

\DeclareMathOperator{\sing}{sing}

\DeclareMathOperator{\Hnot}{H^0}

\DeclareMathOperator{\dd}{{d}}

\newcommand{\ddc}{dd^c}
\newcommand{\dc}{d^c}
\newcommand{\de}{d}
\newcommand{\ddcT}{dd^c_T}

\DeclareMathOperator{\ddbar}{\partial\overline{\partial}}

\DeclareMathOperator{\Lie}{Lie}

\DeclareMathOperator{\Alb}{Alb}

\newcommand{\dbar}{\overline{\partial}}

\newcommand{\ext}{\mathrm{ext}}

\newcommand{\rel}{\mathrm{rel}}

\newcommand{\Autr}{\mathrm{LAut}}

\newcommand{\D}{\Delta}

\newcommand{\simto}{\overset\sim\to}

\numberwithin{equation}{section}       

\newtheorem{prop} {Proposition} [section]
\newtheorem{thm}[prop] {Theorem} 
\newtheorem{defi}[prop] {Definition}
\newtheorem{lem}[prop] {Lemma}
\newtheorem{cor}[prop]{Corollary}
\newtheorem{prop-def}[prop]{Proposition-Definition}

\newtheorem*{thmA}{Theorem A}

\newtheorem*{corB}{Corollary B}

\newtheorem{exam}[prop]{Example}
\newtheorem{rmk}[prop]{Remark}

\theoremstyle{remark}
\newtheorem*{ackn}{Acknowledgment}

\title{Weighted extremal Kähler metrics on resolutions of singularities}
\date{\today}

\author{S{\'e}bastien Boucksom \and Mattias Jonsson \and Antonio Trusiani}

\address{Sorbonne Universit\'e and Universit\'e Paris Cit\'e\\
CNRS\\
IMJ-PRG\\
F-75005 Paris\\
France}
\email{sebastien.boucksom@imj-prg.fr}

\address{Dept of Mathematics\\
  University of Michigan\\
  Ann Arbor, MI 48109-1043\\
  USA}
\email{mattiasj@umich.edu}

\address{Dept of Mathematics\\
University of Rome Tor Vergata\\
IT 00133 Rome\\
Italy}
\email{trusiani@mat.uniroma2.it}


\begin{document} 

\begin{abstract} Generalizing previous results of Arezzo--Pacard--Singer, Seyyedali--Sz\'ekelyhidi, and Hallam, we prove the invariance under smooth blowups of the class of weighted extremal Kähler manifolds, modulo a log-concavity assumption on the first weight. Through recent work of Di Nezza--Jubert--Lahdili and Han--Liu, this is obtained as a consequence of a general uniform coercivity estimate for the (relative, weighted) Mabuchi energy on the blowup, which applies more generally to any equivariant resolution of singularities of Fano type of a compact Kähler klt space whose Mabuchi energy is assumed to be coercive. 
\end{abstract}

\maketitle

\setcounter{tocdepth}{1}
\tableofcontents

%
%
%
%

%
%
%
%
\section*{Introduction}

The search for canonical metrics on compact Kähler manifolds is a natural, classical and still highly active topic, with more recent developments allowing the space to have mild (log terminal) singularities, motivated by the Minimal Model Program and moduli spaces considerations. 

In their pioneering works~\cite{AP1,AP2,APS11}, Arezzo, Pacard and Singer provided a detailed analysis of the existence of extremal Kähler metrics in the sense of Calabi (and, in particular, constant scalar curvature K\"ahler metrics) in a small perturbation of the pullback of an extremal Kähler class under point blowups. The main purpose of this work is to extend such results, and other subsequent developments such as \cite{Sze12, Sze15, SSz20, DS21,Hal,Naj}, to a fairly broad class of resolutions of singularities of log terminal spaces and in the general weighted setting developed in~\cite{Lah,AJL}. In contrast to the above mentioned works, which relied on gluing techniques, our approach is based on pluripotential theory, building on the breakthrough work of Chen--Cheng~\cite{CC1,CC2,He} and its recent extension to the weighted setting~\cite{DJL24a,HL24}. 

\medskip

In order to state our main result, consider first a (nonsingular) compact K\"ahler manifold $X$. The notion of a \emph{weighted extremal} K\"ahler metric $\om$ on $X$, due to Lahdili~\cite{Lah} (see also~\cite{Ino}), involves
\begin{itemize}
\item a compact torus $T$ of the \emph{linear automorphism group}\footnote{Sometimes known as the reduced automorphism group in the literature, see \ref{sec:equiv}} $\Autr(X)\subset \Aut_0(X)$, preserving $\om$, and which can be assumed without loss to be maximal;
\item  a moment map $m\colon X\to\ft^\vee$ for $\om$;
\item two weights $v,w\in C^\infty(\ft^\vee)$ that are positive on the moment polytope $m(X)$, the case of usual extremal metrics corresponding to $(v,w)=(1,1)$. 
\end{itemize}
As in the usual cscK and extremal cases, weighted extremal metrics in the class $\{\om\}\in H^{1,1}(X,\R)$ of a given $T$-invariant K\"ahler metric $\om$ correspond to critical points of the \emph{relative weighted Mabuchi energy} $\mab^\rel_{v,w}$ in the space $\cH^T=\cH(X,\om)^T$ of $T$-invariant Kähler potentials, which parametrize $T$-invariant K\"ahler forms in $\{\om\}$. Here the relative weighted Mabuchi energy is obtained from the absolute (weighted) one by adding a term involving an \emph{extremal vector field}, ensuring the invariance of the relative energy under the action of the algebraic torus $T_\C$. This functional admits a maximal lsc extension to the space $\cET=\cE^1(X,\om)^T$ of $T$-invariant potentials of finite energy, \ie the completion of $\cH^T$ with respect to the Darvas metric $\dd_1$. Following the general approach of Darvas--Rubinstein~\cite{DR}, it was proved in~\cite{AJL} that the existence of a weighted extremal K\"ahler metric in $\{\om\}$ implies that $\mab^\rel_{v,w}\colon\cET\to\R\cup\{+\infty\}$ is \emph{coercive modulo $T_\C$}, \ie grows at least linearly with respect to the quotient  metric on $\cET/T_\C$. 

Conversely, it was proved in~\cite{CC1,CC2,He} in the unweighted case $(v,w)=(1,1)$ that coercivity mod $T_\C$ of $\mab^\rel_{v,w}$ implies the existence of an extremal K\"ahler metric; this was very recently extended to the weighted case in~\cite{DJL24a,HL24} under the assumption that $v$ is log-concave. Besides the unweighted case, this technical condition, which is expected to be unnecessary, holds for the weights corresponding to K\"ahler--Ricci solitons, and in many other differential geometrically relevant cases. 

As discussed in~\S\ref{sec:normal} of the present paper, the notions of a weighted extremal K\"ahler metric and of the associated relative weighted Mabuchi energy $\mab^\rel_{v,w}\colon\cET\to\R\cup\{+\infty\}$ still make sense in the more general case where $X$ is a compact K\"ahler space with log terminal singularities. We may now state our main result as follows. 

\begin{thmA}
Let $(X,\om_X)$ be a compact K\"ahler space with log terminal singularities, and suppose given: 
\begin{itemize}
\item  a compact torus $T\subset\Autr(X)$ preserving $\om_X$;
\item  two smooth positive weights  $v,w$ on the moment polytope with respect to a choice of moment map for $\om_X$; 
\item  a $T$-equivariant resolution of singularities $\pi\colon Y\to X$, assumed to be of Fano type;
\item a sequence of $T$-invariant K\"ahler forms $\om_j$ on $Y$ converging smoothly to $\pi^\star\om_X$ and such that $\om_j\ge(1-\e_j)\pi^\star\om_X$ with $\e_j\to 0$. 
\end{itemize}
If the relative weighted Mabuchi energy $\mab^\rel_{\om_X,v,w}$ is coercive modulo $T_\C$ on $\cE^1(X,\om_X)^T$, then so is $\mab^\rel_{\om_j,v,w}$ on $\cE^1(Y,\om_j)^T$ for all $j$ large enough, with uniform coercivity constants. 
\end{thmA}
Here a resolution of singularities $\pi\colon Y\to X$ is understood as any bimeromorphic morphism with $Y$ a compact K\"ahler manifold; we say that $\pi$ is of \emph{Fano type} if $Y$ admits a (finite) measure of the form $\hnu_Y=e^{-2f}\nu_Y$ with $f$ a quasi-psh function and $\nu_Y$ a smooth volume form, such that the Ricci current $\Ric(\hnu_Y)=\Ric(\nu_Y)+\ddc f$ is \emph{$\pi$-semipositive}, \ie $\Ric(\hnu_Y)+C\pi^\star\om_X\ge 0$ for $C\gg 1$. This holds for instance if there exists an effective $\Q$-divisor $B$ on $Y$ with $(Y,B)$ klt and $-(K_Y+B)$ $\pi$-ample, the converse being true when $X,Y$ are projective, in accordance with algebro-geometric terminology. In particular:
\begin{itemize}
\item any crepant resolution is of Fano type;
\item if $X$ is smooth, then the blowup $\pi\colon Y\to X$ along any submanifold is of Fano type;
\item more generally, any (projective) resolution of singularities with irreducible exceptional divisor is of Fano type. 
\end{itemize} 
If $X$ is a surface, then its minimal resolution of singularities is of Fano type. More generally, it is of course natural to ask whether any compact Kähler klt space admits a resolution of Fano type, but the answer appears to be unknown. 

Note that, when $X$ is smooth, applying Theorem~A to $Y=X$ in particular recovers \cite[Theorem 7]{LS94}.

As an immediate consequence of Theorem~A and the results of~\cite{AJL,DJL24a,HL24} mentioned above, we get: 

\begin{corB} Assume further that $X$ is smooth, $T$ is a maximal torus of $\Autr(X)$, and $v$ is log-concave. If the Kähler class $\{\om_X\}$ contains a weighted extremal metric, then so does $\{\om_j\}$ for all $j$ large enough. 
\end{corB} 
When $T$ is not maximal, the isometry group of any $T$-invariant weighted extremal metric always contains a maximal torus $T_{\max}$ of $\Autr(X)$ \cite[Corollary B.1]{Lah}, but our method nevertheless appears to require $\pi$ to be $T_{\max}$-equivariant (see Remark \ref{rmk:NotMaxTorus}), and we are thus not able to capture the most general setting in this respect. We emphasize that the problem is obstructed in the case of a non-maximal torus, see~\cite{DS21,Hal} for the most recent results in that direction. 

Modulo this restriction, Corollary~B applies to the blowup of an arbitrary submanifold, including the case of codimension 2 that had to be excluded in~\cite[Theorem~1]{SSz20}. The result can further be iterated to get weighted extremal metrics on any finite sequence of $T$-invariant smooth blowups of a weighted extremal Kähler manifold. 

When $\pi$ is the blowup of a set of distinct points and $T$ is maximal, Corollary~B is related to the results in \cite{AP1, AP2, APS11, Sze12, Hal}. More precisely, in~\cite[Theorem 1.1]{AP1} Arezzo--Pacard showed the existence of a cscK metric on the perturbed classes $\{\om_j\}$ for $j$ large enough starting from a cscK metric on $\{\om\}$, assuming $X$ has no nonzero holomorphic vector fields.
In~\cite[Theorem 1.3]{AP2}, the same authors extended their result on cscK metrics allowing holomorphic vector fields, provided one considers a suitable and sufficiently large collection of points.
Arezzo--Pacard--Singer in~\cite[Theorem 2.1]{APS11} presented some conditions involving automorphism groups to produce $T$-invariant extremal metrics in $\{\om_j\}$ for $j$ large enough starting from an extremal metric in $\{\om\}$.
Then Székelyhidi in~\cite[Theorem 1]{Sze12} reduced the assumptions in~\cite{APS11} to a stability condition bearing on the set of blownup points, which was later generalized to the weighted case in~\cite[Theorems~1.1, 1.2]{Hal}. 
When $\pi$ is the blow-up at a single point, \cite{Sze15, DS21} related the existence of extremal metrics on the blowup to relative K-stability. Finally, let us mention the very recent work~\cite[Theorem~3]{Sze24}, where the author independently obtains the existence of cscK metrics $\om_j$ assuming $\om_X$ to be a Kähler--Einstein metric on a klt space $X$ with no nontrivial holomorphic vector fields, and the upcoming work~\cite{PT24}, where the uniform coercivity estimate of Theorem~A is used to establish uniform \emph{a priori} estimates for the  weighted extremal metrics $\om_j$ in order to produce examples of klt spaces where coercivity of the relative weighted Mabuchi energy implies the existence of a singular weighted extremal metric in $\{\om_X\}$. 

\medskip

The proof of Theorem~A argues by contradiction, using a compactness argument in the \emph{strong topology} of finite energy spaces with respect to varying Kähler forms. A similar strategy, for finite energy spaces with varying prescribed singularities, was first utilized in the third author's PhD thesis~\cite{Tru22, Tru23}, and recently in~\cite{PTT23}, where the authors introduced a notion of cscK metrics for varieties with mild singularities that is extended to the weighted extremal case in the present paper. 

The strategy of proof can be expressed in terms of a general recipe for families of functionals on finite energy spaces, which can be adapted and transferred to other situations (see section \ref{sec:opencoer}). In the setting of Theorem~A, the main difficulty lies in the so-called `Ricci energy' part of the functional $\mab^\rel_{\om_j,w,v}$, which is computed with respect to the Ricci form of a fixed volume form $\nu_Y$ on $Y$ while the Kähler forms $\om_j$ degenerate in the limit. The Fano type assumption on $\pi$ allows to replace $\nu_Y$ with the singular measure $\hnu_Y$ and its Ricci current, whose one-sided control by $\om_X$ turns out to be enough to get the required semi-continuity property. 

\subsection*{Structure of the paper}
The article is organized as follows: 
\begin{itemize}
\item Section~\ref{sec:prel} recalls a number of facts on finite energy spaces, the action of automorphisms and the associated formalism of Futaki invariants for Euler--Lagrange functionals. 
\item Section~\ref{sec:open} studies the notion of strong convergence of sequences of finite energy potentials with respect to a varying Kähler class, and establishes a general uniform coercivity criterion for a sequence of functionals on varying finite energy spaces. 
\item Section~\ref{sec:weighted} reviews Lahdili's weighted Kähler formalism, recast in the language of equivariant forms and currents, and shows how to extend it to spaces with log terminal singularities. 
\item Section~\ref{sec:main} discusses the notion of resolution of Fano type, and then establishes the main results of the paper, \ie Theorem~A and Corollary~B. 
\end{itemize}

%
%
\begin{ackn}
The authors would like to thank Tamas Darvas, Ruadhaí Dervan, Eleonora di Nezza, Simon Jubert, Chung-Ming Pan, Lars Sektnan, Gabor Sz\'ekelyhidi, Tat Dat Tô and Dror Varolin for helpful discussions and feedback on the contents of this paper. The second author was partially supported by NSF grant DMS-2154380 and the Simons Foundation. The third author was supported by the Knut and Alice Wallenberg Foundation. 
\end{ackn}

%
%
%
%
\section{Preliminaries}\label{sec:prel}
Throughout the paper (except in~\S\ref{sec:normal}) $X$ is a compact K\"ahler manifold of dimension $n$, equipped with a reference K\"ahler form $\om_X$.  The notation $x\lesssim y$ means $x\le C_n  y$ for a constant $C_n>0$ only depending on $n$, and $x\approx y$ means $x\lesssim y$ and $y\lesssim x$. 

In this paper we use $\dc=i(\dbar -\partial)$, so that $\ddc=2i\ddbar$. The Laplacian of a smooth function $f$ with respect to a K\"ahler form $\om$ is then defined by 
$$
\D_\om f:=\tr_\om\ddc f=\frac{n\ddc f\wedge\om^{n-1}}{\om^n}.
$$

%
%
\subsection{Euler--Lagrange functionals and Futaki invariants}\label{sec:func} 
We collect here some well-known general remarks on Euler--Lagrange functionals. 

\subsubsection{} Consider a closed linear subspace $\cF$ of the Fr\'echet space $C^\infty(X)$, an open convex subset $\cU\subset\cF$ containing the origin, and a smooth operator $\mu\colon\f\mapsto\mu_\f$ on $\cU$, with values in the space of distributions on $X$. 

\begin{defi} An \emph{Euler--Lagrange functional} $F_\mu\colon \cU\to\R$ for $\mu$ is a smooth functional whose directional derivatives satisfy
$$
\langle F_\mu'(\f),f\rangle:=\frac{d}{dt}\bigg|_{t=0}F_\mu(\f+t f)=\int_X f\,\mu_\f
$$
for all $\f\in\cU$, $f\in\cF$.
\end{defi} 
Note that $F_\mu$ is unique up to an additive constant. Its existence is equivalent to the symmetry property
\begin{equation}\label{equ:symm}
\int_X f\,\mu_{\f;g}=\int_X g\,\mu_{\f;f},\quad\f\in\cU,\,f,g\in\cF
\end{equation}
of the directional derivatives 
$$
\mu_{\f;f}:=\frac{d}{dt}\bigg|_{t=0}\mu_{\f+t f}. 
$$
Indeed, \eqref{equ:symm} expresses that $\mu$, viewed as a $1$-form on $\cU$, is closed, and hence exact by convexity of $\cU$. Integrating along line segments, we get
\begin{equation}\label{equ:EL}
F_\mu(\f)-F_\mu(\p)=\int_0^1 dt\int_X(\f-\p)\mu_{t\f+(1-t)\p}
\end{equation}
for all $\f,\p\in \cU$.

Assume further that $\cU$ and $\mu$ are invariant under the action of $\R$ on $C^\infty(X)$ by translation. By~\eqref{equ:symm}, $\bar\mu:=\int_X\mu_\f$ is then independent of $\f$, and $F_\mu$ satisfies the translation equivariance property
\begin{equation}\label{equ:ELtrans}
F_\mu(\f+c)=F_\mu(\f)+\bar\mu c,\quad\f\in \cU,\,c\in\R. 
\end{equation}

\subsubsection{} Consider now a Lie subgroup $G\subset\Aut_0(X)$ with a smooth right-action $(\f,g)\mapsto\f^g$ on $\cU$ such that 
$$
\f^g-\p^g=g^\star(\f-\p)
$$ 
for all $\f,\p\in\cU$ and $g\in G$. Then $g\mapsto\tau_g:=0^g$ defines a (right) \emph{cocycle} $\tau\colon G\to\cU$, \ie a smooth map 
such that 
\begin{equation*}
\tau_{gh}=\tau_h+h^\star\tau_g,\quad g,h\in G,
\end{equation*}
which encodes the $G$-action since $\f^g=\tau_g+g^\star\f$. 

Assume further that $\mu$ above is $G$-equivariant, \ie $\mu_{\f^g}=g^\star\mu_\f$ for all $\f\in \cU$ and $g\in G$. Then $F_\mu$ is \emph{quasi-invariant}, in the sense that 
$$
F_\mu(\f^g)-F_\mu(\p^g)=F_\mu(\f)-F_\mu(\p)
$$
for all $\f,\p\in\cU$ and $g\in G$. As a consequence,
$$
\chi_\mu(g):=F_\mu(\f^g)-F_\mu(\f)=F_\mu(\tau_g)
$$ 
is independent of $\f\in\cU$, and defines a group character 
$$
\chi_\mu\colon G\to\R,
$$
which vanishes iff $F_\mu$ is $G$-invariant. This holds in particular if $F_\mu$ admits a one-sided bound. Finally, the differential of $\chi_\mu$ defines in turn a Lie algebra character 
$$
\Fut_\mu\colon\mathfrak{g}\to\R,
$$
the \emph{Futaki character} of $\mu$, which vanishes iff $F_\mu$ is $G$-invariant.  Concretely, differentiating $\chi_\mu(g)=F_\mu(\f^g)-F_\mu(\f)$ 
yields  
$$
\Fut_\mu(\xi)=\int_X(\dot\tau_\xi+\cL_\xi\f)\mu_\f
$$
for all $\xi\in\mathfrak{g}$ and $\f\in\cU$, where $\dot\tau_\xi:=\frac{d}{dt}\big|_{t=0}\tau_{\exp(t\xi)}$. In particular, if $F_\mu$ admits a critical point $\f\in\cU$, \ie $\mu_\f=0$, then $\Fut_\mu\equiv 0$, \ie $F_\mu$ is $G$-invariant. 

%
%
\subsection{The space $\cE^1$}\label{sec:E1}
In what follows, we fix a closed semipositive $(1,1)$-form $\om$ on $X$, assumed to be big, \ie 
$$
V=V_\om:=\int_X\om^n>0.
$$
\subsubsection{} For any $\f\in C^\infty(X)$ we set as usual
\begin{equation}
    \label{eqn:Notation}
    \om_\f:=\om+\ddc\f.
\end{equation}
As is well-known, a simple integration by parts shows that the \emph{Monge--Amp\`ere operator}\footnote{In this article, it is more convenient to drop the normalization to mass $1$ sometimes used in the literature.}
$$
C^\infty(X)\ni\f\mapsto\MA(\f)=\MA_\om(\f):=\om_\f^n
$$
satisfies the symmetry property~\eqref{equ:symm}. It thus admits an Euler--Lagrange functional 
$$
\en=\en_\om\colon C^\infty(X)\to\R, 
$$
normalized by $\en(0)=0$, and called the \emph{Monge--Amp\`ere energy} with respect to $\om$. Using~\eqref{equ:EL} we get
\begin{equation}\label{equ:endiff}
\en(\f)-\en(\p)=\frac{1}{n+1}\sum_{p=0}^n \int_X(\f-\p)\,\om_\f^p\wedge\om_\p^{n-p}
\end{equation}
for all $\f,\p\in C^\infty(X)$, and hence 
$$
\en(\f)=\frac{1}{n+1}\sum_{p=0}^n \int_X\f\,\om_\f^p\wedge\om^{n-p}. 
$$
\subsubsection{} Denote by $\PSH(\om)$ the space of all $\om$-psh functions, equipped with the $L^1$-topology. By Bedford--Taylor, the expression for $\en(\f)$ makes sense for any bounded $\f\in\PSH(\om)$. This defines a monotone increasing functional on such functions, which admits a unique extension 
$$
\en=\en_\om\colon\PSH(\om)\to\R\cup\{-\infty\}
$$ 
that is continuous along decreasing sequences, and still called Monge--Amp\`ere energy. 
Note that $\en$ is monotone increasing on $\PSH(\om)$, upper semicontinuous, and satisfies $\en(\f)=\lim_{k\to\infty}\en(\max\{\f,-k\})$ and the equivariance property
\begin{equation}
    \label{equ:equivariance}
    \en(\f+c)=\en(\f)+Vc
\end{equation}
for $\f\in\PSH(\om)$ and $c\in\R$. 

The space of \emph{$\om$-psh functions of finite energy} is defined as 
$$
\cE^1=\cE^1(\om):=\{\f\in\PSH(\om)\mid\en(\f)>-\infty\},
$$
and its \emph{strong topology} is the coarsest refinement of the weak topology (\ie the one induced by $\PSH(\om)$) in which $\en\colon\cE^1\to\R$ becomes continuous. 

One then shows (see \eg~\cite{BBGZ}) that the \emph{mixed Monge--Amp\`ere operator} 
$$
(\f_1,\dots,\f_n)\mapsto\om_{\f_1}\winter\om_{\f_n}, 
$$
defined on bounded $\om$-psh functions by Beford--Taylor, admits a unique strongly continuous extension to tuples in $\cE^1$. In particular, $\f\mapsto\MA(\f)=\om_\f^n$ is strongly continuous on $\cE^1$, with values in the space of positive measures of mass $V$. 

For all $\f_0,\dots,\f_n\in\cE^1$, $\f_0$ is further integrable with respect to $\om_{\f_1}\winter\om_{\f_n}$, and 
$$
(\f_0,\dots,\f_n)\mapsto\int_X\f_0\,\om_{\f_1}\winter\om_{\f_n}
$$ 
is continuous in the strong topology of $\cE^1$. Moreover, \eqref{equ:endiff} is still valid for $\f,\p\in\cE^1$, and 
\begin{equation}\label{equ:enconc}
\en(\f)-\en(\p)\le\int(\f-\p)\MA(\p).
\end{equation}
In particular, the translation invariant functionals
$$
\ii(\f,\p):=\int(\f-\p)(\MA(\p)-\MA(\f)),\quad\jj(\f):=\int\f\,\om^n-\en(\f)
$$ 
are both nonnegative.

\subsubsection{} The strong topology of $\cE^1$ is defined by the \emph{Darvas metric} $\dd_1$, given by
$$
\dd_1(\f,\p)=\en(\f)-\en(\p)
$$
whenever $\f,\p\in\cE^1$ satisfy $\f\ge\p$, and 
\begin{equation}\label{equ:Darvas}
\dd_1(\f,\p)=\dd_1(\f,\env(\f,\p))+\dd_1(\env(\f,\p),\p)=\en(\f)+\en(\p)-2\en(\env(\f,\p))
\end{equation}
in general. Here $\env(\f,\p)\in\cE^1$ denotes the \emph{rooftop envelope} of $\f,\p$, \ie the largest function in $\cE^1$ dominated by both $\f$ and $\p$. 

Using~\eqref{equ:Darvas}, one easily checks that
\begin{equation}\label{equ:enlip}
\left|\en(\f)-\en(\p)\right|\le \dd_1(\f,\p). 
\end{equation}
Note also that 
\begin{equation}\label{equ:denv}
\dd_1(\env(\f,\p),0)\lesssim\max\{\dd_1(\f,0),\dd_1(\p,0)\},
\end{equation}
as a consequence of~\eqref{equ:Darvas} and the triangle inequality. We further record the simple estimate
\begin{equation}\label{equ:d0}
\dd_1(\f,0)\le 2V|\sup \f|-\en(\f), 
\end{equation}
which follows from 
$$
\dd_1(\f,0)\le \dd_1(\f,\sup \f)+\dd_1(\sup \f,0)=V\sup \f-\en(\f)+V|\sup \f|. 
$$

Any two points in $\cE^1$ can be joined by a unique \emph{psh geodesic} $(\f_t)_{t\in [0,1]}$, which is in particular a constant speed geodesic for $\dd_1$, \ie $\dd_1(\f_s,\f_t)=c|t-s|$ for a constant $c\in\R_{\ge 0}$. Further, $t\mapsto\en(\f_t)$ is affine linear. 

By~\cite[Theorem~3.7]{DDL}, the Darvas metric satisfies the key estimate
\begin{equation}\label{equ:dI1}
\dd_1(\f,\p)\approx\ii_1(\f,\p):=\int|\f-\p|(\MA(\f)+\MA(\p))
\end{equation}
for $\f,\p\in\cE^1$. In particular, 
\begin{equation}\label{equ:Id}
\ii(\f,\p)\lesssim \dd_1(\f,\p).
\end{equation}
Further, 
\begin{equation}\label{equ:I1max}
\ii_1(\f,\p)=\ii_1(\f,\max\{\f,\p\})+\ii_1(\max\{\f,\p\},\p).
\end{equation}
Combined with~\eqref{equ:dI1} and the triangle inequality, this yields
\begin{equation}\label{equ:dmax}
\dd_1(\max\{\f,\p\},0)\lesssim\max\{\dd_1(\f,0),\dd_1(\p,0)\}.
\end{equation}

For $R>0$ denote by 
$$
\cE^1_R:=\{\f\in\cE^1\mid \dd_1(\f,0)\le R\}
$$
the $R$-ball in $\cE^1$. The following slightly more precise version of estimates in~\cite{BBGZ,DDL} play a key role in what follows. 

\begin{thm}\label{thm:MAcont} For all $\f,\p,\f_1,\dots,\f_n\in\cE^1_R$ we have the H\"older estimate
\begin{equation}\label{equ:MAcont}
\int|\f-\p|\,\om_{\f_1}\winter\om_{\f_n}\lesssim \dd_1(\f,\p)^\a R^{1-\a}
\end{equation}
with $\a:=2^{-n}$. In particular, 
\begin{equation}\label{equ:MAbd}
\int|\f|\,\om_{\f_1}\winter\om_{\f_n}\lesssim R. 
\end{equation}
\end{thm}

\begin{cor}\label{cor:MAhold} For all $\f,\p,\tau\in\cE^1_R$ we have
\begin{equation}\label{equ:MAhold}
\left|\int \tau(\MA(\f)-\MA(\p))\right|\lesssim \dd_1(\f,\p)^{1/2} R^{1/2}. 
\end{equation}
\end{cor}

\begin{lem}\label{lem:mean} For all $\f_1,\dots,\f_k\in\cE^1$ we have 
$\dd_1(\tfrac 1k\sum_{i=1}^k\f_i,0)\lesssim\max_i \dd_1(\f_i,0)$. 
\end{lem}
\begin{proof} Set $\f:=\frac 1k\sum_{i=1}^k\f_i$ and $\tau:=\env(\f_1,\dots,\f_k)$. Arguing as in the proof of~\cite[Lemma~3.8]{DDL}, we note that $\f\ge\tau$ and~\eqref{equ:enconc} yield
$$
\dd_1(\f,\tau)=\en(\f)-\en(\tau)\le\int(\f-\tau)\MA(\tau)=\frac 1k\sum_{i=1}^k\int(\f_i-\tau)\MA(\tau)
$$
$$
\le\max_i\int(\f_i-\tau)\MA(\tau)\le (n+1)\max_i (\en(\f_i)-\en(\tau))=(n+1)\max_i \dd_1(\f_i,\tau),
$$
using~\eqref{equ:endiff}. By~\eqref{equ:denv}, we further have $\dd_1(\tau,0)\lesssim \max_i \dd_1(\f_i,0)$, and the triangle inequality yields the result. 
\end{proof}

\begin{proof}[Proof of Theorem~\ref{thm:MAcont}] By~\cite[(1.37)]{synthetic} (which refines~\cite[Lemma~5.8]{BBGZ}), we have for all $\f,\p,\tau,\rho\in\cE^1_R$, 
\begin{equation}\label{equ:synth}
\left|\int (\f-\p)(\MA(\tau)-\MA(\rho))\right|\lesssim \ii(\f,\p)^\a \ii(\tau,\rho)^{1/2} R^{1/2-\a}. 
\end{equation}
In view of~\eqref{equ:Id}, this yields
$$
\left|\int(\f-\p)\MA(\tau)\right|\lesssim \left|\int(\f-\p)\MA(\f)\right|+\dd_1(\f,\p)^\a R^{1-\a}, 
$$
for $\f,\p,\tau\in\cE^1_R$. Since $|\int(\f-\p)\MA(\f)|\lesssim \dd_1(\f,\p)$ by~\eqref{equ:dI1}, we get
$$
\left|\int(\f-\p)\MA(\tau)\right|\lesssim \dd_1(\f,\p)^\a R^{1-\a}.
$$
Writing $|\f-\p|=2(\max\{\f,\p\}-\p)-(\f-\p)$ and using~\eqref{equ:dI1},~\eqref{equ:I1max},~\eqref{equ:dmax}, this implies
$$
\int|\f-\p|\MA(\tau)\lesssim \dd_1(\f,\p)^\a R^{1-\a}
$$
for $\f,\p,\tau\in\cE^1_R$, and~\eqref{equ:MAcont} follows in the special case when $\f_i=\tau$ for any $i=1,\dots,n$. Finally, pick $\f_1,\dots,\f_n\in\cE^1_R$, and set $\tau:=\frac1n\sum_{i=1}^n\f_i$. Then $\om_{\f_1}\winter\om_{\f_n}\lesssim \om_\tau^n=\MA(\tau)$. By Lemma~\ref{lem:mean}, we further have $\dd_1(\tau,0)\lesssim R$, and the previous estimate thus yields~\eqref{equ:MAcont} in the general case. 
\end{proof}

\begin{proof}[Proof of Corollary~\ref{cor:MAhold}] Replacing $\f,\p,\tau$ with $\max\{\f,-k\}$, $\max\{\p,-k\}$, $\max\{\tau,-k\}$, we may assume without loss that $\f,\p,\tau$ are bounded. By integration-by-parts we then have 
$$
\int\tau(\om_\f^n-\om_\p^n)=\sum_{p=0}^{n-1}\int\tau\ddc(\f-\p)\wedge\om_\f^p\wedge\om_\p^{n-1-p}
=-\sum_{p=0}^{n-1}\int \de\tau\wedge\dc (\f-\p)\wedge\om_\f^p\wedge\om_\p^{n-1-p}. 
$$
By Cauchy--Schwarz, 
$$
\left|\int \de\tau\wedge\dc (\f-\p)\wedge\om_\f^p\wedge\om_\p^{n-1-p}\right|\le\left(\int\de\tau\wedge\dc\tau\wedge\om_\f^p\wedge\om_\p^{n-1-p}\right)^{1/2}
$$
$$
\times\left(\int\de (\f-\p)\wedge\dc (\f-\p)\wedge\om_\f^p\wedge\om_\p^{n-1-p}\right)^{1/2}. 
$$
Now
$$
\int\de\tau\wedge\dc\tau\wedge\om_\f^p\wedge\om_\p^{n-1-p}=-\int\tau\ddc\tau\wedge\om_\f^p\wedge\om_\p^{n-1-p}
=\int\tau(\om-\om_\tau)\wedge\om_\f^p\wedge\om_\p^{n-1-p}, 
$$
and~\eqref{equ:MAbd} thus yields
$$
\int\de\tau\wedge\dc\tau\wedge\om_\f^p\wedge\om_\p^{n-1-p}\lesssim R.
$$
On the other hand, 
$$
\ii(\f,\p)=\int(\f-\p)(\om_\p^n-\om_\f^n)=\sum_{p=0}^{n-1}\int(\f-\p)\ddc(\p-\f)\om_\f^p\wedge\om_\p^{n-p-1}
$$
$$
=\sum_{p=0}^{n-1}\int \de(\f-\p)\wedge\dc(\f-\p)\wedge\om_\f^p\wedge\om_\p^{n-p-1}
$$
yields 
$$
\int\de (\f-\p)\wedge\dc (\f-\p)\wedge\om_\f^p\wedge\om_\p^{n-1-p}\le\ii(\f,\p)\lesssim \dd_1(\f,\p),
$$
by~\eqref{equ:Id}, and the result follows. 
\end{proof}

Finally, set 
\begin{equation}\label{equ:Tom}
T_\om:=\sup_{\f\in\PSH(\om)}\left\{\sup\f-V^{-1}\int\f\,\om^n\right\}\in [0,+\infty). 
\end{equation}

\begin{lem}\label{lem:supd} For all $\f\in\cE^1$ we have 
$$
|\sup\f|\lesssim V^{-1}\dd_1(\f,0)+T_\om.
$$
\end{lem}
\begin{proof} We have 
$$
|\sup\f|\le V^{-1}\int|\f|\om^n+T_\om\approx V^{-1} \dd_1(\f,0)+T_\om,
$$
by~\eqref{equ:dI1}. 
\end{proof}

\begin{lem}\label{lem:dJ} For $\f\in\cE^1$ such that $\en(\f)=0$, we have 
$$
\jj(\f)\lesssim \dd_1(\f,0)\lesssim \jj(\f)+V T_\om.
$$
\end{lem}
\begin{proof} Since $\en(\f)=0$ we have $\jj(\f)=\int\f\,\om^n=\int\f\,\MA(0)$, and the left-hand inequality follows from~\eqref{equ:dI1}. Note that $V\sup\f\ge\en(\f)=0$, and hence $\dd_1(\f,0)\le 2 V\sup\f$, by~\eqref{equ:d0}. The right-hand estimate now follows from $V\sup\f\le\int\f\,\om^n+V T_\om$. 
\end{proof}

%
%
%
\subsection{Action of automorphisms and coercivity}\label{sec:auto}
\subsubsection{} The neutral component $\Aut_0(X)$ of the group of biholomorphisms of $X$ acts trivially on cohomology. By the $\ddc$-lemma, for each $g\in\Aut_0(X)$ we can thus find a unique $\tau_g\in C^\infty(X)$ such that 
$$
g^\star\om=\om+\ddc\tau_g,\quad\en(\tau_g)=0. 
$$
\begin{lem} The map $\Aut_0(X)\to C^\infty(X)$ $g\mapsto\tau_g$ is a (right) cocycle, \ie 
\begin{equation}\label{equ:cocycle}
\tau_{gh}=\tau_h+h^\star\tau_g.
\end{equation}
for all $g,h\in\Aut_0(X)$. 
\end{lem}
\begin{proof} We have 
$$
\om+\ddc\tau_{gh}=(gh)^\star\om=h^\star g^\star\om=h^\star(\om+\ddc\tau_g)
$$
$$
=h^\star\om+\ddc h^\star\tau_g=\om+\ddc\tau_h+\ddc h^\star\tau_g,
$$
and hence $\tau_{gh}=\tau_h+h^\star\tau_g+c$ for some $c\in\R$. Furthermore, 
$$
0=(n+1)(\en(\tau_{gh})-\en(\tau_h))=\sum_{p=0}^n\int(\tau_{gh}-\tau_h)(\om+\ddc\tau_{gh})^p\wedge(\om+\ddc\tau_h)^{n-p}
$$
$$
=\sum_{p=0}^n\int(h^\star\tau_g+c)(gh)^\star\om^p\wedge h^\star\om^{n-p}=\sum_{p=0}^n\int h^\star(\tau_g g^\star\om^p\wedge\om^{n-p})+c\sum_{p=0}^n\int (gh)^\star\om^p\wedge h^\star\om^{n-p}
$$
$$
=\sum_{p=0}^n\int \tau_g g^\star\om^p\wedge\om^{n-p}+c(n+1)V=\sum_{p=0}^n\int \tau_g (\om+\ddc\tau_g)^p\wedge\om^{n-p}+(n+1)Vc
$$
$$
=(n+1)\en(\tau_g)+(n+1)V c=(n+1)Vc, 
$$
and hence $c=0$. 
\end{proof}

For each $\f\in\PSH(\om)$ and $g\in\Aut_0(X)$, we set 
$$
\f^g:=\tau_g+g^\star\f.
$$
Then 
\begin{equation}\label{equ:omphig}
\om+\ddc\f^g=g^\star(\om+\ddc\f),
\end{equation}
and hence $\f^g\in\PSH(\om)$, since $\f^g$ is quasi-psh. Furthermore, the cocycle relation~\eqref{equ:cocycle} formally implies that $(\f,g)\mapsto\f^g$ defines a right-action of $\Aut_0(X)$ on $\PSH(\om)$. 

Note also that the action map $\PSH(\om)\times\Aut_0(X)\to\PSH(\om)$ is continuous for the weak topology of $\PSH(\om)$, since the map $L^1(X)\times\Aut_0(X)\to L^1(X)$ $(\f,g)\mapsto g^\star\f$ is easily seen to be continuous. 

\begin{prop}\label{prop:isom} The right-action of $\Aut_0(X)$ on $\PSH(\om)$ restricts to an isometric action on $\cE^1$, which preserves the Monge--Amp\`ere energy. When $\om>0$ is K\"ahler, this action is further proper, in the sense that 
$$
\{g\in\Aut_0(X)\mid \dd_1(\f^g,\f)\le C\}
$$ 
is compact for any $\f\in\cE^1$ and $C>0$. 
\end{prop}
\begin{proof} For $g\in\Aut_0(X)$, we claim that 
\begin{equation}\label{equ:phig}
\en(\f^g)=\en(\f)
\end{equation} for all $\f\in\cE^1$, which will show that $\f^g\in\cE^1$, and also that the action map $\cE^1\times\Aut_0(X)\to\cE^1$ is strongly continuous. 

Assume first that $\f$ is bounded. Using~\eqref{equ:omphig} we get
$$
(n+1)\en(\f^g)=(n+1)(\en(\f^g)-\en(\tau_g))=\sum_{p=0}^n\int_X(\f^g-\tau_g)(\om+\ddc\f^g)^p\wedge(\om+\ddc\tau_g)^{n-p}
$$
$$
=\sum_{p=0}^n\int g^\star\f\, g^\star\om_\f^p\wedge g^\star\om^{n-p}=\sum_{p=0}^n\int \f\,\om_\f^p\wedge\om^{n-p}=(n+1)\en(\f),
$$
which proves the claim in that case. The general case following by writing $\f$ as the decreasing limit of the bounded $\om$-psh functions $\f_k=\max\{\f,-k\}$, and noting that $\f_k^g$ then decreases to $\f^g$. 

In view of~\eqref{equ:Darvas} and~\eqref{equ:phig}, to show that $g\in\Aut_0(X)$ defines an isometry of $\cE^1$ it is enough to show that $\env(\f^g,\p^g)=\env(\f,\p)^g$ for all $\f,\p\in\cE^1$, which is a direct consequence of the fact that $\f\mapsto \f^g=\tau_g+g^\star\f$, and hence also $\f\mapsto\f^{g^{-1}}$, are monotone increasing for any $g\in\Aut_0(X)$. 

Finally, assume $\om>0$. To prove properness, we may assume $\f=0$, since $\Aut_0(X)$ acts by isometries. Since $\{g\in\Aut_0(X)\mid \dd_1(\tau_g,0)\le C\}$ is closed in $\Aut(X)$, it is enough to show that any sequence $(g_i)$ in $\Aut_0(X)$ with $\dd_1(\tau_{g_i},0)\le C$ admits a convergent subsequence in $\Aut(X)$. This is a simple consequence of~\cite[Lemma~6.4]{DL}. Indeed, the latter implies that the Laplacian $\Delta_\om \tau_{g_i}$ is bounded. Equivalently, $g_i^\star\om=\om+\ddc\tau_{g_i}$ is bounded with respect to $\om$, which means that $g_i\colon X\to X$ is uniformly Lipschitz with respect to the Riemannian metric induced by $\om$. By Ascoli, after passing to a subsequence we may thus assume that $g_i$ converges uniformly to a map $g\colon X\to X$, automatically holomorphic. Since $\dd_1(\tau_{g_i},0)=\dd_1(0,\tau_{g_i^{-1}})$ is bounded, $(g_i^{-1})$ similarly admits a subsequence converging to a holomorphic map $h\colon X\to X$. Now $g_i g_i^{-1}=g_i^{-1} g_i=\id_X$ yields $g h=hg=\id_X$, and hence $g\in\Aut(X)$. The result follows. 
\end{proof}

\begin{rmk}\label{rmk:properness} If we drop the assumption that $\om$ is K\"ahler, properness in Proposition~\ref{prop:isom} appears to be unknown, but we can at least show that the action of any reductive subgroup $G$ of the linear automorphism group $\Autr(X)$ (see~\S\ref{sec:equiv}) is proper. Here is a sketch of the argument: 
\begin{itemize}
\item[(i)] A computation shows that
$p_2^\star\om+\ddc_{(g,x)}\tau_g=\rho^\star\om$ on $\Autr(X)\times X$, where $p_2(g,x):=x$ and $\rho(g,x):=g\cdot x$. 
\item[(ii)] Write $G=K_\C$ for a compact Lie group $K$. By averaging over $K$, we can reduce to the case where $\om$ is $K$-invariant. Then $g\mapsto \dd_1(\tau_g,0)$ descends to a function on $G/K$, which is convex by (i) and the convexity of $\dd_1$ along psh geodesics. If it fails to be proper, then it is identically $0$ along some geodesic ray $t\mapsto e^{t\xi}$ with $\xi\in i\Lie K\setminus\{0\}$. This implies $i(\xi)\om=0$, and the holomorphic vector field $\xi$ is thus $0$, since $\om>0$ on a non-empty open set, a contradiction. 
\end{itemize}

\end{rmk}

\subsubsection{} Given a closed (Lie) subgroup $G\subset\Aut_0(X)$, we define $\jj_G\colon\cE^1\to\R_+$ by setting 
$$
\jj_G(\f):=\inf_{g\in G}\jj(\f^g).
$$
If $\en(\f)=0$, then $\en(\f^g)=0$ for all $g$, and Lemma~\ref{lem:dJ} yields 
\begin{equation}\label{equ:JG}
\jj_G(\f)\lesssim \dd_{1,G}(\f,0)\lesssim \jj_G(\f)+VT_\om
\end{equation}
where 
\begin{equation}
    \label{eqn:d_G}
    \dd_{1,G}(\f,\p):=\inf_{g\in G} \dd_1(\f^g,\p), 
\end{equation}
induces the quotient pseudometric on $\cE^1/G$ (which is 
a metric when the action is proper). 

Consider now a functional $M\colon\cF\to\R\cup\{+\infty\}$ defined on a subset $\cF\subset\cE^1$, such that both $\cF$ and $M$ are translation invariant, and assume also that $\cF$ is $G$-invariant. 

\begin{defi}\label{defi:coerG} We say that $M$ is \emph{coercive modulo $G$} if there exists $\d,C>0$ such that $M\ge\d\jj_G-C$ on $\cF$. 
\end{defi}
By~\eqref{equ:JG} and translation invariance, $M$ is coercive modulo $G$ iff $M(\f)\ge\d' \dd_{1,G}(\f,0)-C'$ on $\cF^0:=\{\f\in\cF\mid\en(\f)=0\}$ for some $\d',C'>0$. 

In practice for us, $M$ will be quasi-invariant (see~\S\ref{sec:func}), and coercivity modulo $G$ then implies that $M$ is $G$-invariant, since it is in particular bounded below (see~\S\ref{sec:func}).

\begin{lem}\label{lem:propersub} Assume that $M$ is $G$-invariant, and that $G$ acts properly on $\cF$. For any closed subgroup $H\subset G$, the following are equivalent:
\begin{itemize}
\item[(i)] $M$ is coercive modulo $H$;
\item[(ii)] $M$ is coercive modulo $G$ and $G/H$ is compact.
\end{itemize}
\end{lem}
\begin{proof} Assume (i), \ie $M\ge\d d_H(\cdot,0)-C$ for some $\d,C>0$. Since $d_H\ge \dd_{1,G}$, $M$ is coercive modulo $G$. For any $g\in G$, we also have $M(0)=M(0^g)\ge\d d_H(0^g,0)-C$, and hence $d_H(0^g,0)\le C'$. Since $G$ acts properly on $\cF$, this proves that $G/H$ is compact, and hence (i)$\Rightarrow$(ii). 

Conversely assume (ii). Each element of $G$ can be written as $g=kh$ with $h\in H$ and $k$ in a compact subset $K\subset G$. For each $\f\in\cF$ we thus have
$$
d_H(\f,0)\le \dd_1(\f^h,0)=\dd_1(\f^g,0^k)\le \dd_1(\f^g,0)+\dd_1(0^k,0). 
$$
This shows 
$$
\dd_{1,G}(\f,0)\le d_H(\f,0)\le \dd_{1,G}(\f,0)+C
$$
for a uniform constant $C>0$, and the result follows.     
\end{proof}

For later use we also introduce: 
\begin{defi}\label{defi:min} We say that a family $(\f_i)$ in $\cE^1$ is \emph{$G$-minimal} if it satisfies $\dd_{1,G}(\f_i,\f_j)=\dd_1(\f_i,\f_j)$ for all $i,j$. 
\end{defi} 
\begin{lem}\label{lem:Gmin} A metric geodesic $(\f_t)_{t\in [0,1]}$ is $G$-minimal iff $\{\f_0,\f_1\}$ is $G$-minimal. 
\end{lem}
Note that $G$-minimal geodesics (which are called \emph{$G$-calibrated} in~\cite{DL}) induce geodesics in $\cE^1/G$.

\begin{proof} We reproduce the simple argument of~\cite[Lemma~3.7]{DR}: for $0\le s,t\le 1$ we have 
$$
\dd_{1,G}(\f_0,\f_1)\le \dd_{1,G}(\f_0,\f_s)+\dd_{1,G}(\f_s,\f_t)+\dd_{1,G}(\f_t,\f_1)
$$
$$
\le \dd_1(\f_0,\f_s)+\dd_1(\f_s,\f_t)+\dd_1(\f_t,\f_1)=\dd_1(\f_0,\f_1)=\dd_{1,G}(\f_0,\f_1),
$$
which forces equality everywhere, and hence yields $\dd_{1,G}(\f_s,\f_t)=\dd_1(\f_s,\f_t)$. 
\end{proof}

%
%
\section{Strong convergence and openness of coercivity}\label{sec:open}
%
%

Letting $\om_j$ converging smoothly to $\om$ as described below, in this section we introduce the notion of strong convergence for functions in $\cE^1(\om_j)$ as $j\to +\infty$, to then establish a general result on openness of coercivity of functionals on $\cE^1(\om_j)$ (Theorem \ref{thm:coeropen}).

\subsection{Setup and notation}
In this section, we assume given a sequence of semipositive, big $(1,1)$-forms $\om_j$ on $X$, converging smoothly to a big semipositive $(1,1)$-form $\om$, and such that 
\begin{equation}\label{equ:ommin}
\om_j\ge (1-\e_j)\om\quad\text{with}\quad\e_j\to 0.
\end{equation}
Set $V_j:=\int_X\om_j^n$, which converges to $V:=\int_X\om^n$. We denote by 
$$
\en_j:=\en_{\om_j}\colon\PSH(\om_j)\to\R\cup\{-\infty\},\quad
\en:=\en_{\om}\colon\PSH(\om)\to\R\cup\{-\infty\}
$$
the corresponding Monge--Amp\`ere energy functionals. We set
$$
\cE^1_j:=\cE^1(\om_j),\quad\cE^1:=\cE^1(\om), 
$$
and denote by $\dd_{1,j}$ and $\dd_1$ their Darvas metrics, and by 
$$
\cE^1_j\ni\f\mapsto\MA_j(\f)=\om_{j,\f}^n,\quad\cE^1\ni\f\mapsto\MA(\f)=\om_\f^n
$$ 
their Monge--Amp\`ere operators where, similarly to~\eqref{eqn:Notation}, we set $\om_{j,\f}:=\om_j+\ddc \f$. Finally, we fix a volume form $\nu$ on $X$, and denote by $\Ent=\Ent(\cdot|\nu)$ the associated relative entropy function (see~Appendix~\ref{sec:ent}). We then define the entropy functionals
$$
\ent_j\colon\cE^1_j\to\R\cup\{+\infty\},\quad\ent\colon\cE^1\to\R\cup\{+\infty\}
$$
by setting
$$
\ent_j(\f):=\tfrac 12\Ent(\MA_j(\f)|\nu),\quad\ent(\f):=\tfrac 12\Ent(\MA(\f)|\nu).
$$
Note that $\ent_j$ and $\ent$ are lsc, by lower semincontinuity of the relative entropy in the weak topology of measures. 
%
%
\subsection{Weak convergence}

\begin{defi} We say that a sequence $(\f_j)$ with $\f_j\in\PSH(\om_j)$ for all $j$ \emph{converges weakly} to $\f\in\PSH(\om)$ if $\f_j\to\f$ in $L^1(X)$. 
\end{defi}
Note that, as $X$ is compact, the $L^1$-convergence does not depend on the choice of the smooth positive volume form.
Since $\om_j\le C\om_X$ for a uniform $C>0$, this equivalently means $\f_j\to\f$ in $\PSH(C\om_X)$. By standard properties of psh functions, it follows that: 

\begin{lem}\label{lem:weak} For any weakly convergent sequence $\PSH(\om_j)\ni\f_j\to\f\in\PSH(\om)$ we have 
$\sup_X \f_j\to\sup_X \f$ and $(\sup_{k\ge j}\f_k)^\star\searrow\f$ pointwise as $j\to\infty$. 
\end{lem}
Here the upper $\star$ denotes, as usual, usc regularization. We also note: 
\begin{lem}\label{lem:weakcomp} There exists a uniform constant $A>0$ such that the submean value inequality
$$
\sup_X\f\le V_j^{-1}\int_X\f\,\om_j^n+A
$$
holds for all $j$ and all $\f\in\PSH(\om_j)$. In particular, each sequence $\f_j\in\PSH(\om_j)$ with either $\sup_X\f_j$ or $\int_X\f_j\,\om_j^n$ bounded admits a weakly convergent subsequence.
\end{lem}
\begin{proof} Pick $C>0$ such that $\om_j\le C\om_X$ for all $j$. Since $\f\mapsto\int\f\,\om^n$ is continuous on $\PSH(C\om_X)$, it is bounded on the compact subset defined by $\sup\f=0$. Since $\om_j>0$ and $\om_j\ge (1-\e_j)\om$, we also have $\om_j\ge c\om$ for all $j$ and a uniform constant $c>0$. Further using that $V_j$ is bounded away from $0$, we get a uniform constant $A>0$ such that $V_j^{-1}\left|\int\f\,\om_j^n\right|\le A$ for all $\f\in\PSH(\om_j)\subset\PSH(C\om_X)$ such that $\sup_X\f=0$. This implies the first point, and the second one follows by weak compactness of $\{\f\in\PSH(C\om_X)\mid\sup_X\f=0\}$.  
\end{proof}

As we next show, the Monge--Amp\`ere energy is usc with respect to weak convergence. 

\begin{prop}\label{prop:weakusc} For any weakly convergent sequence $\PSH(\om_j)\ni\f_j\to\f\in\PSH(\om)$, we have $\limsup_j\en_j(\f_j)\le\en(\f)$.
\end{prop}

\begin{lem}\label{lem:weakusc} If $\f\in\PSH(\om)$ is bounded, then $(1-\e_j)\f\in\cE^1_j$ for all $j$, and $\en_j((1-\e_j)\f)\to\en(\f)$. 
\end{lem}
\begin{proof} Since $\om_j\ge (1-\e_j)\om$, $(1-\e_j)\f$ is $\om_j$-psh and bounded, and hence lies in $\cE^1_j$. We  further have 
$$
\en_j((1-\e_j)\f)=\frac{1}{n+1}\sum_{p=0}^n\int_X (1-\e_j)\f\,(\om_j+(1-\e_j) \ddc\f)^p\wedge\om_j^{n-p}
$$
in the sense of Bedford--Taylor. Since $\e_j\to 0$ and $\om_j\to\om$ smoothly, this converges to 
$$
\en(\f)=\frac{1}{n+1}\sum_{p=0}^n\int_X\f\,(\om+\ddc\f)^p\wedge\om^{n-p},
$$
completing the proof.
\end{proof}

\begin{proof}[Proof of Proposition~\ref{prop:weakusc}] Since $\sup \f_j\to\sup \f$ (see Lemma~\ref{lem:weak}), the $\om_j$-psh functions $\tf_j:=\f_j-\sup \f_j\le 0$ 
converge weakly to $\tf:=\f-\sup \f\in\PSH(\om)$. By translation equivariance, we further have 
$$
\en_j(\f_j)-\en_j(\tf_j)=V_j\sup \f_j\to V\sup\f=\en(\f)-\en(\tf). 
$$
Replacing $\f_j$ and $\f$ with $\tf_j$ and $\tf$, we may thus assume without loss that $\f_j\le 0$ for all $j$. 

Assume for the moment that $\f$ and all $\f_j$ are further bounded. Then~\eqref{equ:enconc} yields 
$$
\en_j(\f_j)-\en_j((1-\e_j)\f)\le \int\left(\f_j-(1-\e_j)\f\right)\left(\om_j+(1-\e_j)\ddc\f\right)^n, 
$$
where $\en_j((1-\e_j)\f)\to\en(\f)$ by Lemma~\ref{lem:weakusc}. Thus
\begin{equation}\label{equ:enTusc}
\limsup_j\en_j(\f_j)\le\en(\f)+\limsup_j \int\left(\f_j-(1-\e_j)\f\right)\left(\om_j+(1-\e_j)\ddc\f\right)^n. 
\end{equation}
On the one hand,  
$$
\int_X\f_j (\om_j+(1-\e_j)\ddc\f)^n=\int_X\f_j \left(\om_j-(1-\e_j)\om+(1-\e_j)(\om+\ddc\f)\right)^n
$$
$$
\le (1-\e_j)^n\int\f_j\,(\om+\ddc\f)^n,
$$
since $\f_j\le 0$ and $\om_j-(1-\e_j)\om\ge 0$, and
$$
\limsup_j\int\f_j\,(\om+\ddc\f)^n\le\int\f\,(\om+\ddc\f)^n
$$
using $\f_j\le(\sup_{k\ge m}\f_k)^\star\searrow\f$ (see Lemma~\ref{lem:weak}) and monotone convergence. On the other hand, we have 
$$
\lim_j\int\f\,(\om_j+(1-\e_j)\ddc\f)^n=\int\f\,(\om+\ddc\f)^n
$$ 
as in Lemma~\ref{lem:weakusc}. We conclude
\begin{multline*}
\limsup_j\int_X\left(\f_j-(1-\e_j)\f\right) \left(\om_j+(1-\e_j)\ddc\f\right)^n\\
\le\limsup_j (1-\e_j)^n\int_X\f_j\,\left(\om+\ddc\f\right)^n
-\lim_j (1-\e_j)\int\f\,(\om_j+(1-\e_j)\ddc\f)^n\\
\le\int\f\,(\om+\ddc\f)^n-\int\f\,(\om+\ddc\f)^n=0, 
\end{multline*}
and the result now follows from~\eqref{equ:enTusc}.

In the general case, $\f^k_j=\max\{\f_j,-k\}$ converges weakly to $\f^k=\max\{\f,-k\}$ as $j\to\infty$, for each given $k$. Since $\en_j(\f_j)\le\en_j(\f_j^k)$, the previous step yields
$$
\limsup_j\en_j(\f_j)\le\limsup_j\en_j(\f^k_j)\le\en(\f^k),
$$
and the result follows since $\lim_k\en(\f^k)=\en(\f)$. 
\end{proof}

%
%
\subsection{Strong convergence}\label{sec:strong}
Following~\cite{PTT23} we introduce: 

\begin{defi} We say that a sequence $\f_j\in\cE^1_j$  \emph{converges strongly} to $\f\in\cE^1$ if $\f_j\to\f$ weakly and $\en_j(\f_j)\to\en(\f)$.
\end{defi}

When $\om_j$ is independent of $j$, this recovers the usual characterization of strong convergence in $\cE^1$ with respect to the Darvas metric $\dd_1$, cf~\S\ref{sec:E1}.  By Lemma~\ref{lem:weakusc}, we have: 

\begin{exam}\label{exam:strong} For any bounded function $\f\in\cE^1$, $(1-\e_j)\f\in\cE^1_j$ converges strongly to $\f$.
\end{exam}

As a consequence of Bedford--Taylor, we also have:

\begin{exam}\label{exam:decrstrong} If $\f\in\cE^1$ is bounded and $\f_j\in\cE^1_j$ decreases to $\f$, then $\f_j\to\f$ strongly. 
\end{exam} 

As we next show, the Darvas metrics are continuous with respect to strong convergence. 

\begin{prop}\label{prop:dcont} If $\f_j,\p_j\in\cE^1_j$ converge strongly to $\f,\p\in\cE^1$, then $\dd_{1,j}(\f_j,\p_j)\to \dd_1(\f,\p)$.
\end{prop}
We recall that $\dd_{1,j}$ denotes the Darvas metric on $\cE^1_j$ (see subsection~\ref{sec:E1}).
\begin{proof} By~\eqref{equ:d0}, $\dd_{1,j}(\f_j,0)\le 2|\sup\f_j|-\en_j(\f_j)$ remains bounded. The same holds for $\p_j$, and $\dd_{1,j}(\f_j,\p_j)\le \dd_{1,j}(\f_j,0)+\dd_{1,j}(\p_j,0)$ is thus bounded as well. Fix $k$ and set $\f^k:=\max\{\f,-k\}$, so that $(1-\e_j)\f^k\in\cE^1_j$ converges strongly to $\f^k$ (see Example~\ref{exam:strong}). We claim that 
\begin{equation}\label{equ:limsupd}
\limsup_j \dd_{1,j}(\f_j,(1-\e_j)\f^k)\le \dd_1(\f,\f^k).
\end{equation} 
To see this, set
$$
\tf_j^k:=\max\{\f_j,(1-\e_j)\f^k\}\in\cE^1_j.
$$
Since $\f_j\to\f$ weakly and $\e_j\to 0$, we have 
$$
\tf_j^k\to\max\{\f,\f^k\}=\f^k
$$
weakly as $j\to\infty$, and hence 
\begin{equation}\label{equ:lmsupjk}
\limsup_j\en_j(\tf_j^k)\le\en(\f^k), 
\end{equation}
by Proposition~\ref{prop:weakusc}. As $\tf_j^k\ge\f_j$ and $\tf^k_j\ge (1-\e_j)\f^k$, we have 
$$
\dd_{1,j}(\f_j,\tf_j^k)=\en_j(\tf_j^k)-\en_j(\f_j),\quad \dd_{1,j}(\tf_j^k,(1-\e_j)\f^k)=\en_j(\tf^k_j)-\en_j((1-\e_j)\f^k), 
$$
and hence 
$$
\dd_{1,j}(\f_j,(1-\e_j)\f^k)\le \dd_{1,j}(\f_j,\tf_j^k)+\dd_{1,j}(\tf_j^k,(1-\e_j)\f^k)=2\en_j(\tf_j^k)-\en_j(\f_j)-\en_j((1-\e_j)\f^k). 
$$
Now $\en_j(\f_j)\to\en(\f)$ by assumption, and $\en_j((1-\e_j)\f^k)\to\en(\f^k)$ by Lemma~\ref{lem:weakusc}. Thus~\eqref{equ:lmsupjk} yields
$$
\limsup_j \dd_{1,j}(\f_j,(1-\e_j)\f^k)\le\en(\f^k)-\en(\f)=\dd_1(\f,\f^k),
$$
which proves~\eqref{equ:limsupd}. Next, set 
$$
\p^k:=\max\{\p,-k\},\quad\tau^k=\env(\f^k,\p^k).
$$
For all $j$, we then have 
$$
\dd_{1,j}((1-\e_j)\f^k,(1-\e_j)\p^k)\le \dd_{1,j}((1-\e_j)\f^k,(1-\e_j)\tau^k)+\dd_{1,j}((1-\e_j)\tau^k,(1-\e_j)\p^k)
$$
$$
=\en_j((1-\e_j)\f^k)+\en_j((1-\e_j)\p^k)-2\en_j((1-\e_j)\tau^k),
$$
and hence 
$$
\limsup_j \dd_{1,j}((1-\e_j)\f^k,(1-\e_j)\p^k)\le \en(\f^k)+\en(\p^k)-2\en(\tau^k)=\dd_1(\f^k,\p^k), 
$$
by Lemma~\ref{lem:weakusc}. Injecting this into
$$
\dd_{1,j}(\f_j,\p_j)\le \dd_{1,j}(\f_j,(1-\e_j)\f^k)+\dd_{1,j}((1-\e_j)\f^k,(1-\e_j)\p^k)+\dd_{1,j}(\p_j,(1-\e_j)\p^k)
$$
and using~\eqref{equ:limsupd} together with its analogue for $\p_j$, we infer 
$$
\limsup_j \dd_{1,j}(\f_j,\p_j)\le \dd_1(\f^k,\p^k)+\dd_1(\f,\f^k)+\dd_1(\p,\p^k),
$$
and hence $\limsup_j \dd_{1,j}(\f_j,\p_j)\le \dd_1(\f,\p)$ since $\lim_k \dd_1(\f^k,\p^k)=\dd_1(\f,\p)$ while $\lim_k \dd_1(\f,\f^k)=\lim_k \dd_1(\p,\p^k)=0$. 

In order to conversely prove $\liminf_j \dd_{1,j}(\f_j,\p_j)\ge \dd_1(\f,\p)$, we may assume, after passing to a subsequence, that $\dd_{1,j}(\f_j,\p_j)$ converges. For each $j$ set $\tau_j:=\env(\f_j,\p_j)$. As noted above, $\dd_{1,j}(\f_j,0)$ and $\dd_{1,j}(\p_j,0)$ are bounded, and $\dd_1(\tau_j,0)$ is bounded as well, by~\eqref{equ:denv}. Thus $\int\tau_j\om_j^n$ is bounded, and we may thus assume, after passing to a subsequence, that $\tau_j$ converges weakly to $\tau\in\PSH(\om)$ (see Lemma~\ref{lem:weakcomp}). As $\en_j(\tau_j)$ is bounded, Proposition~\ref{prop:weakusc} further yields $\en(\tau)\ge\limsup_j\en_j(\tau_j)>-\infty$, and hence $\tau\in\cE^1$. Since $\tau_j\le\f_j,\p_j$, we have $\tau\le\f,\p$, and hence
$$
\dd_1(\f,\p)\le \dd_1(\f,\tau)+\dd_1(\tau,\p)=\en(\f)+\en(\p)-2\en(\tau), 
$$
where $\en(\f)=\lim_j \en_j(\f_j)$, $\en(\p)=\lim_j\en_j(\p_j)$ and $\en(\tau)\ge\limsup\en(\tau_j)$. Thus
$$
\dd_1(\f,\p)\le\liminf_j(\en_j(\f_j)+\en_j(\p_j)-2\en_j(\tau_j))=\liminf_j \dd_{1,j}(\f_j,\p_j),
$$
which concludes the proof. 
\end{proof}

\subsection{Strong compactness}
Adapting the proof of~\cite[Theorem~2.17]{BBEGZ}, we next establish the following strong compactness result. 

\begin{thm}\label{thm:entcomp} Any sequence $\f_j\in\cE^1_j$ such that $\sup_X\f_j$ and the entropy $\ent_j(\f_j)$ are both bounded admits a subsequence that converges strongly to some $\f\in\cE^1$. 
\end{thm}

\begin{lem}\label{lem:entcomp} There exists a uniform constant $C>0$ such that 
$$
\en_j(\f)\ge V_j\sup_X\f-C(\ent_j(\f)+1)
$$
for all $\f\in\cE^1_j$. 
\end{lem}
\begin{proof} Pick $C>0$ such that $\om_j\le C\om_X$ for all $j$. By~\cite{Zer}, there exists $\a,B>0$ such that $\int_X e^{-\a\f}\,\nu\le B$ for all $\f\in\PSH(C\om_X)$ such that $\sup_X \f=0$. For any $\f\in\cE^1_j\subset\PSH(C\om_X)$ with $\sup_X\f=0$, \eqref{equ:enconc} and~\eqref{equ:entleg} thus yield
\begin{align*}
    -\a\en_j(\f)\le\int(-\a\f)\MA_j(\f)&\le2\ent_j(\f)+V_j\log\int e^{-\a\f}\,\nu-V_j \log V_j\\
    &\le2\ent_j(\f)+C
\end{align*}
for a uniform constant $C>0$, since $V_j$ is bounded above. The result follows, by translation equivariance of $\en_j$. 
\end{proof}

\begin{lem}\label{lem:entcomp2} Assume $\cE^1_j\ni\f_j\to\f\in\cE^1$ weakly. If $\ent_j(\f_j)$ is bounded, then $\f_j\to\f$ strongly. 
\end{lem}

\begin{proof} By Proposition~\ref{prop:weakusc}, it is enough to show $\liminf_j\en_j(\f_j)\ge\en(\f)$. Set $\mu_j:=\MA_j(\f_j)$ and $\f^k=\max\{\f,-k\}$. By~\eqref{equ:enconc}, we have 
\begin{align*}
\en_j(\f_j)-\en_j((1-\e_j)\f^k)
&\ge\int(\f_j-(1-\e_j)\f^k)d\mu_j\\
&=\int(\f_j-\f)d\mu_j+\int(\f-\f^k)d\mu_j+(1-(1-\e_j))\int\f^k\,d\mu_j. 
\end{align*}
Since $\{\f_j\mid j\in\N\}\cup\{\f\}$ is a (weakly) compact subset of $\PSH(X,C\om_X)$ all of whose elements have zero Lelong numbers (by~\cite[Theorem~1.1]{DDL2}), \cite{Zer} shows that
$$
\sup_j\int\exp\left(p|\f_j|\right)\nu<+\infty
$$
for each $p<\infty$. As $\Ent(\mu_j|\nu)=2\ent_j(\f_j)$ is bounded, Lemma~\ref{lem:ent} thus yields $\lim_j\int(\f_j-\f)d\mu_j=0$. Since we further have $\lim_j\en_j((1-\e_j)\f^k)=\en(\f^k)$ and $\lim_j(1-(1-\e_j))\int\f^k\,d\mu_j=0$ (as $\e_j\to 0$ while $|\int\f^k\,\mu_j|\le\sup |\f^k|$), the above inequality yields 
$$
\liminf_j\en_j(\f_j)\ge \en(\f^k)-\sup_j\int|\f-\f^k|\,d\mu_j. 
$$
Since $\f\le\f^k\le\f_0$ and $\f$ has zero Lelong numbers, we have as above $\sup_k\int\exp(p|\f^k|)\nu<+\infty$. Further,  $\Ent(\mu_j|\nu)$ is bounded, so Lemma~\ref{lem:ent} yields $\sup_j\int|\f-\f^k|\,d\mu_j\to 0$ as $k\to+\infty$. Using $\lim_k\en(\f^k)=\en(\f)$, we conclude, as desired, that $\liminf_j\en_j(\f_j)\ge\en(\f)$. 
\end{proof}

\begin{proof}[Proof of Theorem~\ref{thm:entcomp}] By Lemma~\ref{lem:weakcomp}, we may assume, after passing to a subsequence, that $(\f_j)$ converges weakly to $\f\in\PSH(\om)$. By Lemma~\ref{lem:entcomp}, $\en_j(\f_j)$ is further bounded below, and Proposition~\ref{prop:weakusc} thus yields $\en(\f)>-\infty$, \ie $\f\in\cE^1$. Finally,  Lemma~\ref{lem:entcomp2} shows that $\f_j\to\f$ strongly, and we are done. 
\end{proof}

%
%
\subsection{Strong and weak (semi)continuity criteria}\label{sec:criteria} 
In what follows, we assume given, for each $j$, a functional
$$
F_j\colon\cF_j\to\R\cup\{-\infty\}
$$
defined on a subset $\cF_j\subset\cE^1_j$. We also assume given subsets $\cR\subset\cF\subset\cE^1$, and a functional 
$$
F\colon\cR\to\R, 
$$
such that 
\begin{itemize}
\item[(R)] $\cR$ consists of bounded functions, and each $\f\in\cF$ can be written as the limit of some decreasing sequence $(\f^k)$ in $\cR$; 
\item[(F)] for each $\f\in\cR$, $(1-\e_j)\f\in\cE^1_j$ lies in $\cF_j$, and $F_j((1-\e_j)\f)\to F(\f)$. 
\end{itemize}

We first establish a general strong continuity criterion. 

\begin{lem}\label{lem:strongcont} In the setup above, assume that each $F_j$ is finite valued, and satisfies the uniform H\"older estimate 
\begin{equation}\label{equ:holdFj}
|F_j(\f)-F_j(\p)|\le C \dd_{1,j}(\f,\p)^\a\max\{\dd_{1,j}(\f,0),\dd_{1,j}(\p,0)\}^{1-\a}
\end{equation}
for all $\f,\p\in\cF_j$, where $\a\in (0,1]$ and $C>0$ are uniform constants. Then $F$ admits a unique continuous extension $F\colon\cF\to\R$, which satisfies the analogue of~\eqref{equ:holdFj}. For each strongly convergent sequence $\cF_j\ni\f_j\to\f\in\cF$, we further have $F_j(\f_j)\to F(\f)$. 
\end{lem}

\begin{proof} Assume first $\f,\p\in\cR$. By condition (F) above, we have $F_j((1-\e_j)\f)\to F(\f)$ and $F_j((1-\e_j)\p)\to F(\p)$. Since $\f,\p$ are bounded, $(1-\e_j)\f\in\cF_j$ converges strongly to $\f$, and similarly for $\p$ (see Example~\ref{exam:strong}). By Proposition~\ref{prop:dcont}, we thus have 
$$
\dd_{1,j}((1-\e_j)\f,(1-\e_j)\p)\to \dd_1(\f,\p),\quad \dd_{1,j}((1-\e_j)\f,0)\to \dd_1(\f,0),\quad \dd_1((1-\e_j)\p,0)\to \dd_1(\p,0),
$$
and~\eqref{equ:holdFj} thus yields
$$
|F(\f)-F(\p)|\le C \dd_1(\f,\p)^\a\max\{\dd_1(\f,0),\dd_1(\p,0)^{1-\a}. 
$$
This shows that $F$ is uniformly continuous on $\cR$, which is dense in $\cF$ by (R), so that $F$ admits a unique continuous extension to $\cF$, satisfying the same H\"older estimate. 

Next pick any $\f\in\cF$, and write it as the limit of a decreasing sequence $(\f^k)$ in $\cR$. Fix $k\in\N$, and write  
\begin{equation}\label{equ:3terms}
|F_j(\f_j)-F(\f)|\le |F_j(\f_j)-F_j((1-\e_j)\f^k)|+|F_j((1-\e_j)\f^k)-F(\f^k)|+|F(\f^k)-F(\f)|. 
\end{equation}
By (F), the first term on the right satisfies
$$
|F_j(\f_j)-F_j((1-\e_j)\f^k)|\le C \dd_{1,j}(\f_j,(1-\e_j)\f^k)^\a\max\{\dd_{1,j}(\f_j,0),\dd_{1,j}((1-\e_j)\f^k,0)\}^{1-\a}. 
$$
By strong continuity of the Darvas metrics (see Proposition~\ref{prop:dcont}), this yields 
$$
\limsup_j|F_j(\f_j)-F_j((1-\e_j)\f^k)|\le C \dd_1(\f,\f^k)^\a\max\{\dd_1(\f,0),\dd_1(\f^k,0)\}.
$$
Since $\lim_j|F_j((1-\e_j)\f^k)-F(\f^k)|=0$, by (F), we infer  
$$
\limsup_j|F_j(\f_j)-F(\f)|\le C \dd_1(\f,\f^k)^\a\max\{\dd_1(\f,0),\dd_1(\f^k,0)\}+|F(\f^k)-F(\f)|, 
$$
which tends to $0$ as $k\to\infty$, by continuity of $F$. 
\end{proof}

We next establish a general weak semicontinuity criterion. 

\begin{lem}\label{lem:weakext} Assume that: 
\begin{itemize}
\item[(i)] $\cR$ consists of continuous functions, and contains the constants; 
\item[(ii)] each $\cF_j$ is translation invariant, and $F_j$ is monotone increasing and translation equivariant.
\end{itemize}
Then $F$ admits a unique extension $F\colon\cF\to\R\cup\{-\infty\}$ that is continuous along decreasing sequences. It is further weakly usc, monotone increasing and translation equivariant, and for any weakly convergent sequence $\cF_j\ni\f_j\to\f\in\cF$, we have 
$$
\limsup_j F_j(\f_j)\le F(\f).
$$
\end{lem}
\begin{proof} Uniqueness is clear, by condition (R) above. By translation equivariance of $F_j$, there exists $a_j\in\R$ such that
\begin{equation}\label{equ:transFj}
F_j(\f+c)=F_j(\f)+c a_j
\end{equation}
for all $\f\in\cF_j$ and $c\in\R$. In particular, $F_j(1-\e_j)-F_j(0)=(1-\e_j)a_j$, and hence $a_j\to a:=F(1)-F(0)$. By (F), we conclude that $F(\f+c)=F(\f)+ac$ for $\f\in\cR$ and $c\in\R$, and that $F$ is also monotone increasing. Setting
\begin{equation}\label{equ:Fext}
F(\f):=\inf\{F(\p)\mid\p\in\cR,\,\p\ge\f\}
\end{equation}
for $\f\in\cF$ thus yields an extension $F\colon\cF\to\R\cup\{-\infty\}$, which is clearly also monotone increasing and translation equivariant. Assume $\f\in\cF$ is the limit of a decreasing sequence $(\f^k)$ in $\cF$. Pick $\p\in\cR$ such that $\f\le\p$. For each $\e>0$, we have $\f<\p+\e$ pointwise on $X$. Since $\f^k-\p$ is usc  and decreases to $\f-\p<\e$, a Dini type argument yields $\f^k<\p+\e$ for all $k$ large enough. Thus $\lim_k F(\f^k)\le F(\p)+a\e$, and hence $\lim_k F(\f^k)\le F(\f)$ by letting $\e\to 0$ and taking the infimum over $\p$. 

Now assume $\f\in\cF$ is the weak limit of $\f_j\in\cF_j$, and pick $\p\in\cR$ such that $\f\le\p$. Since $\p$ is continuous and $\f_j\to\f$ in $\PSH(C\om_X)$, Hartogs' lemma yields $\sup_X(\f_j-\p)\to\sup_X(\f-\p)\le 0$, and hence 
$$
\f_j\le\p+o(1)\le (1-\e_j)\p+o(1)
$$ 
as $j\to\infty$. Since $F_j$ is increasing and satisfies~\eqref{equ:transFj}, we infer
$F_j(\f_j)\le F_j((1-\e_j)\p)+o(1)$, and hence $\limsup_j F_j(\f_j)\le F(\p)$, by (F). Taking the infimum over $\p$ yields, as desired, $\limsup_j F_j(\f_j)\le F(\f)$. Finally, the very same argument shows that $F$ is weakly usc on $\cF$, and the proof is complete. 
\end{proof}

%
%
\subsection{Openness of coercivity}\label{sec:opencoer} 
In this section, we fix a closed subgroup $G\subset\Aut_0(X)$, and consider functionals
$$
M\colon\cF\to\R\cup\{+\infty\},\quad M_j\colon \cF_j\to\R\cup\{+\infty\},\quad j\in\N,  
$$
respectively defined on subsets of $\cE^1_j$ and $\cE^1$, and satisfying the following conditions:

\begin{itemize}
\item\textbf{invariance}: both $\cF_j,\cF$ and $M_j,M$ are invariant under translation and under the action of $G$;  
\item\textbf{normalization}: $0$ lies in $\cF_j$ and $\cF$, and $M_j(0)\to M(0)$; 
\item\textbf{lower semicontinuity}: if $\f_j\in\cF_j$ converges strongly to $\f\in\cE^1$, then $\f\in\cF$, and $\liminf_j M_j(\f_j)\ge M(\f)$; 
\item\textbf{convexity}:  the subset $\cF_j$ is convex with respect to psh geodesics, and the function $M_j$ is convex along such geodesics; 
\item\textbf{entropy growth}: there exist $\d,C>0$ such that 
$$
M_j(\f)\ge\d\ent_j(\f)-C(\dd_{1,j}(\f,0)+1)
$$ 
for all $j$ and all $\f\in\cE^1_j$; 
\item\textbf{properness}: for each $j$, the action of $G$ on $\cE^1_j$ is proper. 
\end{itemize} 
Recall that this last condition holds as soon as $\om_j$ is K\"ahler, see Proposition~\ref{prop:isom}.

\begin{thm}\label{thm:coeropen} Set 
$$
\cF^0:=\{\f\in \cF\: |\: \en(\f)=0\},\quad\cF^0_j=\{\f\in \cF\: |\, \en_j(\f)=0 \}, 
$$
and suppose given $\sigma,C\in\R$ such that 
$$
M(\f)\geq \sigma\, \dd_{1,G}(\f,0)-C
$$
for all $\f\in \cF^0$. Then, for any $\sigma'<\sigma$, there exist $C'\in\R$ and $j_0\in\Z_{\ge0}$ such that
$$
M_j(\f)\geq \sigma' \, d_{j,G}(\f,0)-C'
$$
for any $\f\in\cF_j^0$ and any $j\geq j_0$.
In particular if $M$ is coercive modulo $G$ (see Definition~\ref{defi:coerG}), then so is $M_j$ for all $j$ large enough.
\end{thm}
We are using the notation $d_{j,G}(\f,\p):=\inf_{g\in G}d_{1,j}(\f^g,\p)$, similarly to~\eqref{eqn:d_G}. 
\begin{rmk}\label{rmk:sigmalsc1}
    If we introduce the coercivity thresholds    
    \begin{gather*}
    \sigma^G:=\sup\left\{\sigma\in \R\: |\: M\geq \sigma\, \dd_{1,G}(\cdot,0)+O(1)\ \text{on $\cF^0$}\right\},\\
    \sigma^G_j:=\sup\left\{\sigma\in \R\: |\: M_j\geq \sigma\, d_{j,G}(\cdot,0)+O(1)\ \text{on $\cF^0_j$}\right\},
    \end{gather*}
    then Theorem~\ref{thm:coeropen} yields the semicontinuity property $\liminf_j \sigma^G_j\geq \sigma^G $.
\end{rmk}
\begin{lem}\label{lem:coeropen} If $\cE^1_j\ni\f_j\to\f\in\cE^1$ strongly, then $\f_j^g\to\f^g$ strongly for any $g\in\Aut_0(X)$. 
\end{lem}
\begin{proof} Since $\en_j(\f_j^g)=\en_j(\f_j)$ converges to $\en(\f)=\en(\f^g)$ (see Proposition~\ref{prop:isom}), it suffices to show that $\f_j^g\to\f^g$ weakly. Recall from~\S\ref{sec:auto} that $\f^g=\tau_g+g^\star\f$ (resp.~$\f_j^g=\tau_{j,g}+g^\star\f_j$) where $\tau_g$ (resp.~$\tau_{j,g}$) denotes the unique smooth function such that $g^\star\om=\om+\ddc\tau_g$ and $\en(\tau_g)=0$ (resp.~$g^\star\om_j=\om_j+\ddc\tau_{j,g}$ and $\en_j(\tau_{j,g})=0$). It thus suffices to show $\tau_{j,g}\to\tau_g$ smoothly. Since $g^\star\om_j=\om_j+\ddc\tau_{j,g}$ converges smoothly to $g^\star\om=\om+\ddc\tau_g$, we can find constants $a_j\in\R$ such that $\tau_{j,g}+a_j\to\tau_g$ smoothly. By~\eqref{equ:equivariance}, $V_ja_j=\en_j(\tau_{j,g}+a_j)$ thus converges to $\en(\tau_g)=0$. Since $V_j\to V>0$, this shows $a_j\to 0$, and hence $\tau_{j,g}\to\tau_g$ smoothly, as desired. 
\end{proof}

\begin{proof}[Proof of Theorem~\ref{thm:coeropen}] Since $M_j(0)\to M(0)$, we may replace $M_j,M$ by $M_j-M_j(0),M-M(0)$ and assume without loss that $M_j(0)=M(0)=0$. Pick any $\sigma'<\sigma$. Arguing by contradiction, and passing to a subsequence, if necessary, we can find $\f_j\in\cF_j^0$ such that $M_j(\f_j)\le\sigma'\, d_{j,G}(\f_j,0)-C_j$ where $C_j\to+\infty$.

As $G$ acts properly on $\cE^1_j$, Lemma \ref{lem:dJ} implies that we can find $g_j\in G$ such that $d_{j,G}(\f_j,0)=\dd_{1,j}(\f_j^{g_j},0)$ for any $j$. By $G$-invariance of $M_j$ we may further assume $d_{j,G}(\f_j,0)=\dd_{1,j}(\f_j,0)$, \ie $\{\f_j,0\}$ is $G$-minimal (see Definition~\ref{defi:min}), so that
\begin{equation}\label{equ:Fmup}
    M_j(\f_j)\le\sigma' \dd_{1,j}(\f_j,0)-C_j.
\end{equation}

Denote by $(\f_{j,t})_{t\in [0,T_j]}$ the unit speed psh geodesic joining $0$ to $\f_j$ in $\cE^1_j$, where $T_j:=\dd_{1,j}(\f_j,0)$. The entropy growth condition combined with~\eqref{equ:Fmup} yield 
$$
-C'(T_j+1)\le M_j(\f_j)\le\sigma' T_j-C_j,
$$
where $C'>0$ can be chosen such that $C'+\sigma'>0$,
and $C_j\to+\infty$ thus implies $T_j\to+\infty$. Since $t\mapsto M_j(\f_{j,t})$ is convex and vanishes at $t=0$, it satisfies
\begin{equation}\label{equ:Fmbd}
M_j(\f_{j,t})\le\frac{t}{T_j}M_j(\f_j)<t\sigma',
\end{equation}
by~\eqref{equ:Fmup}. Fix $t>0$, to be determined later on.
Since $\dd_{1,j}(\f_{j,t},0)=t$ is bounded, $|\sup_X\f_{j,t}|$ remains bounded (see Lemmas~\ref{lem:supd} and~\ref{lem:weakcomp}), and the entropy growth condition implies that $\sup_j\ent_j(\f_{j,t})<+\infty$. By Theorem~\ref{thm:entcomp}, we may thus assume, after passing to a subsquence, that $\f_{j,t}$ converges strongly to $\f_t\in\cE^1$ (recall that $t$ is fixed here). By~\eqref{equ:Fmbd} and the lower semicontinuity condition, we get $\f_t\in\cF^0$, $\dd_1(\f_t,0)=t$ and $M(\f_t)\le t\sigma'$. 

We claim that $\dd_{1,G}(\f_t,0)=\dd_1(\f_t,0)=t$. Granted this, the coercivity of $M$ modulo $G$ yields
$$
M(\f_t)\ge \sigma\, \dd_{1,G}(\f_t,0)-C =t\sigma - C
$$
for some $C>0$ independent of $t$, and it only remains to chose $t>C/(\sigma-\sigma')$ to reach a contradiction. To show the claim, note that Lemma~\ref{lem:Gmin} implies that the geodesic $(\f_{j,t})_{t\in [0,T_j]}$ is $G$-minimal, and hence $\dd_{1,j}(\f_{j,t},0)\le \dd_{1,j}(\f_{j,t}^g,0)$ for all $g\in G$. By Lemma~\ref{lem:coeropen}, $\f_{j,t}^g\to\f_t^g$ strongly, and Proposition~\ref{prop:dcont} thus yields $\dd_1(\f_t,0)\le \dd_1(\f_t^g,0)$ for all $g\in G$, which proves the claim.
\end{proof}

\section{The weighted formalism}\label{sec:weighted}
The purpose of this section is to review material from~\cite{Lah,AJL}, and explain how to extend it to spaces with log terminal singularities. Along the way, we provide some slightly more explicit estimates required for our later applications, together with a small generalization of some of the constructions to equivariant currents. 

%
%

\subsection{Equivariant forms and currents}\label{sec:equiv}
Until further notice, $X$ denotes a compact K\"ahler manifold, of dimension $n=\dim X$. 
\subsubsection{} The \emph{linear automorphism group}\footnote{This is sometimes known as the \emph{reduced automorphism group}, but we prefer this terminology which goes back to \cite{Fuj}.} 
$$
\Autr(X)\subset\Aut_0(X)
$$
is defined as the identity component of the subgroup of automorphisms acting trivially on the Albanese torus 
$$
\Alb(X)=\Hnot(X,\Om^1_X)^\vee/\mathrm{H}_1(X,\Z).
$$ 
The Lie algebra of $\Autr(X)$ thus consists of all holomorphic vector fields $\xi\in\Hnot(X,T_X)$ such that $\a(\xi)=0$ for any holomorphic $1$-form $\a\in\Hnot(X,\Om^1_X)$. 
\begin{exam} Assume $X$ is projective, and pick an ample line bundle $L$. Then $\Autr(X)$ coincides with the image of the map $\Aut_0(X,L)\to\Aut_0(X)$, which thus induces an isomorphism 
$$
\Autr(X)\simeq\Aut_0(X,L)/\C^\times. 
$$
\end{exam}
In the general K\"ahler case, $\Autr(X)$ still carries a canonical structure of linear algebraic group~\cite[Corollary~5.8]{Fuj} (which explains the chosen terminology), and $\Aut_0(X)/\Autr(X)$ is a compact complex torus~\cite[Theorem~5.5]{Fuj}. 

\subsubsection{Moment maps} From now on, we fix a compact torus 
$$
T\subset\Autr(X),
$$
with Lie algebra $\ft$. As is well-known, 
any closed, real, $T$-invariant $(1,1)$-form $\om$ on $X$ admits a \emph{moment map} 
$$
m\colon X\to\ft^\vee,
$$ 
\ie a $T$-invariant smooth map such that, for each $\xi\in\ft$, $m^\xi:=\langle m,\xi\rangle$ satisfies 
\begin{equation}\label{equ:moment}
-\de m^\xi=i(\xi)\om=\om(\xi,\cdot). 
\end{equation}
Note that $m$ is uniquely determined up to an additive constant in $\ft^\vee$, and satisfies
\begin{equation}\label{equ:dcm}
-\dc m^\xi=i(J\xi)\om,\quad-\ddc m^\xi=\cL_{J\xi}\om,
\end{equation}
where $J$ denotes the complex structure of $X$. Using the language 
of equivariant de Rham cohomology, we will refer to the pair 
$$
\Om=(\om,m)
$$
as a \emph{(closed, real) equivariant form} with moment map $m_\Om:=m$. 

More generally, the above definitions make sense when $\om$ is a closed, real, $T$-invariant $(1,1)$-current and $m$ a $\ft^\vee$-valued distribution; we then say that $\Om=(\om,m)$ is a (closed, real) \emph{equivariant current}. Since any closed $(1,1)$-current $\om$ is $\ddc$-cohomologous to a closed $(1,1)$-form, the existence of a moment map for $\om$ reduces to the following fact: 

\begin{exam}\label{exam:ddcT} For any $T$-invariant distribution $f$ on $X$, setting 
\begin{equation}\label{equ:mddc}
m_f^\xi:=\dc f(\xi)=-\de f(J\xi)
\end{equation}
for $\xi\in\ft$ defines a moment map $m_f$ for the $T$-invariant $(1,1)$-current $\ddc f$. We denote by
$$
\ddcT f=(\ddc f,m_f)
$$
the resulting equivariant current. 
\end{exam}
For later use, we record the identity
\begin{equation}\label{equ:dcf}
m_{f}^\xi\,\om^n=n\,\dc f\wedge \om^{n-1}\wedge i(\xi)\om
\end{equation}
for any closed $(1,1)$-form $\om$, which follows from applying the antiderivation $i(\xi)$ to the trivial relation $\dc f\wedge\om^n=0$.

\subsubsection{Equivariant curvature forms} Let $L$ be a $T$-equivariant holomorphic line bundle on $X$, and $\phi$ a (smooth, Hermitian) $T$-invariant metric on $L$, with curvature form $\ddc\phi\in 2\pi c_1(L)$ (using additive notation). The algebraic torus $T_\C$ acts on $L$, and hence on metrics on $L$, and setting
\begin{equation}\label{equ:mphi}
m_{\phi}^\xi:=-\frac{d}{dt}\bigg|_{t=0} (e^{tJ\xi})^\star\phi
\end{equation}
for $\xi\in\ft$ defines a moment map for $\ddc\phi$. We denote by 
$$
\ddcT\phi:=(\ddc\phi,m_{\phi})
$$
the resulting equivariant form, which we call the \emph{equivariant curvature form} of $(L,\phi)$. 

Conversely, any $T$-invariant form $\om\in 2\pi c_1(L)$ can be written as the curvature form of a $T$-invariant metric $\phi$ on $L$, unique up to an additive constant. By~\eqref{equ:mphi}, the corresponding equivariant form $\Om$ only depends on $\om$ and the equivariant line bundle $L$, and will be called the \emph{$L$-normalized lift} of $\om$. 

More generally, one can define the \emph{equivariant curvature current} $\ddcT\phi$ of any singular $T$-invariant metric $\phi$ on $L$, thanks to Example~\ref{exam:ddcT}.

\begin{exam}\label{exam:Ricci} There is a 1--1 correspondence between (smooth, positive) $T$-invariant volume forms $\nu$ on $X$ and (smooth) $T$-invariant metrics $\tfrac 12\log \nu$ on the canonical bundle $K_X$. The \emph{equivariant Ricci form} 
$$
\Ric^T(\nu)=(\Ric(\nu),m_{\nu})
$$
is defined as the equivariant curvature form of the anticanonical bundle $(-K_X,-\tfrac 12\log \nu)$, \ie the $(-K_X)$-normalized lift of the usual Ricci form
$$
\Ric(\nu)=-\ddc\tfrac 12\log \nu\in 2\pi c_1(X). 
$$
The corresponding moment map satisfies 
\begin{equation}\label{equ:momu}
m_{\nu}^\xi=\frac{\cL_{J\xi}\nu}{2 \nu}, 
\end{equation} 
and is thus characterized by the normalization $\int_X m_{\nu}\,\nu=0$. If $\nu:=\om^n$ is the volume form of a K\"ahler form $\om$,  then $\Ric(\om):=\Ric(\om^n)$ is the usual Ricci curvature form of $\om$, and one easily checks using~\eqref{equ:dcm} that 
\begin{equation}\label{equ:momRic}
m_{\om^n}=-\tfrac 12\Delta_\om m
\end{equation}
for any choice of moment map $m$ for $\om$. 
\end{exam} 

\begin{exam}\label{exam:intequ} Let $D$ be a $T$-invariant divisor on $X$, and denote by $[D]$ the associated integration current. The canonical meromorphic section $s_D$ of $\cO_X(D)$ defines a $T$-invariant singular metric $\log|s_D|$ on $\cO_X(D)$, such that $\ddc\log|s_D|=2\pi[D]$ by the Lelong--Poincar\'e formula. We define the \emph{equivariant integration current} of $D$ as 
\begin{equation}\label{equ:intequ}
[D]_T:=\frac{1}{2\pi}\ddcT\log|s_D|.
\end{equation}
\end{exam}
%
%
\subsection{Moment polytopes}\label{sec:moments}

Recall first that the set $X^T\subset X$ of $T$-fixed points is a (finite) disjoint union of connected closed submanifolds, along which any moment map is clearly constant by~\eqref{equ:moment}. 
For any equivariant form $\Om=(\om,m_\Om)$, the image $m_\Om(X^T)\subset\ft^\vee$ is thus finite, and depends continuously on $\Om$. 

\begin{defi} The \emph{Duistermaat--Heckman measure} of an equivariant form $\Om=(\om,m_\Om)$ is defined as the compactly supported signed measure 
$$
\DH_\Om:=(m_\Om)_\star\om^n
$$
on $\ft^\vee$. We say that $\Om$ (or the moment $m_\Om$) is \emph{centered} if the barycenter of $\DH_\Om$ is at the origin, \ie 
$\int_X m_\Om^\xi\,\om^n=0$ for all $\xi\in\ft$. 
\end{defi}
If $\{\om\}^n=\int_X\om^n\ne 0$, this last condition uniquely determines $m_\Om$ in terms of $\om$. We then call it the \emph{centered moment} of $\om$, and write $m_\om:=m_\Om$. 

\begin{lem}\label{lem:smoothmom} If a sequence $(\om_j)$ of closed $(1,1)$-forms converges smoothly to $\om$ such that $\{\om\}^n\ne 0$, then their centered moments satisfy $m_{\om_j}\to m_\om$ smoothly. 
\end{lem}
  \begin{proof} Fix $\xi\in\ft$. By~\eqref{equ:dcm}, $\ddc m^\xi_{\om_j}=-\cL_{J\xi}\om_j$ converges smoothly to $\ddc m^\xi_\om=-\cL_{J\xi}\om$. By ellipticity of $\ddc$, it follows that $m^\xi_{\om_j}+c_j\to m^\xi_\om$ smoothly for some constants $c_j\in\R$. In particular, $c_j\{\om_j\}^n=\int_X(m^\xi_{\om_j}+c_j)\om_j^n\to\int_X m^\xi_\om\om^n=0$. Since $\{\om_j\}^n\to \{\om\}^n\ne 0$, we get $c_j\to 0$, and the result follows. 
\end{proof}

The above discussion applies in particular when $\om$ is semipositive and big. In that case, we further have:

\begin{lem}\label{lem:poly} Assume $\om$ is semipositive and big, with centered moment $m_\om$. Then: 
\begin{itemize} 
\item[(i)] $m_\om(X)$ coincides with the convex envelope of the finite set $m_\om(X^T)$, and only depends on $\{\om\}\in H^{1,1}(X)$; 
\item[(ii)] consider the cocyle $(\tau_g)_{g\in\Aut_0(X)}$ of~\S\ref{sec:auto}. For each $\xi\in\ft$ we then have
\begin{equation}\label{equ:mtau}
-m_\om^\xi=\dot\tau_{J\xi}:=\frac{d}{dt}\bigg|_{t=0}\tau_{\exp(tJ\xi)}. 
\end{equation} 
\end{itemize}
\end{lem} 
\begin{proof} To prove (i), assume first that $\om>0$ is K\"ahler. Then $m_\om(X)$ coincides with the support of the Duistermaat-Heckman measure $(m_\om)_\star\om^n$, which only depends on $\{\om\}$ (see Corollary~\ref{cor:DH} below). The Atiyah--Guillemin--Sternberg theorem further implies that $m_\om(X)$ is the convex envelope of the finite set $m_\om(X^T)$. These properties still hold in the general case, since the centered moments satisfy $\lim_{\e\to 0_+} m_{\om+\e\om_X}=m_\om$ smoothly, by Lemma~\ref{lem:smoothmom}. 


To see (ii), recall that $g^\star\om=\om+\ddc\tau_g$ and $\en(\tau_g)=0$. Applying this with $g=\exp(tJ\xi)$ and differentiating at $t=0$ yields $\cL_{J\xi}\om=\ddc\dot\tau_{J\xi}$ and $\int_X\dot\tau_{J\xi}\om^n=0$. On the other hand, $m_\om^\xi$ also satisfies $\int_X m_\om^\xi\,\om^n=0$ and $-\ddc m_\om^\xi=\cL_{J\xi}\om$, by~\eqref{equ:dcm}, and~\eqref{equ:mtau} follows. 
\end{proof}

\begin{defi} For any equivariant form $\Om=(\om,m_\Om)$ with $\om$ semipositive and big, 
$$
P_\Om:=m_\Om(X)\subset\ft^\vee
$$
is called the \emph{moment polytope} of $\Om$.
\end{defi}
By Lemma~\ref{lem:poly}~(i), $P_\Om$ is indeed a polytope, being a translate of the image of the centered moment map. 
%
%
\subsection{Weighted Monge--Amp\`ere operators and weighted energy}\label{sec:wen}
\subsubsection{Weighted Monge--Amp\`ere operators} Pick any equivariant form $\Om=(\om,m_\Om)$. For any $\f\in C^\infty(X)^T$ we set
$$
\Om_\f:=\Om+\ddcT\f=(\om_\f,m_\Om+m_\f), 
$$
see Example~\ref{exam:ddcT}. 

\begin{defi} Given any smooth function $v\in C^\infty(\ft^\vee)$ and equivariant current $\Theta=(\theta,m_\Theta)$, we define the following operators on $C^\infty(X)^T$: 
\begin{itemize}
\item[(i)] the \emph{$v$-weighted Monge--Amp\`ere operator} 
\begin{equation}\label{equ:MAv}
\MA_{\Om,v}(\f):=v(m_{\Om_\f})\om_\f^n;
\end{equation}
\item[(ii)] the \emph{$\Theta$-twisted, $v$-weighted Monge--Amp\`ere operator} 
\begin{equation}\label{equ:twMA}
\MA_{\Om,v}^\Theta(\f):=v(m_{\Om_\f}) n\,\theta\wedge\om_\f^{n-1}+\langle v'(m_{\Om_\f}),m_\Theta\rangle\om_\f^n. 
\end{equation}
\end{itemize}
\end{defi}
Here $\MA_{\Om,v}(\f)$ is a smooth $(n,n)$-form, while $\MA_{\Om,v}^\Theta(\f)$ is an $(n,n)$-current. When $\Theta$ is smooth, we can more simply write 
\begin{equation}\label{equ:twMAder}
\MA_{\Om,v}^\Theta(\f)=\frac{d}{dt}\bigg|_{t=0}\MA_{\Om+t\Theta,v}(\f).
\end{equation}
In the general case, both $\MA_{\Om,v}(\f)$ and $\MA_{\Om,v}^\Theta(\f)$ are linear with respect to $v$, the latter being further linear with respect to $\Theta$. Note also that
\begin{equation}\label{equ:MAtrans}
\MA_{\Om,v}(\f+f)=\MA_{\Om_f,v}(\f),\quad\MA_{\Om,v}^\Theta(\f+f)=\MA_{\Om_f,v}^\Theta(\f)
\end{equation}
for all $\f,f\in C^\infty(X)^T$, and hence
\begin{equation}\label{equ:MAdiff}
\frac{d}{dt}\bigg|_{t=0}\MA_{\Om,v}(\f+t f)=\MA_{\Om,v}^{\ddcT f}(\f),
\end{equation}
by~\eqref{equ:twMAder}. 

The following symmetry property will be crucial for what follows: 

\begin{lem}\label{lem:symm} Pick $\f\in C^\infty(X)^T$, and let $f,g$ be distributions, at least one of which is smooth. 
\begin{itemize}
\item[(i)] If $f$ is $T$-invariant, then 

\begin{equation}\label{equ:twMAddc}
\int_X g\MA_{\Om,v}^{\ddcT f}(\f)=-n\int_X v(m_{\Om_\f}) \de g\wedge \dc f\wedge\om_\f^{n-1}.
\end{equation}
\item[(ii)] If $g$ is also $T$-invariant, then 
\begin{equation}\label{equ:twMAsymm}
\int_X g\MA_{\Om,v}^{\ddcT f}(\f)=\int_X f\MA_{\Om,v}^{\ddcT g}(\f). 
\end{equation}
\end{itemize}
\end{lem}

\begin{proof} By~\eqref{equ:MAtrans}, we may replace $\Om$ with $\Om_\f$ and assume $\f=0$, for notational simplicity.  By~\eqref{equ:twMA}, we have
$$
\int_X g\MA_{\Om,v}^{\ddcT f}(0)=n\int_X g\,v(m_\Om)\ddc f\wedge\om^{n-1}+\int_X g\langle v'(m_\Om), m_{f}\rangle\om^n
$$
$$
=-n\int_X v(m_\Om) \de g\wedge\dc f\wedge\om^{n-1}-n\int_X g\,\de v(m_\Om)\wedge\dc f\wedge\om^{n-1}+\int_X g\langle v'(m_\Om), m_{ f}\rangle\om^n,
$$
by integration-by-parts. We thus need to show
\begin{equation}\label{equ:twMA1}
n\int_X g\,\de v(m_\Om)\wedge\dc f\wedge\om^{n-1}=\int_X g\langle v'(m_\Om), m_{ f}\rangle\om^n.
\end{equation}
To see this, pick a basis $(\xi_\a)$ of $\ft$, and denote by $v'_\a$ the partial derivatives of $v$ in the dual basis. Then 
\begin{equation}\label{equ:dv}
\de v(m_\Om)=\sum_\a v'_\a(m_\Om) d m_\Om^{\xi_\a}=-\sum_\a v'_\a(m_\Om) i(\xi_\a)\om,
\end{equation}
since $m_\Om$ is a moment for $\om$. On the other hand, 
$$
\langle v'(m_\Om), m_{ f}\rangle=\sum_\a v'_\a(m_\Om) m_{ f}^{\xi_\a}, 
$$
and~\eqref{equ:twMA1} now follows from~\eqref{equ:dcf}. Finally, we get \eqref{equ:twMAsymm} by symmetry of the $(1,1)$-part of $\de g\wedge \dc f$ with respect to $(f,g)$. 
\end{proof}
\subsubsection{Weighted energy} In view of~\eqref{equ:MAdiff} and the symmetry property~\eqref{equ:twMAsymm}, we may introduce (see~\S\ref{sec:func}):

\begin{defi}\label{defi:wen} The \emph{$v$-weighted Monge--Amp\`ere energy} with respect to $\Om$ 
$$
\en_{\Om,v}\colon C^\infty(X)^T\to\R
$$
is defined as the Euler--Lagrange functional of the $v$-weighted Monge--Amp\`ere operator, normalized by $\en_{\Om,v}(0)=0$. 
\end{defi}
Thus 
\begin{equation}\label{equ:wen}
\langle\en_{\Om,v}'(\f),f\rangle:=\frac{d}{dt}\bigg|_{t=0}\en_{\Om,v}(\f+t f)=\int_X f\,\MA_{\Om,v}(\f)
\end{equation}
for all $\f,f\in C^\infty(X)^T$, and hence
\begin{equation}\label{equ:wenexp}
\en_{\Om,v}(\f)-\en_{\Om,v}(\p)=\int_0^1 dt\int_X(\f-\p)\MA_{\Om,v}(t\f+(1-t)\p).
\end{equation}
for all $\f,\p\in C^\infty(X)^T$. 

Note further that $\en_{\Om,v}$ is linear with respect to $v$. By general principles (see~\S\ref{sec:func}), the translation invariance of the $v$-weighted Monge--Amp\`ere operator implies that 
$$
\int_X\MA_{\Om,v}(\f)=\int_X v(m_{\Om_\f})\,\om_\f^n=\int_{\ft^\vee} v\,\DH_{\Om_\f}
$$ 
is independent of $\f$, and yields the translation equivariance property
\begin{equation}\label{equ:wenequ}
\en_{\Om,v}(\f+c)=\en_{\Om,v}(\f)+c\int_X v(m_\Om)\om^n
\end{equation}
for $\f\in C^\infty(X)^T$ and $c\in\R$. Since this holds for all $v\in C^\infty(\ft^\vee)$, we infer: 

\begin{cor}\label{cor:DH} For any equivariant form $\Om$, the Duistermaat--Heckman measure $\DH_{\Om_\f}=(m_{\Om_\f})_\star\om_\f^n$ is independent of $\f\in C^\infty(X)^T$. In particular, if $\Om$ is centered, then so is $\Om_\f$ for all $\f\in C^\infty(X)^T$. 
\end{cor}

\subsubsection{The twisted Monge--Amp\`ere energy} 

We next pass to the twisted case and show: 

\begin{prop}\label{prop:twEL} For any equivariant current $\Theta$, the twisted weighted Monge--Amp\`ere operator $C^\infty(X)^T\ni\f\mapsto\MA_{\Om,v}^\Theta(\f)$ admits an Euler--Lagrange functional. 
\end{prop}

\begin{lem}\label{lem:twEL} For any $T$-invariant distribution $g$, $C^\infty(X)^T\ni\f\mapsto\int_X g\MA_{\Om,v}(\f)$ is an Euler--Lagrange functional for $\MA_{\Om,v}^{\ddcT g}(\f)$. 
\end{lem}
\begin{proof} For any $f\in C^\infty(X)^T$, \eqref{equ:MAdiff} and~\eqref{equ:twMAsymm} yield
$$
\frac{d}{dt}\bigg|_{t=0}\int_X g\MA_{\Om,v}(\f+tf)=\int_X g\MA_{\Om,v}^{\ddcT f}(\f)=\int_X f\MA_{\Om,v}^{\ddcT g}(\f),
$$
which proves the result. 
\end{proof}

\begin{proof}[Proof of Proposition~\ref{prop:twEL}] Assume first that $\Theta$ is smooth. Since $\en_{\Om,v}(\f)$ is an Euler--Lagrange functional for $\MA_{\Om,v}(\f)$, 
$\frac{d}{dt}\big|_{t=0}\en_{\Om+t\Theta,v}(\f)$
is then an Euler--Lagrange equation for 
$$
\MA_{\Om,v}^\Theta(\f)=\frac{d}{dt}\bigg|_{t=0}\MA_{\Om+t\Theta,v}(\f),
$$
see~\eqref{equ:twMAder}. In the general case, we can find a equivariant (smooth) $(1,1)$-form $\tilde\Theta$ and a distribution $g$ such that $\Theta=\tilde\Theta+\ddcT g$. By the previous step and Lemma~\ref{lem:twEL}, $\MA_{\Om,v}^{\tilde\Theta}(\f)$ and $\MA_{\Om,v}^{\ddcT g}(\f)$ both admit an Euler--Lagrange functional, and hence so does 
$$
\MA_{\Om,v}^\Theta(\f)=\MA_{\Om,v}^{\tilde\Theta}(\f)+\MA_{\Om,v}^{\ddcT g}(\f),
$$
completing the proof.
\end{proof}

\begin{defi} For any equivariant current $\Theta$, the \emph{$\Theta$-twisted, $v$-weighted Monge--Amp\`ere energy} 
$$
\en_{\Om,v}^\Theta\colon C^\infty(X)^T\to\R
$$ 
is defined as the Euler--Lagrange functional of $\MA_{\Om,v}^\Theta(\f)$, normalized by $\en_{\Om,v}^\Theta(0)=0$. 
\end{defi}
For all $\f,\p\in C^\infty(X)^T$ we thus have
\begin{equation}\label{equ:twenexp}
\en_{\Om,v}^\Theta(\f)-\en_{\Om,v}^\Theta(\p)=\int_0^1 dt\int_X(\f-\p)\MA_{\Om,v}^\Theta(t\f+(1-t)\p),
\end{equation}
and $\en_{\Om,v}^\Theta$ is linear in both $\Theta$ and $v$. When $\Theta$ is smooth, we further have 
$$
\en_{\Om,v}^\Theta(\f)=\frac{d}{dt}\bigg|_{t=0}\en_{\Om+t\Theta,v}(\f),
$$
as noted in the proof of Proposition~\ref{prop:twEL}. As a direct consequence of Lemma \ref{lem:twEL}, we further have: 
\begin{lem}\label{lem:3.14revisited}
    For any $T$-invariant distribution $g$ and $\f,\psi\in C^{\infty}(X)^T$ we have
    $$
    \en_{\Om,v}^{\ddcT g}(\f)-\en_{\Om,v}^{\ddcT g}(\psi)=\int_X g\big(\MA_{\Om,v}(\f)-\MA_{\Om,v}(\psi)\big).
    $$
\end{lem}
%
%
%
%
\subsection{The weighted scalar curvature}\label{sec:scal}

\subsubsection{Notation} From now on, we fix a $T$-equivariant K\"ahler form\footnote{Note that, while only centered moments are considered in~\cite{Lah}, this is no longer the case in~\cite{AJL}, especially in Section 2 thereof.} $\Om=(\om,m_\Om)$, \ie an equivariant form such that $\om>0$ is K\"ahler, and denote by 
$$
P=P_\Om=m_\Om(X)\subset\ft^\vee
$$ 
its moment polytope. We also fix a smooth function $v\in C^\infty(\ft^\vee)$. 

For any $\f\in C^\infty(X)^T$ and any equivariant current $\Theta$, we set for simplicity
\begin{equation*}
\MA_v(\f):=\MA_{\Om,v}(\f),\quad \MA_v^\Theta(\f):=\MA_{\Om,v}^\Theta(\f)
\end{equation*}
and 
$$
\en_v(\f):=\en_{\Om,v}(\f),\quad\en_v^\Theta(\f):=\en_{\Om,v}^\Theta(\f). 
$$
We denote by 
$$
\cH^T=\left\{\f\in C^\infty(X)^T\mid\om_\f>0\right\}
$$ 
the space of $T$-invariant K\"ahler potentials. As a consequence of Corollary~\ref{cor:DH} we note:

\begin{lem}\label{lem:polyinde} For any $\f\in\cH^T$ the moment polytope of $\Om_\f$ coincides with $P$.
\end{lem}

The centralizer 
$$
\Aut_0^T(X)\subset\Aut_0(X)
$$
of $T$ in $\Aut_0(X)$ acts on equivariant forms and currents, and the right-action of $\Aut_0(X)$ on $\cH$ (see~\S\ref{sec:auto}) restricts to a right-action of $\Aut_0^T(X)$ on $\cH^T$. 

\begin{lem}\label{lem:quasiinv} For all $\f,\p\in\cH^T$ and $g\in\Aut_0^T(X)$ we have
\begin{equation}\label{equ:invOm}
\Om_{\f^g}=g^\star\Om_\f; 
\end{equation}
\begin{equation}\label{equ:invMA}
\MA_v(\f^g)=g^\star\MA_v(\f); 
\end{equation}
\begin{equation}\label{equ:invtwMA}
\MA_v^\Theta(\f^g)=g^\star\MA_v^{g_\star\Theta}(\f); 
\end{equation}
\begin{equation}\label{equ:quasien}
\en_v(\f^g)-\en_v(\p^g)=\en_v(\f)-\en_v(\p); 
\end{equation}
\begin{equation}\label{equ:quasitwen}
\en_v^\Theta(\f^g)-\en_v^\Theta(\p^g)=\en_v^{g_\star\Theta}(\f)-\en_v^{g_\star\Theta}(\p). 
\end{equation}
\end{lem}
\begin{proof} After translating $m_\Om$ by a vector in $\ft^\vee$, we may assume without loss that $\Om$ is centered. 
By definition of the right-action of $\Aut_0(X)$ on $\cH^T$, we have $\om_{\f^g}=g^\star\om_\f$. Since $\Om$ is centered, $\Om_{\f^g}$, $\Om_\f$, and hence also $g^\star\Om_\f$, are centered as well (see Corollary~\ref{cor:DH}). This implies~\eqref{equ:invOm} Using~\eqref{equ:MAv}, \eqref{equ:twMA}, we infer~\eqref{equ:invMA}, \eqref{equ:invtwMA}, and~\eqref{equ:wenexp}, \eqref{equ:twenexp} now yield~\eqref{equ:quasien}, \eqref{equ:quasitwen}. 
\end{proof} 

\subsubsection{Weighted trace and weighted Laplacian} 
From now on we assume $v>0$ on $P_\Om$, \ie $v(m_\Om)>0$ on $X$. 

\begin{defi} For any equivariant current $\Theta=(\theta,m_\Theta)$ and any $T$-invariant distribution $f$, we introduce
\begin{itemize}
\item the \emph{$v$-weighted trace} 
\begin{equation}\label{equ:wtr}
\tr_{\Om,v}(\Theta):=\frac{\MA_v^\Theta(0)}{\MA_v(0)}=\tr_\om(\theta)+\langle (\log v)'(m_\Om),m_\Theta\rangle; 
\end{equation}
\item the \emph{$v$-weighted Laplacian} 
\begin{equation}\label{equ:wlap}
\D_{\Om,v} f:=\tr_{\Om,v}(\ddcT f)=\D_\om f+\langle (\log v)'(m_\Om),m_{ f}\rangle. 
\end{equation}
\end{itemize}
\end{defi}
When $\Theta$ is smooth, its weighted trace computes the logarithmic directional derivative 
\begin{equation}\label{equ:wtrdiff} 
\tr_{\Om,v}(\Theta)=\frac{d}{dt}\bigg|_{t=0}\log \MA_{\Om+t\Theta,v}(0)
\end{equation}
of the weighted volume form. Since
\begin{equation}\label{equ:lapMA}
\left(\D_{\Om,v} f\right)\MA_v(0)=\MA^{\ddcT f}_{\Om,v}(0), 
\end{equation}
Lemma~\ref{lem:symm} yields the integration-by-parts formula 

\begin{equation}\label{equ:wlap1}
\int_X g\left(\D_{\Om,v} f\right) \MA_v(0)=\int_X f\left(\D_{\Om,v} g\right) \MA_v(0)=-\int_X\langle \de f,\de g\rangle_\om \MA_v(0)
\end{equation}
for all $T$-invariant distributions $f,g$, at least one of which is smooth. Here 
$$
\langle \de f,\de g\rangle_\om:=\tr_\om(\de f\wedge \dc g), 
$$
which computes the pointwise scalar product of the $1$-forms $\de f,\de g$ when $f,g$ are smooth. In other words, we have 
$$
-\D_{\Om,v}=\de^{\star,v}\de,
$$
where $\de^{\star,v}$ denotes the adjoint of $\de$ with respect to the weighted volume form $\MA_v(0)$, which shows that $-\D_{\Om,v}$ coincides with the (geometers') weighted Laplacian of~\cite[Appendix~A]{AJL}. Applying~\eqref{equ:wlap1} with $f=g$ further yields: 
\begin{lem}\label{lem:wlap} The kernel of $\D_{\Om,v}$ in $C^\infty(X)^T$ consists of constants.
\end{lem}

\subsubsection{Weighted scalar curvature} We next introduce: 
\begin{defi} The \emph{$v$-weighted equivariant Ricci curvature} and \emph{$v$-weighted scalar curvature} of $\Om$ are defined as 
$$
\Ric^T_v(\Om):=\Ric^T(\MA_v(0)),\quad S_v(\Om):=\tr_{\Om,v}(\Ric^T_v(\Om)).
$$
\end{defi}

\begin{exam}\label{exam:sol} Assume $\om\in 2\pi c_1(X)=2\pi c_1(-K_X)$, with $(-K_X)$-normalized lift $\Om$ (see~\S\ref{sec:moments}). Then $\om$ is a \emph{$v$-soliton} in the sense~\cite[Definition~2.1]{AJL} (see also~\cite{BWN}) iff $\Om$ satisfies the weighted K\"ahler--Einstein equation
$$
\Ric^T_v(\Om)=\Om. 
$$
\end{exam}

We next provide a more explicit expression for the weighted scalar curvature, and compare it with the one originally introduced by Lahdili\footnote{Note that this author uses a different normalization of the scalar curvature and Laplacian.} in~\cite{Lah}, here denoted by
\begin{equation}\label{equ:Slah0}
S^{\Lah}_v(\Om):= v(m_\Om) S(\om)-\Delta_\om v(m_\Om)+\tfrac 12\langle v''(m_\Om),g_\om\rangle, 
\end{equation}
where $g_\om$ denotes the Riemannian metric defined by $\om$. More concretely, 
$$
\langle v''(m_\Om),g_\om\rangle=\sum_{\a,\b}v''_{\a\b}(m_\Om)g_\om(\xi_\a,\xi_\b)
$$
in terms of a basis $(\xi_\a)$ of $\ft$.

\begin{lem}\label{lem:Sv} We have  
\begin{equation}\label{equ:Sv1}
S_v(\Om)=S(\om)-\frac{\langle v'(m_\Om),\D_\om m_\Om\rangle+\Delta_\om v(m_\Om)}{2 v(m_\Om)},
\end{equation}
and 
\begin{equation}\label{equ:SvLah}
S^{\Lah}_v(\Om)=v(m_\Om) S_v(\Om). 
\end{equation}
\end{lem}
\begin{proof} Set $f:=\tfrac 12\log v(m_\Om)$. Then
$$
\Ric_v^T(\Om)=\Ric^T(e^{2f} \om^n)=\Ric^T(\om^n)-\ddcT f,
$$
and hence
\begin{equation}\label{equ:Sv2}
S_v(\Om)=\tr_{\Om,v}(\Ric^T(\om^n))-\D_{\Om,v} f=S(\om)+\langle (\log v)'(m_\Om),-\tfrac 12\D_\om m_\Om\rangle-\D_{\Om,v} f, 
\end{equation}
by~\eqref{equ:momRic} and~\eqref{equ:wtr}. Now
$$
2\D_\om f=\frac{1}{v(m_\Om)}\D_\om v(m_\Om)-\frac{1}{v(m_\Om)^2} \tr_\om(\de v(m_\Om)\wedge \dc v(m_\Om)), 
$$
where  
$$
\de v(m_\Om)\wedge \dc v(m_\Om)=\sum_{\a,\b} v'_\a(m_\Om) v'_\b(m_\Om) i(\xi_\a)\om\wedge i(J\xi_\b)\om, 
$$
by~\eqref{equ:dcm}, and hence 
$$
\tr_\om (\de v(m_\Om)\wedge \dc v(m_\Om))=\sum_{\a,\b} v'_\a(m_\Om) v'_\b(m_\Om) g_\om(\xi_\a,\xi_\b).
$$
On the other hand, \eqref{equ:mddc} and~\eqref{equ:dcm} yield
\begin{align*}
2\langle (\log v)'(m_\Om),m_{ f}\rangle & =-\frac{1}{v(m_\Om)^2}\sum_\a v'_\a(m_\Om) i(J\xi_\a) \de v(m_\Om) \\
& =\frac{1}{v(m_\Om)^2}\sum_{\a,\b} v'_\a(m_\Om) v'_\b(m_\Om) g_\om(\xi_\a,\xi_\b).
\end{align*}
We conclude that
$$
\D_{\Om,v} f=\D_\om f+\langle (\log v)'(m_\Om),m_{ f}\rangle=\frac{\D_\om v(m_\Om)}{2 v(m_\Om)},
$$
and injecting this into~\eqref{equ:Sv2} proves~\eqref{equ:Sv1}. Finally, a further simple computation yields
$$
\D_\om v(m_\Om)=\langle v'(m_\Om),\D_\om m_\Om\rangle+\langle v''(m_\Om),g_\om\rangle.
$$
This implies~\eqref{equ:SvLah} and concludes the proof. 
\end{proof}

\begin{defi}\label{defi:cscK} Given $w\in C^\infty(\ft^\vee)$, we say that $\Om$ is a \emph{$(v,w)$-cscK} metric if $S_v(\Om)=w(m_\Om)$. 
\end{defi}

\begin{rmk} By~\eqref{equ:SvLah}, $\Om$ is a $(v,w)$-cscK metric in the present sense iff $\om$ is a $(v,vw)$-cscK metric in the sense of~\cite{Lah,AJL} (for the given choice of moment map $m_\Om$). 
\end{rmk}

In order to illustrate the present formalism, we provide a simple proof of \cite[Proposition~1]{AJL}. 

\begin{prop} \label{Prop:Ricci_Solitons} Assume that $\om\in 2\pi c_1(X)$, with $(-K_X)$-normalized lift $\Om$. The following are equivalent:
\begin{itemize}
\item[(i)] $\om$ is a $v$-soliton, \ie $\Ric^T_v(\Om)=\Om$ (see Example~\ref{exam:sol}); 
\item[(ii)] $\Om$ is $(v,w)$-cscK with $w(\a):=n+\langle(\log v)'(\a),\a\rangle$.
\end{itemize}
\end{prop}
\begin{proof} By~\eqref{equ:wtr}, we have $\tr_{\Om,v}(\Om)=n+\langle(\log v)'(m_\Om),m_\Om\rangle$, and (ii) thus holds iff $\Ric_v^T(\Om)-\Om$ has $v$-weighted trace $0$. Since both $\Ric_v^T(\Om)$ and $\Om$ are equivariant curvature forms of $T$-invariant metrics on $-K_X$, we have $\Ric_v^T(\Om)=\Om+\ddcT\rho$ for some $\rho\in C^\infty(X)^T$, and hence 
$$
\tr_{\Om,v}(\Ric_v^T(\Om)-\Om)=\D_{\Om,v}\rho. 
$$
By Lemma~\ref{lem:wlap}, this vanishes iff $\rho$ is constant, and the desired equivalence follows. 
\end{proof}

%
%
%
%
\subsection{The weighted Mabuchi energy}\label{sec:mab}
Fix a $T$-invariant reference volume form $\nu_X$ on $X$, with associated relative entropy functional $\Ent(\cdot|\nu_X)$, see~Appendix~\ref{sec:ent}. 

\begin{defi} The \emph{$v$-weighted entropy} and \emph{$v$-weighted Ricci energy} 
$$
\ent_v\colon\cH^T\to\R,\quad\enR_v\colon\cH^T\to\R
$$
are respectively defined by setting for $\f\in\cH^T$
$$
\ent_v(\f):=\tfrac 12\Ent(\MA_v(\f)|\nu_X)=\tfrac12\int_X\log\left(\frac{\MA_v(\f)}{\nu_X}\right)\MA_v(\f)
$$
and
$$
\enR_v(\f):=\en_v^{-\Ric^T(\nu_X)}(\f).
$$
\end{defi}
The factor $1/2$ in the entropy owes to the fact that the metric on $K_X$ corresponding to $\nu_X$ is $\tfrac 12\log \nu_X$ in our convention. Recall also that its equivariant curvature form is $-\Ric^T(\nu_X)$. The key point is now: 

\begin{lem}\label{lem:mabscal} The functional $\cH^T\ni\f\mapsto\ent_v(\f)+\enR_v(\f)$
is an Euler--Lagrange functional for $\f\mapsto - S_v(\Om_\f)\MA_v(\f)$. 
\end{lem}
\begin{proof} By~\eqref{equ:wtr} and~\eqref{equ:lapMA}, we have 
\begin{align*}
-S_v(\Om_\f)\MA_v(\f) & =\MA_v^{-\Ric_v^T(\Om_\f)}(\f)\\
& =\MA_v^{-\Ric^T(\nu_X)}(\f)+(\D_{\Om_\f,v}\rho_\f)\MA_v(\f)
\end{align*}
with
$$
\rho_\f:=\frac 12\log\left(\frac{\MA_v(\f)}{\nu_X}\right). 
$$
Since $\enR_v(\f)$ is, by definition, an Euler--Lagrange functional of $\f\mapsto\MA_v^{-\Ric_v^T(\nu_X)}(\f)$, we are reduced to showing that 
$$
\ent_v(\f)=\int_X \rho_\f\,\MA_v(\f)
$$
is an Euler--Lagrange functional for the operator $\f\mapsto (\D_{\Om_\f,v}\rho_\f)\MA_v(\f)$. To see this, pick $f\in C^\infty(X)^T$. By~\eqref{equ:MAdiff} and~\eqref{equ:lapMA}, we have 
$$
\frac{d}{dt}\bigg|_{t=0}\MA_v(\f+t f)=\left(\D_{\Om_\f,v} f\right)\MA_v(\f),\quad\frac{d}{dt}\bigg|_{t=0} \rho_{\f+t f}=\frac 12\D_{\Om_\f,v} f, 
$$
and hence 
$$
\frac{d}{dt}\bigg|_{t=0}\ent_v(\f+t f)=\frac 12\int_X\left(\D_{\Om_\f,v} f\right)\MA_v(\f)+\int_X \rho_\f\left(\D_{\Om_\f,v} f\right)\MA_v(\f). 
$$
Using~\eqref{equ:wlap1} we infer, as desired, 
\begin{equation}\label{equ:diffent}
\frac{d}{dt}\bigg|_{t=0}\ent_v(\f+t f)=\int_X f\left(\D_{\Om_\f,v} \rho_\f\right)\MA_v(\f). 
\end{equation}
\end{proof}

\begin{defi} Given any $w\in C^\infty(\ft^\vee)$, we define the \emph{$(v,w)$-weighted Mabuchi energy} $\mab_{v,w}\colon\cH^T\to\R$ by setting
\begin{equation}\label{equ:MCT}
\mab_{v,w}(\f):=\ent_v(\f)+\enR_v(\f)+\en_{vw}(\f).
\end{equation}
\end{defi}

As a consequence, of Lemma~\ref{lem:mabscal}, we get: 

\begin{prop}\label{prop:mabder} The weighted Mabuchi energy $\mab_{v,w}$ is an Euler--Lagrange functional of 
$$
\cH^T\ni\f\mapsto (w(m_{\Om_\f})-S_v(\Om_\f))\MA_v(\f).
$$
In particular, $\f\in\cH^T$ is a critical point of $\mab_{v,w}$ iff $\Om_\f$ is $(v,w)$-cscK. 
\end{prop}

\begin{rmk} By~\eqref{equ:SvLah}, $S_v(\Om_\f)\MA_v(\f)=S_v^{\Lah}(\Om_\f)\om_\f^n$, and Proposition~\ref{prop:mabder} is thus equivalent to~\cite[Theorem~5]{Lah}.
\end{rmk}

\begin{rmk}\label{rmk:invvol} In~\cite{AJL}, the volume form $\nu_X=\om^n$ is used in the definition of the weighted Mabuchi energy. Any other volume form $\nu_X$ can be used as well, as Lemma~\ref{lem:mabscal} implies that $\ent_v(\f)+\enR_v(\f)$, and hence $\mab_{v,w}(\f)$, only depend on $\nu_X$ by an overall additive constant. This can also be directly checked using Lemma~\ref{lem:twEL}, see Lemma~\ref{lem:indep} below. 
\end{rmk}

%
%
\subsection{The relative weighted Mabuchi energy}\label{sec:relmab}

\subsubsection{Invariance properties}
Assume, as above, that $v,w$ are smooth functions on $\ft^\vee$, with $v>0$ on the moment polytope $P$. As recalled in~\S\ref{sec:func}, the translation invariance of 
$$
\cH^T\ni\f\mapsto\mab'_{v,w}(\f)=\left(w(m_{\Om_\f})-S_v(\om_\f)\right)\MA_v(\f)
$$ 
implies that
$$
\int_X\mab'_{v,w}(\f)=\int_X\mab'_{v,w}(0)
$$
is independent of $\f$, and yields the translation equivariance property
\begin{equation}\label{equ:mabtrans}
\mab_{v,w}(\f+c)=\mab_{v,w}(\f)+c\int_X\mab'_{v,w}(0),\quad \f\in\cH^T,\,c\in\R. 
\end{equation} 
On the other hand, \eqref{equ:invOm} implies that $\mab'_{v,w}$ is invariant under the right-action of $\Aut_0^T(X)$ on $\cH^T$, which yields the quasi-invariance property
\begin{equation}\label{equ:quasimab}
\mab_{v,w}(\f^g)-\mab_{v,w}(\p^g)=\mab_{v,w}(\f)-\mab_{v,w}(\p).
\end{equation}
for all $\f,\p\in\cH^T$ and $g\in\Aut_0^T(X)$. Futhermore, $\mab_{v,w}$ is $T_\C$ -invariant iff the Futaki invariant
$$
\Fut_{v,w}(\xi)=\int_X\dot\tau_\xi\mab'_{v,w}(0)
$$
vanishes for all $\xi\in\ft_\C$ (we recall the notation $\dot\tau_\xi:=\frac{d}{dt}\big|_{t=0}\tau_{\exp(t\xi)}$). Since $T$ acts trivially on $\cH^T$, the Futaki invariant vanishes for $\xi\in\ft\subset\ft_\C$, while~\eqref{equ:mtau} yields
$$
\Fut_{v,w}(J\xi)=-\int_X m_\Om^\xi\mab'_{v,w}(0)+\langle\a,\xi\rangle\int_X \mab'_{v,w}(0), 
$$
where $\a\in\ft^\vee$ is such that the moment map $m_\Om-\a$ is centered. We conclude: 

\begin{lem}\label{lem:mabinv} The following properties are equivalent: 
\begin{itemize}
\item[(i)] $\mab_{v,w}$ is translation invariant and $T_\C$-invariant;
\item[(ii)] for all affine functions $\ell=\xi+c\in\ft\oplus\R$ on $\ft^\vee$ we have $\int_X\ell(m_\Om)\mab'_{v,w}(0)=0$, \ie
$$
\int_X\ell(m_\Om)w(m_\Om)\MA_v(0)=\int_X\ell(m_\Om) S_v(\Om)\MA_v(0). 
$$ 
\end{itemize}
\end{lem}
\subsubsection{Extremal function and relative Mabuchi energy}
From now on, we assume $w>0$ on $P$. Extending~\cite{FM}, we introduce as in~\cite{Lah}: 

\begin{defi} The \emph{weighted Futaki--Mabuchi pairing} on the space $\ft\oplus\R$ of affine functions $\ell$ on $\ft^\vee$ is defined by setting
$$
\langle\ell,\ell'\rangle:=\int_X\ell(m_\Om)\ell'(m_\Om) w(m_\Om)\MA_v(0)=\int_X\ell(m_\Om)\ell'(m_\Om) \MA_{vw}(0). 
$$
\end{defi}

\begin{lem}\label{lem:FM} The weighted Futaki--Mabuchi pairing is positive definite.
\end{lem}
\begin{proof} The pairing is clearly symmetric and semipositive. Furthermore, any $\ell=\xi+c\in\ft\oplus\R$ in the kernel satisfies $\ell(m_\Om)=0$, \ie $m_\Om^\xi=-c$ on $X$. Thus $-d m_\Om^\xi=i(\xi)\om=0$, and hence $|\xi|^2=\om(\xi,J\xi)=0$, which yields $\xi=0$, and then $c=0$ as well.  
\end{proof}
We may thus introduce: 

\begin{defi}\label{defi:Mrel} Assume $v,w>0$ on $P$. Then: 
\begin{itemize}
\item[(i)] the \emph{weighted extremal function} is defined as the unique affine function $\ell^\ext=\ell^\ext_{\Om,v,w}$ on $\ft^\vee$ such that 
\begin{equation}\label{equ:ellext}
\langle\ell,\ell^\ext\rangle=\int_X\ell(m_\Om)S_v(\Om)\MA_v(0)
\end{equation}
for all affine functions $\ell$; 
\item[(ii)] $\Om$ is called a \emph{$(v,w)$-extremal metric} if it is $(v,w\ell^\ext)$-cscK, \ie
$$
S_v(\Om)=w(m_\Om)\ell^\ext(m_\Om).
$$
\end{itemize}
\end{defi}
By construction, $\ell^\ext$ is the unique affine function such $\mab_{v,w\ell^\ext}$ satisfies the equivalent conditions of Lemma~\ref{lem:mabinv}. The functional 
$$
\mab^\rel_{v,w}:=\mab_{v,w\ell^\ext}\colon\cH^T\to\R
$$
is called the \emph{relative weighted Mabuchi energy}. Explicitly, 
$$
\mab^\rel_{v,w}(\f)=\ent_v(\f)+\enR_v(\f)+\en_{v w\ell^\ext}(f). 
$$
Note that $\f\in\cH^T$ is a critical point of $\mab_{v,w}^\rel$ iff $\Om_\f$ is $(v,w)$-extremal. 

\begin{exam} For any given $\f\in\cH^T$, $\Om_\f$ is $(v,w)$-cscK iff it is $(v,w)$-extremal and the equivalent conditions of Lemma~\ref{lem:mabinv} hold. In that case, $\ell^\ext=1$, and hence $\mab^\rel_{v,w}=\mab_{v,w}$. 
\end{exam}


%
%
\subsection{Extension to $\cE^1$}\label{sec:mabext}

In what follows, we extend the weighted Monge--Amp\`ere operator and weighted energy functionals to the space
$$
\cET:=\cE^1(X,\om)^T
$$
of $T$-invariant potentials of finite energy. Our treatment provides a slightly simpler approach compared to~\cite{BWN,HL,AJL}, along with more explicit estimates that are crucial for our later purposes. 

\subsubsection{Weighted Monge--Amp\`ere operator} We first recall:  

\begin{lem}\label{lem:ETcomp} Each $\f\in\cET$ can be written as the limit of a decreasing sequence in $\cH^T$. In particular, $\cET$ coincides with the closure of $\cH^T$ in $\cE^1$. 
\end{lem} 
\begin{proof} By~\cite{BK}, $\f$ can be written as the limit of a decreasing sequence $(\f_j)$ in $\cH$. For each $j$, the average $\tf_j:=\int_T \f_j^\tau d\tau$  with respect to the Haar measure of $T$ lies in $\cH^T$, and $(\tf_j)$ is a decreasing sequence that also converges to $\f$, since the latter is $T$-invariant. 
\end{proof}

\begin{prop}\label{prop:MAvE1} For any weight $v\in C^\infty(\ft^\vee)$, the weighted Monge--Amp\`ere operator 
$$
\cH^T\ni\f\mapsto\MA_v(\f)=v(m_{\Om_\f})\om_\f^n
$$
admits a unique continuous extension to $\cET$ with values in the space of $T$-invariant signed Radon measures. For each $\f\in\cET$, the measure $\MA_v(\f)$ further integrates all functions in $\cE^1$, and the H\"older estimate
\begin{equation}\label{equ:MAvhold}
\left|\int\tau(\MA_v(\f)-\MA_v(\p))\right|\lesssim\|v\|_{C^0(P)} \dd_1(\f,\p)^{1/2}\max\{\dd_1(\f,0),\dd_1(\p,0),\dd_1(\tau,0)\}^{1/2}
\end{equation}
holds for all $\f,\p\in\cET$ and $\tau\in\cE^1$. 
\end{prop}
By continuity, the equivariance property~\eqref{equ:invMA} remains valid on $\cET$. 

\begin{proof} Set $C:=\|v\|_{C^0(P)}$. Writing $v=(v+C)-C$, we may assume without loss $v\ge 0$ on $P$, and hence $\MA_v(\f)\ge 0$ for $\f\in\cH^T$. We first establish~\eqref{equ:MAvhold} for $\f,\p\in\cH^T$ and $\tau\in\cH$. For $t\in [0,1]$ set $\f_t:=t\f+(1-t)\p$. By  Lemma~\ref{lem:symm}, we get
$$
\int\tau(\MA_v(\f)-\MA_v(\p))=\int_0^1 dt\int_X\tau\frac{d}{dt}\MA_v(\f_t)
$$
$$
=-n\int_0^1 dt\int_X v(m_{\Om_{\f_t}})\de\tau\wedge\dc (\f-\p)\wedge\om_{\f_t}^{n-1}
$$
$$
=-\sum_{p=0}^{n-1}\int_X a_p\,\de\tau\wedge\dc (\f-\p)\wedge\om_\f^p\wedge\om_\p^{n-1-p},
$$
$$
a_p= \binom{n-1}{p}\int_0^1 t^p(1-t)^{n-1-p}v(m_{\Om_{\f_t}})dt\in C^\infty(X). 
$$
Since $v\ge 0$, the bilinear forms
$$
(f,g)\mapsto\int a_p\,\de f\wedge\dc g\wedge\om_\f^p\wedge\om_\p^{n-1-p}
$$
are semipositive, and the Cauchy--Schwarz inequality thus yields 
$$
\left|\int_X a_p\,\de\tau\wedge\dc (\f-\p)\wedge\om_\f^p\wedge\om_\p^{n-1-p}\right|\le\left(\int_X a_p\,\de\tau\wedge\dc\tau\wedge\om_\f^p\wedge\om_\p^{n-1-p}\right)^{1/2}
$$
$$
\times\left(\int_X a_p\,\de (\f-\p)\wedge\dc (\f-\p)\wedge\om_\f^p\wedge\om_\p^{n-1-p}\right)^{1/2}. 
$$
By Lemma~\ref{lem:poly}, we have $m_{\Om_{\f_t}}(X)=P$ for all $t$, thus $0\le v(m_{\Om_{\f_t}})\le C$, and hence $0\le a_p\lesssim C$ for all $p$ and $t$. We have thus proved 
$$
\left|\int\tau(\MA_v(\f)-\MA_v(\p))\right|\lesssim C\sum_{p=0}^{n-1} \left(\int_X\de\tau\wedge\dc\tau\wedge\om_\f^p\wedge\om_\p^{n-1-p}\right)^{1/2}
$$
$$
\times\left(\int_X \de (\f-\p)\wedge\dc (\f-\p)\wedge\om_\f^p\wedge\om_\p^{n-1-p}\right)^{1/2},
$$
and~\eqref{equ:MAvhold} now follows just as in the proof of Corollary~\ref{cor:MAhold}. 

Since $v\ge 0$, $\MA_v(\f)$ is a positive Radon measure, of mass independent of $\f\in\cH^T$ (see Corollary~\ref{cor:DH}). In order to extend $\MA_v$ to $\cET$, it therefore suffices to extend $\f\mapsto\int_X f\,\MA_v(\f)$ for any $f\in C^\infty(X)$. Since $\om$ is K\"ahler, $f$ can be written as a linear combination of functions in $\cH$, and~\eqref{equ:MAvhold} thus shows that $\f\mapsto\int_Xf\,\MA_v(\f)$ is uniformly continuous on $\cH^T$ with respect to the Darvas metric $\dd_1$, and we conclude by Lemma~\ref{lem:ETcomp}.

Finally~\eqref{equ:MAvhold} remains valid for all $\f,\p\in\cET$ and $\tau\in\cH$, by approximation by functions in $\cH^T$, and then for all $\tau$ in $\cE^1$, by monotone approximation by functions in $\cH$. This completes the proof. 
\end{proof}

\subsubsection{Weighted functionals} We next consider a equivariant form $\Theta=(\theta,m_\Theta)$. 
We choose a norm $\|\cdot\|$ on $\ft$, and use the same notation for the dual norm on $\ft^\vee$.

\begin{prop}\label{prop:wenE1} For any $v\in C^\infty(\ft^\vee)$, the functionals $\en_v$, $\en_v^\Theta$ admit unique continuous extensions
$$
\en_v\colon\cET\to\R,\quad \en_v^\Theta\colon\cET\to\R, 
$$
which further respectively satisfy the Lipschitz estimate
\begin{equation}\label{equ:wenlip}
\left|\en_v(\f)-\en_v(\p)\right|\le A(v)\,\dd_1(\f,\p)
\end{equation}
and the H\"older estimate
\begin{equation}\label{equ:twenhold}
\left|\en_v^\Theta(\f)-\en_v^\Theta(\p)\right|\lesssim B(v,\Theta)\,\dd_1(\f,\p)^\a\max\{\dd_1(\f,0),\dd_1(\p,0)\}^{1-\a},
\end{equation}
for all $\f,\p\in\cET$, where $\a:=2^{-n}$ and 
$$
A(v):=\sup_P|v|,\quad B(v,\Theta):=\|\theta\|_\om\sup_P|v|+(\sup_X\|m_\Theta\|)(\sup_P\|v'\|). 
$$
\end{prop}
Here again, \eqref{equ:quasien}, \eqref{equ:quasitwen} remain valid on $\cET$, by continuity. 

\begin{lem}\label{lem:wenE1} For all $\f,\p\in\cH^T$ we have $\dd_1(\f,\p)=\inf\{\int_0^1dt\int_X|\dot\f_t|\om_{\f_t}^n\}$, where the infimum ranges over all smooth paths $(\f_t)_{t\in [0,1]}$ in $\cH^T$ joining $\f$ to $\p$.
\end{lem}
\begin{proof} When $T$ is trivial, the result is proved in~\cite[Theorem~3.5]{Dar}, by showing that X.X.~Chen's $\e$-geodesics $(\f^\e_t)_{t\in [0,1]}$ joining $\f$ to $\p$, which are smooth for $\e>0$, satisfy $\dd_1(\f,\p)=\inf_{\e>0}\int_0^1 dt\int_X|\dot\f^\e_t|\om_{\f^\e_t}^n$. The same argument applies in the general case, since the $\e$-geodesic, being the unique solution of a complex Monge--Amp\`ere equation with boundary data induced by $\f,\p\in\cH^T$, is automatically $T$-invariant. 
\end{proof}

\begin{proof}[Proof of Proposition~\ref{prop:wenE1}] Arguing as in the proof of Proposition~\ref{prop:MAvE1}, it is enough to show~\eqref{equ:wenlip} and~\eqref{equ:twenhold} when $\f,\p\in\cH^T$ are K\"ahler potentials. Pick any smooth path $(\f_t)_{t\in [0,1]}$ in $\cH^T$ joining $\f$ to $\p$. Since $\en_v'(\f)=\MA_v(\f)=v(m_{\Om_{\f}})\om_\f^n$, we get
$$
\en_v(\f)-\en_v(\p)=\int_0^1 dt\int_X\dot\f_t v(m_{\Om_{\f_t}})\om_{\f_t}^n, 
$$
and hence 
$$
\left|\en_v(\f)-\en_v(\p)\right|\le A(v)\int_0^1 dt\int_X|\dot\f_t|\om_{\f_t}^n.
$$
Taking the infimum over $(\f_t)$ yields~\eqref{equ:wenlip}, by Lemma~\ref{lem:wenE1}. Next we consider the path $\f_t:=t\f+(1-t)\p$, and use~\eqref{equ:twenexp}, \eqref{equ:twMA} to get
\begin{align*}
\en_v^\Theta(\f)-\en_v^\Theta(\p)&=\int_0^1 dt\int_X(\f-\p)\left[v(m_{\Om_{\f_t}}) n\theta\wedge\om_{\f_t}^{n-1}+\langle v'(m_{\Om_{\f_t}}),m_\Theta\rangle\om_{\f_t}^n\right]\\
&=\sum_{p=0}^{n-1}\int_X a_p(\f-\p)\theta\wedge\om_\f^p\wedge\om_\p^{n-1-p}+\sum_{p=0}^n\int_X b_p (\f-\p)\om_\f^p\wedge\om_\p^{n-1-p}
\end{align*}
where $a_p,b_p\in C^\infty(X)$ are defined by 
$$
a_p:=n{n-1\choose p}\int_0^1t^p(1-t)^{n-1-p} v(m_{\Om_{\f_t}})dt,\quad b_p:={n\choose p}\int_0^1 t^p(1-t)^{n-p}\langle v'(m_{\Om_{\f_t}}),m_\Theta\rangle dt.
$$
Since $m_{\Om_{\f_t}}$ takes values in $P$, we have $|a_p|\lesssim \sup_P|v|$. Since $\pm\theta\le\|\theta\|_\om\om$, we also have
$$
\pm\theta\wedge\om_\f^p\wedge\om_\p^{n-1-p}\le\|\theta\|_\om\om\wedge\om_\f^p\wedge\om_\p^{n-1-p},
$$
which yields 
$$
\left|\int_X a_p(\f-\p)\theta\wedge\om_\f^p\wedge\om_\p^{n-1-p}\right|\lesssim\|\theta\|_\om(\sup_P|v|)\int_X|\f-\p|\,\om\wedge\om_\f^p\wedge\om_\p^{n-1-p}
$$
$$
\lesssim B(v,\Theta) \dd_1(\f,\p)^\a\max\{\dd_1(\f,0),\dd_1(\p,0\}^{1-\a},
$$
where the last estimate follows from Theorem~\ref{thm:MAcont}. Similarly, $|b_p|\lesssim(\sup_X\|m_\Theta\|)(\sup_P\|v'\|)\le B(v,\Theta)$, and hence 
$$
\left|\int_X b_p (\f-\p)\om_\f^p\wedge\om_\p^{n-1-p}\right|\lesssim B(v,\Theta)\int_X|\f-\p|\,\om_\f^p\wedge\om_\p^{n-1-p}
$$
$$
\lesssim B(v,\Theta) \dd_1(\f,\p)^\a \max\{\dd_1(\f,0),\dd_1(\p,0\}^{1-\a},
$$
using again Theorem~\ref{thm:MAcont}. Putting these estimates together proves~\eqref{equ:twenhold}. 
\end{proof}

By continuity, Lemma~\ref{lem:3.14revisited} yields: 
\begin{exam}\label{exam:twEL} For all $f\in C^\infty(X)^T$ and $\f,\p\in\cET$ we have 
$$
\en_v^{\ddcT f}(\f)-\en_v^{\ddcT f}(\p)=\int_X f\left(\MA_v(\f)-\MA_v(\p)\right). 
$$
\end{exam}

As we next show, this identity also provides a natural extension of $\en_v^{\ddcT f}$ when the function $f$ is merely quasi-psh. 

\begin{prop}\label{prop:twenext} Pick a weight $v\in C^\infty(\ft^\vee)$ such that $v>0$ on $P$, and a quasi-psh function $f$ on $X$. Then $\en_v^{\ddcT f}\colon\cH^T\to\R$ admits a unique extension
    $$
      \en_v^{\ddcT f}: \cE^{1,T}\to \R\cup \{-\infty\}
    $$
    that is continuous along decreasing sequences. It is further weakly usc, and satisfies 
$$
\en_v^{\ddcT f}(\f)-\en_v^{\ddcT f}(\p)=\int_X f\left(\MA_v(\f)-\MA_v(\p)\right)
$$ 
for all $\f\in\cET$ and $\p\in\cH^T$. 
\end{prop} 

The proof is based on the following monotonicity criterion (which will also be put to use in~\S\ref{sec:semicontfunc}). 

\begin{lem}\label{lem:enmono} Given any equivariant form $\Theta=(\theta,m_\Theta)$, there exists a compact subset $K\subset\ft^\vee$ only depending on $\theta$ such that the following holds: if $v,w\in C^\infty(\ft^\vee)$ satisfy 
\begin{equation}\label{equ:condmono}
v(\a)\ge 0,\quad\langle v'(\a),\b\rangle+w(\a)\ge 0\quad\text{for all}\quad\a\in P,\,\b\in K,
\end{equation}
then $\en_v^{\Theta+\ddcT f}+\en_w$ is monotone increasing on $\cH^T$ for any $\theta$-psh function $f$. 
\end{lem}
\begin{proof} Pick $v,w\in C^\infty(\ft^\vee)$ and a $\theta$-psh function $f$, and write $\theta=\theta^+-\theta^-$ with $\theta^\pm>0$. Since $\theta^++\ddc f\ge 0$, $f$ can be written as the limit of a decreasing sequence $(f_k)$ with $f_k\in C^\infty(X)^T$ and $\theta^++\ddc f_k>0$ (see Lemma \ref{lem:ETcomp}). 

By Lemma~\ref{lem:polyinde}, the moment polytope $m_{\Theta^++\ddcT f_k}(X)$ is independent of $k$. Writing 
$$
m_{\Theta+\ddcT f_k}=m_{\Theta^++\ddcT f_k}-m_\Theta^-,
$$
this shows that $m_{\Theta+\ddcT f_k}(X)$ is contained in a fixed compact set $K\subset\ft^\vee$ only depending on $\theta$. Assume~\eqref{equ:condmono} holds. For any $\f\in\cH^T$ we have $m_{\Om_\f}(X)=P$, again by Lemma~\ref{lem:polyinde}, and hence 
$$
\langle v'(m_{\Om_\f}),m_{\Theta+\ddcT f_k}\rangle+w(m_{\Om_\f})\ge 0
$$
on $X$ for all $k$. Since $m_{\ddcT f_k}\to m_{\ddcT f}$ weakly, we infer 
$$
\langle v'(m_{\Om_\f}),m_{\Theta+\ddcT f}\rangle+w(m_{\Om_\f})\ge 0
$$ as a distribution on $X$. Thus
$$
(\en_v^{\Theta+\ddcT f}+\en_w)'(\f)=\MA_v^{\Theta+\ddcT f}(\f)+\MA_w(\f)
$$
$$
=v(m_{\Om_\f}) n(\theta+\ddc f)\wedge\om_\f^{n-1}+\left(\langle v'(m_{\Om_\f}),m_{\Theta+\ddcT f}\rangle+w(m_{\Om_\f})\right)\om_\f^n\ge 0
$$
and the result follows. 
\end{proof}

\begin{proof}[Proof of Proposition~\ref{prop:twenext}] Pick $A>0$ such that $A\om+\ddc f\ge 0$. Since $v>0$ on $P$, we can choose $B>0$ such that 
$\langle v'(\a),\b\rangle+B v(\a)\ge 0$ for all $\a\in P$, $\b\in K$, where $K\subset\ft^\vee$ denotes the compact subset provided by Lemma~\ref{lem:enmono} applied to $\Theta:=A\Om$ (which can in fact be taken as $K=AP$ in that case). Since $f$ is $A\om$-psh, it follows that the functional
$$
F:=\en_v^{\ddcT f}+A\en_v^\Om+B\en_v
$$
is monotone increasing on $\cH^T$. We extend it by monotonicity to $\cE^{1,T}$ by setting
$$
F(\f):=\inf \left\{F(\psi)\, |\, \psi \in \cH^T, \psi\geq \f\right\}
$$
for $\f\in\cET$. A Dini type argument as in Lemma~\ref{lem:weakext} leads to the continuity of $F$ along decreasing sequences. Since, for any weakly convergent sequence $\cE^{1,T}\ni \f_j\to \f\in \cE^{1,T}$, $(\sup_{l\geq j}\f_l)^\star$ decreases to $\f$ as $j\to\infty$, it follows that $F$ is weakly usc. Setting 
$$
\en_v^{\ddcT f}(\f):=F(\f)-A\en_v^\Om(\f)-B\en_v(\f)
$$
for $\f\in\cET$ yields the desired extension of $\en_v^{\ddcT f}$ to $\cET$. 

Next pick $\f\in\cET$, $\p\in\cH^T$, and choose a decreasing sequence $\f_j\in\cH^T$ such that $\f_j\searrow\f\in\cET$. For each $j$, Lemma~\ref{lem:3.14revisited} yields  
$$
\en_v^{\ddcT f}(\f_j)-\en_v^{\ddcT f}(\p)=\int f\left(\MA_v(\f_j)-\MA_v(\p)\right). 
$$
By continuity of $\en_v^{\ddcT f}$ along decreasing sequences we have $\en_v^{\ddcT f}(\f_j)\to\en_v^{\ddcT f}(\f)$, while the weak convergence of measures $\MA_v(\f_j)\to\MA_v(\f)$ (see Proposition~\ref{prop:MAvE1}) implies
$$
\limsup_j\int f\,\MA_v(\f_j)\le\int f\,\MA_v(\f)
$$
since $f$ is usc. We infer 
$$
\en_v^{\ddcT f}(\f)-\en_v^{\ddcT f}(\p)\le\int f\left(\MA_v(\f)-\MA_v(\p)\right). 
$$
To get the converse inequality, pick further a decreasing sequence $f_k\in C^\infty(X)^T$ such that $A\om+\ddc f_k>0$ and $f_k\searrow f$. Since $\ddcT f_k\to\ddcT f$ as currents, we have $\en_v^{\ddcT f_k}\to\en_v^{\ddcT f}$ pointwise on $C^\infty(X)^T$. 

By Lemma~\ref{lem:enmono}, $F_k:=\en_v^{\ddcT f_k}+A\en_v^\Om+B\en_v$ is monotone increasing for each $k$, and we thus get for all $k,j$
$$
\int_X f_k\left(\MA_v(\f)-\MA_v(\p)\right)=\en_v^{\ddcT f_k}(\f)-\en_v^{\ddcT f_k}(\p)
$$
$$
=F_k(\f)-A\en_v^\Om(\f)-B\en_v(\f)-\en_v^{\ddcT f_k}(\p)\le F_k(\f_j)-A\en_v^\Om(\f)-B\en_v(\f)-\en_v^{\ddcT f_k}(\p)
$$
$$
=\en_v^{\ddcT f_k}(\f_j)-\en_v^{\ddcT f_k}(\p)+A\left(\en_v^\Om(\f_j)-\en_v^\Om(\f)\right)+B\left(\en_v(\f_j)-\en_v(\f)\right). 
$$
Letting first $k\to\infty$ yields
$$
\int_X f\left(\MA_v(\f)-\MA_v(\p)\right)\le\en_v^{\ddcT f}(\f_j)-\en_v^{\ddcT f}(\p)+A\left(\en_v^\Om(\f_j)-\en_v^\Om(\f)\right)+B\left(\en_v(\f_j)-\en_v(\f)\right), 
$$
by monotone convergence and using $\lim_k\en_v^{\ddcT f_k}(\f_j)=\en_v^{\ddcT f}(\f_j)$ since $\f_j$ is smooth. Letting now $j\to\infty$ we conclude 
$$
\int f\left(\MA_v(\f)-\MA_v(\p)\right)\le\en_v^{\ddcT f}(\f)-\en_v^{\ddcT f}(\p),
$$
by continuity along decreasing sequences of $\en_v^{\ddcT f}$, $\en_v^\Om$ and $\en_v$. 
\end{proof}

\subsubsection{Weighted Mabuchi energy} 
Pick smooth functions $v,w$ on $\ft^\vee$ with $v>0$ on $P$. 

\begin{defi} The \emph{weighted entropy, weighted Ricci energy and weighted Mabuchi energy} 
$$
\ent_v\colon\cET\to\R\cup\{+\infty\},\quad\enR_v\colon\cET\to\R,\quad\mab_{v,w}\colon\cET\to\R\cup\{+\infty\}
$$
are defined by setting for $\f\in\cET$
$$
\ent_v(\f):=\tfrac12\Ent(\MA_v(\f)|\nu_X),\quad\enR_v(\f):=\en_v^{-\Ric^T(\nu_X)}(\f)
$$
and
$$
\mab_{v,w}(\f)=\ent_v(\f)+\enR_v(\f)+\en_{vw}(\f).
$$
\end{defi}
By continuity of $\MA_v$ and lower semicontinuity of the relative entropy, $\ent_v$ and $\mab_{v,w}$ are both lsc on $\cET$. 

\begin{rmk}\label{rmk:AJL} The entropy $\ent_v$, and hence also $\mab_{v,w}$, are in fact characterized as the largest lsc extensions from $\cH^T$ to $\cET$ (see~\cite[Lemma~6.16]{AJL}). Equivalently, each $\f\in\cET$ can be written as the limit in $\cET$ of a sequence $\f_j\in\cH^T$ such that $\ent_v(\f_j)\to\ent_v(\f)$ and (hence) $\mab_{v,w}(\f_j)\to\mab_{v,w}(\f)$. However, we prefer to avoid relying on this fact, which is based on Yau's theorem~\cite{Yau} and is not available in the singular case (see~\S\ref{sec:normal}). 
\end{rmk}

\begin{lem}\label{lem:indep} The weighted Mabuchi energy $\mab_{v,w}\colon\cET\to\R\cup\{+\infty\}$ only depends on the choice of reference volume form $\nu_X$ by an overall additive constant.
\end{lem}
\begin{proof} Pick another $T$-invariant volume form $\tilde\nu_X$, and denote by 
$$
\tilde\ent_v(\f)=\tfrac 12\Ent(\MA_v(\f)|\tilde\nu_X),\quad\tilde\enR_v(\f):=\en_v^{-\Ric^T(\tilde\nu_X)}(\f)
$$
the associated entropy and Ricci energy functionals. We need to show that $\tilde\ent_v+\tilde\enR_v$ and $\ent_v+\enR_v$ coincide on $\cET$ up to an additive constant. Writing 
$\tilde\nu_X=e^{2\rho}\nu_X$ with $\rho\in C^\infty(X)^T$ we have
$$
-\Ric^T(\tilde\nu_X)=-\Ric^T(\nu_X)+\ddcT\rho. 
$$
By Example~\ref{exam:twEL}, we have $\en_v^{\ddcT\rho}(\f)=\int_X\rho\MA_v(\f)+c$ for an overall constant $c\in\R$, and hence 
$$
\tilde\enR_v(\f)=\enR_v(\f)+\int_X\rho\MA_v(\f)+c
$$
On the other hand, \eqref{equ:entexp} yields
$$
\tilde\ent_v(\f)=\ent_v(\f)-\int_X\rho\MA_v(\f),  
$$
and the result follows. 
\end{proof}

\begin{lem}\label{lem:quasimab} The weighted Mabuchi energy $\mab_{v,w}$ is quasi-invariant with respect to the action of $\Aut_0^T(X)$ on $\cET$. 
\end{lem}
As a consequence, $\mab_{v,w}$ is invariant under a subgroup $G\subset\Aut_0^T(X)$ iff $\mab_{v,w}(\f^g)=\mab_{v,w}(\f)$ for \emph{some} $\f\in\cET$ and all $g\in G$. 

\begin{proof}
By~\eqref{equ:quasien}, the weighted Monge--Amp\`ere energy $\en_{vw}$ is quasi-invariant on $\cH^T$, and hence on $\cET$, by continuity. As a result, it suffices to show the result for $\mab_{v,0}=\ent_v+\enR_v$. For all $\f\in\cET$, \eqref{equ:invMA} yields
$$
\MA_v(\f^g)=g^\star\MA_v(\f),
$$
and hence 
$$
\ent_v(\f^g)=\tfrac 12\Ent(g^\star\MA_v(\f)|\nu_X)=\tfrac 12\Ent(\MA_v(\f)|g_\star\nu_X). 
$$
On the other hand, \eqref{equ:quasitwen} yields 
$$
\enR_v(\f^g)=\en_v^{-g_\star\Ric^T(\nu_X)}(\f)=\en_v^{-\Ric^T(g_\star\nu_X)}(\f). 
$$
It follows that $\f\mapsto\ent_v(\f^g)+\enR_v(\f^g)$ coincides with the weighted Mabuchi energy $\tilde\mab_{v,0}$ computed with respect to the volume form $\tilde\nu_X:=g_\star\nu_X$. The result now follows from Lemma~\ref{lem:indep}. 
\end{proof}

For later use, we also note:

\begin{lem}\label{lem:entvgr} For all $\f\in\cET$ we have 
$$
(\inf_P v)\ent(\f)-C\le\ent_v(\f)\le(\sup_P v)\ent(\f)+C
$$
for a constant $C>0$ only depending on $\sup_P|v|$, $\int_X\nu_X$ and $\int_X\om^n$. 
\end{lem}
Here $\ent(\f)=\tfrac12\Ent(\MA(\f)|\nu_X)$ is the (unweighted) entropy. 

\begin{proof} For any $\f\in\cH^T$, the positive measure
$\MA_v(\f)=v(m_{\Om_\f})\MA(\f)$ satisfies
$$
(\inf_P v)\MA(\f)\le\MA_v(\f)\le(\sup_P v)\MA(\f).
$$
By continuity of $\MA$ and $\MA_v$, this remains true for $\f\in\cET$, and the result now follows from Lemma~\ref{lem:entcst}. 
\end{proof}

Finally, generalizing the fundamental result of Berman--Berndtsson~\cite{BerBer} and its finite energy extension in~\cite{BDL17}, it was shown in~\cite[Theorem~5]{Lah} and~\cite[Theorem~6.1]{AJL} that $\mab_{v,w}$ is convex along all psh geodesics in $\cET$.

\subsubsection{Relative weighted Mabuchi energy}
Assuming also $w>0$ on $P$, we define the \emph{relative weighted Mabuchi energy}
$$
\mab^\rel_{v,w}\colon\cET\to\R\cup\{+\infty\}
$$
as $\mab^\rel_{v,w}:=\mab_{v,w\ell^\ext}$ (see~\S\ref{sec:relmab}). The above results imply:
\begin{prop}\label{prop:relmabE1} The relative weighted Mabuchi energy $\mab^\rel_{v,w}$ is translation invariant, quasi-invariant under the action of $\Aut_0^T(X)$, and invariant under that of $T_\C\subset\Aut_0^T(X)$. 

Furthermore, it is convex along all psh geodesics in $\cET$. 
\end{prop}
\begin{proof} By construction of $\ell^\ext$, the restriction of $\mab_{v,w}$ to $\cH^T$ is invariant under translation and the action of $T_\C$. Since $\mab_{v,w}$ is quasi-invariant under $\Aut_0^T(X)$ on $\cET$, this already implies that $\mab_{v,w}$ is $T_\C$-invariant on $\cET$. 

Since $\ent_v(\f)$ only depends on $\MA_v(\f)$, it is translation invariant on $\cET$, and we thus need to show that $\enR_v+\en_{vw\ell^\ext}$ is translation invariant on $\cET$. By continuity, it suffices to see this on $\cH^T$, which holds by translation invariance of $\mab^\rel_{v,w}$ on $\cH^T$. 
\end{proof}

We finally recall the following results, respectively proved in~\cite[Theorem~1]{AJL}, \cite[Theorem~3.1]{He}, \cite[Theorem~1.2]{HL24} and~\cite{DJL24a,DJL24b} (building on the fundamental work of Chen--Cheng~\cite{CC1,CC2}). 

\begin{thm}\label{thm:proper} Assume that the torus $T\subset\Autr(X)$ is maximal. Then: 
\begin{itemize}
\item[(i)] the existence of a $(v,w)$-extremal K\"ahler metric in $\{\om\}$ implies that $\mab^\rel_{v,w}$ is coercive modulo $T_\C$; 
\item[(ii)] if $v$ is further log-concave on $P$, the converse holds as well. 
\end{itemize}
\end{thm}

%
%
\subsection{Extension to K\"ahler spaces with klt singularities}\label{sec:normal}
From now on we assume that $X$ is a normal compact K\"ahler space, and explain how to adapt the above discussion to that setting. We refer to~\cite{SMF} for the basic properties of differential calculus and psh functions on complex spaces. 
%
%
\subsubsection{Automorphisms and vector fields}
 
The automorphism group $\Aut(X)$ is a complex Lie group, with Lie algebra the space $\Hnot(X,T_X)$ of holomorphic vector fields, \ie global sections of the tangent sheaf $T_X=\mathrm{Hom}(\Om^1_X,\cO_X)$. Since $X$ is normal, restriction to $X_\regu$ induces an isomorphism
$$
\Hnot(X,T_X)\simeq\Hnot(X_\regu,T_{X_\regu}). 
$$
The flow of each $\xi\in\Hnot(X,T_X)$ defines a global holomorphic action of $(\C,+)$ on $X$. For any differential form $\om$ on $X$, the Lie derivatives $\cL_\xi\om$, $\cL_{J\xi}\om$ can then be defined as differential forms on $X$ by differentiating along the real and imaginary directions of the flow of $\xi$. While the interior derivative $i(\xi)$ is a priori only defined on $X_\regu$, we similarly have: 

\begin{lem} For any $\xi\in\Hnot(X,T_X)$ and any (smooth) differential form $\om$ on $X$, $i(\xi)\om$ is a (smooth) differential form on $X$. 
\end{lem}
Note that the usual Cartan formula
$$
\cL_\xi\om=d i(\xi)\om+i(\xi) d\om
$$
then holds globally on $X$. 
\begin{proof} By definition, $\om$ is locally given as the restriction of an ambient form in a local embedding of $X$ into a smooth space, and hence is locally a finite sum of decomposable forms, of the form $f_0 df_1\winter df_p$ for some smooth functions $f_0,f_1,\dots,f_p$. Since $i(\xi)$ is an antiderivation, we are then reduced to the case $\om=df$ for a smooth function $f$, in which case $i(\xi)\om=\cL_\xi f$ is indeed smooth, as discussed above. 
\end{proof}

Denote by $\cZ=\cZ(X)$ (resp.~$\cZ'=\cZ'(X)$) the space of locally $\ddc$-exact real $(1,1)$-forms (resp.~currents) on $X$. Any current $\theta\in\cZ'$ can be written as 
$$
\theta=\om+\ddc f
$$
with $\om\in\cZ$ and $f$ a distribution, which is necessarily smooth when $\theta$ is (see \eg\cite[\S 4.6.1]{BG}). The group $\Aut(X)$ acts on $\cZ$ and $\cZ'$. 

\begin{lem}\label{lem:cocycle} For any $g\in\Aut_0(X)$ and $\om\in\cZ$, there exists $\tau\in C^\infty(X)$ such that $g^\star\om=\om+\ddc\tau$. 
\end{lem}
\begin{proof} Pick an $\Aut_0(X)$-equivariant resolution of singularities $\pi\colon Y\to X$ with $Y$ K\"ahler. Then $g^\star\om$ and $\om$ lie in the same (de Rham) cohomology class on $Y$, and the $\ddc$-lemma thus yields $\pi^\star g^\star\om=\pi^\star\om+\ddc\tilde\tau$ with $\tilde\tau\in C^\infty(Y)$. The function $\tilde\tau$ is then pluriharmonic, and hence constant, along the (connected) fibers of $\pi$. It thus descends to a continuous function $\tau\in C^0(X)$ such that $g^\star\om=\om+\ddc\tau$. As recalled above, since $\ddc\tau$ is smooth, $\tau$ is smooth as well, and we are done. 
\end{proof}

We define the \emph{linear automorphism group} 
$$
\Autr(X)\subset\Aut_0(X)
$$
as the identity component of the subgroup of $\Aut_0(X)$ that acts trivially on the Albanese torus of some $\Aut_0(X)$-equivariant resolution of singularities of $X$. This is independent of the choice of resolution, by bimeromorphic invariance of the Albanese torus.

%
\subsubsection{Moment maps} In what follows we fix a compact torus 
$$
T\subset\Autr(X),
$$
with Lie algebra $\ft$. 
\begin{defi} A \emph{moment map} for a $T$-invariant form $\om\in\cZ$ (resp.~current $\om\in\cZ'$) is defined as a $T$-invariant $\ft^\vee$-valued smooth map (resp.~distribution) $m$ on $X$ such that $-dm^\xi=i(\xi)\om$ for all $\xi\in\ft$, with $m^\xi:=\langle m,\xi\rangle$. 
\end{defi}
We call $\Om=(\om,m)$ an \emph{equivariant form} (resp.~\emph{equivariant current}), with moment map $m_\Om:=m$. For any $T$-invariant function/distribution $f$, setting 
$$
m_f^\xi:=d^c f(\xi)=-df(J\xi)
$$
for $\xi\in\ft$ defines a moment map for $dd^c f$, and hence an equivariant form/current 
$$
\ddcT f=(\ddc f,m_f).
$$
More generally: 

\begin{lem} Any $T$-invariant form $\om\in\cZ$ (resp.~current $\om\in\cZ'$) admits a moment map $m$, unique up to an additive constant in $\ft^\vee$. 
\end{lem}
\begin{proof} Assume first that $\om$ is smooth, and pick a $T$-equivariant resolution of singularities $\pi\colon Y\to X$. Then $\pi^\star\om$ admits a smooth moment map on $ Y$. It induces a smooth moment map $m$ for $\om$ on $X_\reg$, and we need to show that $m$ extends to a smooth map on $X$. Any point of $X$ admits an open neighborhood $U$ on which $\om=\ddc\f$ for some $\f\in C^\infty(U)$. For any $\xi\in\ft$, \eqref{equ:dcm} yields 
$$
-\ddc m^\xi=\cL_{J\xi}\om=\ddc\cL_{J\xi}\f
$$
on $U_\reg$. Thus $h:=m^\xi+\cL_{J\xi}\f$ is pluriharmonic on $U_\reg$, and locally bounded near each point of $U$ since $m$ extends to a smooth function on $ Y$ and $\cL_{J\xi}\f$ is smooth on $U$. As a result, $h$ uniquely extends to a pluriharmonic (and hence smooth) function on $U$, and it follows that $m$ is smooth on $U$. 

Assume now $\om$ is a current. Write $\om=\om'+\ddc f$ with $\om'\in\cZ^T$ and $f$ a $T$-invariant distribution. By the previous step, $\om'$ admits a (smooth) moment map $m'$, and $m:=m'+m_f$ is thus a moment map for $\om$. 
\end{proof}
%
\subsubsection{Weighted Monge--Amp\`ere operators and weighted energy} 
With the above preliminaries in our hands, \S\ref{sec:wen} extends to the present setting without change. Given an equivariant form $\Om=(\om,m_\Om)$, an equivariant current $\Theta=(\theta,m_\Theta)$ and a smooth weight $v\in C^\infty(\ft^\vee)$, we set 
$$
\Om_\f:=\Om+\ddcT\f=(\om_\f,m_{\Om}+m_\f)
$$
for $\f\in C^\infty(X)^T$, and introduce the weighted Monge--Amp\`ere operator 
$$
\MA_{\Om,v}(\f):=v(m_{\Om_\f})\om_\f^n
$$
and twisted weighted Monge--Amp\`ere operator
$$
\MA_{\Om,v}^\Theta(\f):=v(m_{\Om_\f}) n\theta\wedge\om_\f^{n-1}+\langle v'(m_{\Om_\f}),m_\Theta\rangle\om_\f^n.
$$
The proof of Lemma~\ref{lem:symm} goes through and yields the key symmetry property
$$
\int_X g\MA_{\Om,v}^{\ddcT f}(\f)=\int_X f\MA_{\Om,v}^{\ddcT g}(\f)
$$
for all $\f\in C^\infty(X)^T$ and $T$-invariant distributions $f,g$, at least one of which is smooth. This allows us to define the weighted energy functionals
$$
\en_{\Om,v}\colon C^\infty(X)^T\to\R,\quad\en_{\Om,v}^\Theta\colon C^\infty(X)^T\to\R
$$
as Euler--Lagrange functionals of the weighted Monge--Amp\`ere operators.
%
\smallskip

For later use, we finally show: 
\begin{lem}\label{lem:LP} Let $\pi\colon Y\to X$ be a $T$-equivariant resolution of singularities, $D$ a $T$-invariant divisor on $Y$, with equivariant integration current $[D]_T$, and assume that $D$ is $\pi$-exceptional. For any $\f\in C^\infty(X)^T$ we then have  
$$
\en_{\pi^\star\Om,v}^{[D]_T}(\pi^\star\f)=0. 
$$
\end{lem}
\begin{proof} Pick $f\in C^\infty(X)^T$. Then 
\begin{align*}
\frac{d}{dt}\bigg|_{t=0}\en_{\pi^\star\Om,v}^{[D]_T}(\pi^\star(\f+t f)) & =\int_Y \pi^\star f
\MA_{\pi^\star\Om,v}^{[D]_T}(\pi^\star\f)\\
& =\int_Y \pi^\star f\left[n v(\pi^\star m_{\Om_\f})[D]\wedge\pi^\star\om_\f^{n-1}+\sum_\a v'_\a(m_{\pi^\star\Om_\f}) m_{\log|s_D|}^{\xi_\a}\pi^\star\om_\f^n\right],
\end{align*}
where $(\xi_\a)$ is a given basis of $\ft$ and $v'_\a$ denotes the partial derivatives in the dual basis of $\ft^\vee$. 
Since $D$ is $\pi$-exceptional, it satisfies $\pi_\star[D]=0$, and hence 
$$
\int_Y \pi^\star f\left[n v(\pi^\star m_{\Om_\f})[D]\wedge\pi^\star\om_\f^{n-1}\right]=0. 
$$
For any $\xi\in\ft$, \eqref{equ:dcm} further yields 
$\ddc m^\xi_{\log|s_D|}=-\cL_{J\xi}[D]$, and hence
$$
\ddc\pi_\star m^{\xi}_{\log|s_D|}=-\cL_{J\xi}\pi_\star[D]=0. 
$$
This shows that the distribution $\pi_\star m^\xi_{\log|s_D|}$ is constant, necessarily equal to $0$ since $m_{\log|s_D|}^\xi$ vanishes outside $D$ by~\eqref{equ:mphi}. We thus get
$$
\int_Y \pi^\star f\sum_\a v'_\a(m_{\pi^\star\Om_\f}) m_{\log|s_D|}^{\xi_\a}\pi^\star\om_\f^n=0, 
$$
and hence $\frac{d}{dt}\big|_{t=0}\en_{\pi^\star\Om,v}^{[D]_T}(\pi^\star(\f+t f))=0$. The result follows since $\en_{\pi^\star\Om,v}^{[D]_T}(0)=0$. 
\end{proof}

\subsubsection{Extension to $\cE^1$}
From now on, we fix an equivariant K\"ahler form $\Om=(\om,m_\Om)$ on $X$, and denote by  
$$
\cH=\cH(X,\om):=\{\f\in C^\infty(X)\mid\om_\f>0\}
$$
the space of K\"ahler potentials. The following key regularization theorem is a consequence of the recent results of Cho--Choi~\cite{CC}. 
\begin{thm} Any $\om$-psh function $\f$ on $X$ can be written as the pointwise limit of a decreasing sequence in $\cH$. If $\f$ is further $T$-invariant, then the sequence can be chosen in $\cH^T$.  
\end{thm}
\begin{proof} By~\cite[Theorem~5.5]{CC}, any $\om$-psh function on $X$ can be written as the limit of a decreasing sequence of (finite-valued) continuous $\om$-psh functions. On the other hand, the classical Richberg regularization theorem~\cite{Ric} (which is valid on any complex space) implies that any continuous $\om$-psh function is the uniform limit of functions in $\cH$. The first point follows, and the second point as well, by averaging with respect to the Haar measure of $T$. 
\end{proof}
In the present setting, the usual properties of the space $\cE^1=\cE^1(X,\om)$ of potentials of finite energy remain valid. In particular, $\cE^1$ is the metric completion of $\cH$ with respect the Darvas metric $\dd_1$. Lemma~\ref{lem:cocycle} further allows us to define as in~\S\ref{sec:auto} an isometric right-action of $\Aut_0(X)$ on $\cET$ 

\begin{rmk}\label{rmk:proper} While this is certainly expected to hold, the action of $\Aut_0(X)$ on $\cET$ is not known to be proper when $X$ is singular. However, this ingredient does not enter the proof of Theorem~\ref{thm:opencoermab} below, where properness is only needed for K\"ahler forms on a resolution of singularities, which holds by Proposition~\ref{prop:isom}. 
\end{rmk}

The estimates of \S\ref{sec:mabext} apply without change, and yield a continuous extension of the weighted Monge--Amp\`ere operator $\MA_v=\MA_{\Om,v}$ to $\cET$, as well as continuous extensions
$$
\en_v\colon\cET\to\R,\quad\en_v^\Theta\colon\cET\to\R
$$
of the weighted energy functionals, for any equivariant form $\Theta$.

%
%
\subsubsection{Adapted measures vs.~volume forms} 

We further assume from now on that $X$ has \emph{log terminal singularities}. One main difference in the present context is that (smooth, positive) volume forms on $X$ correspond to \emph{singular} metrics on the $\Q$-line bundle $K_X$. On the other hand, any smooth metric $\p$ on $K_X$ defines a positive Radon measure $\nu$ on $X$, locally given by 
$$
\nu:=\frac{\left(i^{m n^2}\sigma\wedge\bar\sigma\right)^{2/m}}{|\sigma|_\p^{2/m}}
$$
for any choice of local trivialization $\sigma$ of the line bundle $mK_X$ with $m\in\Z_{>0}$ sufficiently divisible. We call $\nu$ an \emph{adapted measure}, and denote by $\tfrac12\log\nu:=\p$ the corresponding smooth metric on $K_X$. Any two adapted measures differ by a smooth positive density. 

We define the \emph{Ricci form} of an adapted measure $\nu$ as the curvature form 
$$
\Ric(\nu):=-\ddc\tfrac 12\log\nu, 
$$
and its \emph{equivariant Ricci form} as the equivariant curvature form 
$$
\Ric^T(\nu):=-\ddcT\tfrac 12\log\nu. 
$$

\begin{lem}\label{lem:adapted} Pick a (smooth, positive) volume form $\mu$ on $X$, and an adapted measure $\nu$. Then $\mu=e^{\rho}\nu$ for a quasi-psh function $\rho$ with analytic singularities, \ie locally of the form 
$$
\rho=c\log \max_\a|f_\a|+O(1)
$$
for a finite set $f_\a\in\cO_X$ and $c\in\Q_{>0}$. 
\end{lem}
\begin{proof} Any two volume forms on $X$ differ by a smooth positive conformal factor. Arguing locally, we may thus assume $\mu=\om^n$ with $\om=\sum_{j=1}^N idz_j\wedge d\bar z_j$ for some local embedding of $X$ into $\C^N$, and $\nu=|\sigma|^{2/m}$ for a local trivialization $\sigma$ of $mK_X$. Then $\mu=\sum_I i^{n^2}\sigma_I\wedge\overline{\sigma_I}$ where $I$ ranges over all subsets of $\{1,\dots,N\}$ of cardinal $n$ and 
$\sigma_I=\bigwedge_{i\in I} dz_i|_X$. For each $I$ we have $\sigma_I^m=f_I\sigma$ with $f_I\in\cO_X$, and hence $\mu=(\sum_I |f_I|^{2/m})\nu$. The result follows. 
\end{proof}

Any $T$-invariant volume form $\mu$ thus induces a $T$-invariant quasi-psh metric $\tfrac 12\log \mu$ on $K_X$, which is smooth on $X_\regu$ and has analytic singularities along $X_{\sing}$. We define the \emph{Ricci current} of $\mu$ as the curvature current
$$
\Ric(\mu):=-\ddc\tfrac 12\log\mu, 
$$
and its \emph{equivariant Ricci current} as the equivariant curvature current 
$$
\Ric^T(\mu):=-\ddcT\tfrac 12\log\mu. 
$$
Both $\Ric(\mu)$ and $\Ric^T(\mu)$ are smooth on $X_\regu$. Since $\tfrac 12\log\mu$ is quasi-psh, $\Ric(\mu)$ is bounded above on $X$, but it is bounded below only when $X$ is smooth, \cf~\cite[Theorem~1.3]{Li}. 

%
\subsubsection{Weighted scalar curvature and weighted Mabuchi energy}
Applying Lemma~\ref{lem:poly} to the pullback of the equivariant K\"ahler form $\Om$ to an equivariant resolution of singularities shows that 
$$
P:=m_\Om(X)\subset\ft^\vee
$$ 
is a convex polytope, called the \emph{moment polytope} of $\Om$. 

Pick a weight $v\in C^\infty(\ft^\vee)$, positive on $P$. For any $\f\in\cH^T$, we define the \emph{$v$-weighted equivariant Ricci current} of the equivariant K\"ahler form $\Om_\f$ as
$$
\Ric_v^T(\Om_\f):=\Ric^T(\MA_v(\f)),
$$
and the \emph{$v$-weighted scalar curvature} as the distribution
$$
S_v(\Om_\f):=\tr_{\Om_\f,v}(\Ric_v^T(\Om_\f))=\frac{\MA_v^{\Ric_v^T(\Om_\f)}(\f)}{\MA_v(\f)}.
$$
Given a reference $T$-invariant adapted measure $\nu_X$,  we define the \emph{weighted entropy} and \emph{weighted Ricci energy} 
$$
\ent_v\colon\cET\to\R\cup\{+\infty\},\quad\enR_v\colon\cET\to\R
$$
by 
$$
\ent_v(\f):=\tfrac12\Ent\left(\MA_v(\f)|\nu_X\right),\quad\enR_v(\f):=\en_v^{-\Ric_v^T(\nu_X)}(\f). 
$$
The entropy $\ent_v$ is lsc, while $\enR_v$ is continuous. 

Given any other weight $w\in C^\infty(\ft^\vee)$, the \emph{weighted Mabuchi energy} 
$$
\mab_{v,w}\colon\cET\to\R\cup\{+\infty\}
$$ 
is defined by 
$$
\mab_{v,w}(\f):=\ent_v(\f)+\enR_v(\f)+\en_{vw}(\f).
$$

The computation of Lemma~\ref{lem:mabscal} remains valid in our context, and shows that the restriction of $\mab_{v,w}$ to $\cH^T$ is an Euler--Lagrange functional for the operator
$$
\cH^T\ni\f\mapsto\left(w(m_{\Om_\f})-S_v(\Om_\f)\right)\MA_v(\f).
$$
The proofs of Lemma~\ref{lem:indep} and Lemma~\ref{lem:quasimab} also apply, and show that $\mab_{v,w}$ only depends on $\nu_X$ by an overall additive constant, and is quasi-invariant with respect to the action of $\Aut_0^T(X)$ on $\cET$. 


Assuming also $w>0$ on $P$, the \emph{extremal affine function} $\ell^\ext=\ell^\ext_{\Om,v,w}$ is defined as in~\S\ref{sec:relmab} as the unique affine function on $\ft$ such that 
$$
\int_X\ell(m_\Om)\ell^\ext(m_\Om)w(m_\Om)\MA_v(0)=\int_X\ell(m_\Om)S_v(\Om)\MA_v(0)
$$
for all affine functions $\ell\in\ft\oplus\R$. It is characterized as the unique affine function such that the \emph{relative weighted Mabuchi energy} 
$$
\mab^\rel_{v,w}:=\mab_{v,w\ell^\ext}\colon\cET\to\R\cup\{+\infty\}
$$
is invariant under translation and the action of $T_\C$, \cf Proposition~\ref{prop:relmabE1}. 
%
%
%
%
\section{Openness of coercivity on resolutions of singularities}\label{sec:main}
The goal of this section is to establish the main result of this paper, \ie that coercivity of the relative weighted Mabuchi energy is preserved under a large class of resolutions of singularities. To this end, we will use the general criterion in~\S\ref{sec:opencoer}.

%
%
%
\subsection{Setup and main result}\label{sec:setup}
In what follows $X$ is a compact K\"ahler space with log terminal singularities, $T\subset\Autr(X)$ is a compact torus, and $\pi\colon Y\to X$ is a $T$-equivariant \emph{resolution of singularities}, \ie a bimeromorphic morphism with $Y$ a compact K\"ahler manifold. 

\subsubsection{Setup on $X$} Fix a $T$-equivariant K\"ahler form $\Om_X=(\om_X,m_X)$ on $X$, a $T$-invariant adapted measure $\nu_X$, and set 
$$
\cH^T_X:=\cH(X,\om_X)^T,\quad\cET_X:=\cE^1(X,\om_X)^T.
$$
Pick also smooth functions $v,w\colon\ft^\vee\to\R$ such that $v,w>0$ on the moment polytope 
$$
P_X:=m_X(X)\subset\ft^\vee, 
$$
and denote by 
$$
\mab^\rel_X=\mab^\rel_{\Om_X,v,w}\colon\cET_X\to\R\cup\{+\infty\}
$$
the associated relative weighted Mabuchi energy. For each $\p\in\cET_X$ we thus have 
\begin{equation}\label{equ:mabX}
\mab^\rel_X(\p)=\ent_{X,v}(\p)+\enR_{X,v}(\p)+\en_{X,v w\ell_X}(\p), 
\end{equation}
where 
$$
\ent_{X,v}(\p):=\tfrac12\Ent\left(\MA_{X,v}(\p)|\nu_X\right)\quad\text{with}\quad\MA_{X,v}(\p):=\MA_{\Om_X,v}(\p), 
$$
$$
\enR_{X,v}(\p):=\en_{\Om_X,v}^{-\Ric^T(\nu_X)}(\p),\quad
\en_{X,v w\ell_X}(\p)=\en_{\Om_X,v w\ell_X}(\p)
\quad\text{with}\quad\ell_X:=\ell^\ext_{\Om_X,v,w},
$$
 see~\S\ref{sec:relmab}. 

\subsubsection{Setup on $Y$} We fix a $T$-invariant volume form $\nu_Y$ on $Y$, and suppose given a sequence of equivariant K\"ahler forms $\Om_j=(\om_j,m_j)$ on $Y$ such that
\begin{itemize}
\item[(i)] $\Om_j\to\pi^\star\Om_X$ smoothly, \ie $\om_j\to\pi^\star\om_X$ and $m_j\to\pi^\star m_X$; 
\item[(ii)] $\om_j\ge (1-\e_j)\pi^\star\om_X$ with $[0,1)\ni\e_j\to 0$. 
\end{itemize}

The simplest example is $\Om_j=\pi^\star\Om_X+\e_j\Om_Y$ with $\e_j\to 0_+$, for any choice of equivariant K\"ahler form $\Om_Y$ on $Y$. As another case of interest, we have: 

\begin{exam}\label{exam:blowup1} Assume given an effective divisor $E$ on $Y$ such that $-E$ is $\pi$-ample (which exists \eg if $\pi$ is a sequence of blowups with smooth centers). Then one can find a smooth $T$-invariant metric on $\cO(E)$ with equivariant curvature form $\Theta_E=(\theta_E,m_E)$ such that $\om_Y:=\pi^\star\om_X-\e\theta_E$ is positive for $0<\e\ll 1$. For any sequence $\e_j\in(0,\e)$ converging to $0$, $\Om_j:=\pi^\star\Om_X-\e_j\Theta_E$ is then positive for all $j$, and 
$$
\om_j=(1-\e_j/\e)\pi^\star\om_X+(\e_j/\e) \om_Y\ge (1-\e_j/\e)\pi^\star\om_X. 
$$
\end{exam}
Returning to the general setup above, the smooth convergence $m_j\to\pi^\star m_X$ implies that the moment polytope 
$$
P_j:=m_j(Y)\subset\ft^\vee
$$ 
of $\Om_j$ converges to $(\pi^\star m_X)(Y)=m_X(X)=P_X$, and hence that $v,w>0$ on $P_j$ for all $j\gg1$. We henceforth assume without loss that this holds for all $j$. We set
$$
\cH^T_j:=\cH(Y,\om_j)^T,\quad\cET_j:=\cE^1(Y,\om_j)^T, 
$$
and denote by 
$$
\mab^\rel_j=\mab^\rel_{\Om_j,v,w}\colon\cET_j\to\R\cup\{+\infty\}
$$
the relative weighted Mabuchi energy. For any $\f\in\cET_j$, we have 
\begin{equation}\label{equ:mabj}
\mab^\rel_j(\f)=\ent_{j,v}(\f)+\enR_{j,v}(\f)+\en_{j, v w\ell_j}(\f),
\end{equation}
where 
$$
\ent_{j,v}(\f):=\tfrac12\Ent\left(\MA_{j,v}(\f)|\nu_Y\right)\quad\text{with}\quad\MA_{j,v}(\f):=\MA_{\Om_j,v}(\f), 
$$
$$
\enR_{j,v}(\f):=\en_{\Om_j,v}^{-\Ric^T(\nu_Y)}(\f),\quad\en_{j,v w\ell_j}(\f):=\en_{\Om_j,v w\ell_j}(\f)
\quad\text{with}\quad\ell_j:=\ell^\ext_{\Om_j,v,w}.
$$
\subsubsection{Openness of coercivity on resolutions of Fano type}

In order to state our main result we introduce the following terminology. 

\begin{defi}\label{defi:pipos} We say that a closed $(1,1)$-current $\theta$ on $Y$ is \emph{$\pi$-semipositive} if $\theta+\pi^\star\a\ge 0$ for some (smooth) $(1,1)$-form $\a$ on $X$; equivalently, $\theta\ge -C\pi^\star\om_X$ for $C\gg 1$. 
\end{defi}

\begin{exam}\label{exam:amppos} Any $\pi$-ample line bundle on $Y$ admits a smooth metric $\phi$ with $\ddc\phi$ $\pi$-semipositive. 
\end{exam}

\begin{defi}\label{defi:Fanotype} We say that $\pi\colon Y\to X$ is \emph{of Fano type} if there exists a singular metric $\phi$ on $-K_Y$ such that
\begin{itemize}
\item the curvature current $\ddc\phi$ is $\pi$-semipositive; 
\item $\phi$ has trivial multiplier ideal.
\end{itemize}
\end{defi}
Equivalently, we require the existence of a quasi-psh function $f$ on $Y$ such that
\begin{itemize}
    \item the measure $\hnu_Y:=e^{-2f}\nu_Y$ on $Y$ has finite total mass; 
    \item the Ricci current $\Ric(\hnu_Y):=\Ric(\nu_Y)+\ddc f$ is $\pi$-semipositive. 
   \end{itemize}

The chosen terminology is justifies by Lemma~\ref{lem:logFano} below (compare for instance~\cite{Bir}). 

\begin{exam} If $X$ is smooth and $\pi\colon Y\to X$ is the blowup along a smooth submanifold $Z\subset X$, then $-K_Y$ is $\pi$-ample, and $\pi$ is thus of Fano type.
\end{exam}

\begin{lem}\label{lem:ft} Assume there exists an effective $\Q$-divisor $B$ on $Y$ such that 
\begin{itemize}
    \item[(i)] $-(K_Y+B)$ is $\pi$-nef; 
    \item[(ii)] the pair $(Y,B)$ is klt. 
\end{itemize}
Then $\pi$ is of Fano type. 
\end{lem}
\begin{proof} Set $L:=-(K_Y+B)$ and pick K\"ahler forms $\om_X,\om_Y$ on $X,Y$. For $C\gg 1$, $c_1(L)+C\pi^\star\{\om_X\}$ is big, and hence contains a K\"ahler current $T\ge\e\om_Y$ for some $\e>0$. Write $T-C\pi^\star\om_X=\ddc\phi_0$ for a quasi-psh metric $\phi_0$ on $L$. Since $L$ is $\pi$-nef, it also admits a smooth metric $\phi_1$ such that $\ddc\phi_1+\e\om_Y$ is $\pi$-semipositive. Denote by $\phi_B$ the canonical singular metric on $B$. For $t\in [0,1]$, $\phi_t:=\phi_B+(1-t)\phi_0+t\phi_1$ is a quasi-psh metric on $-K_Y$ with $\ddc\phi_t$ $\pi$-semipositive. Since $\phi_B$ has trivial multiplier ideal, so does $\phi_t$ for $0<t\ll 1$, by opennes. This shows, as desired, that $\pi$ is of Fano type. 
\end{proof}

In the projective case we conversely have: 
\begin{lem}\label{lem:logFano} Assume $X,Y$ are projective Then $\pi\colon Y\to X$ is of Fano type iff there exists an effective $\Q$-divisor $B$ on $Y$ such that
\begin{itemize}
    \item[(i)] $-(K_Y+B)$ is $\pi$-ample; 
    \item[(ii)] the pair $(Y,B)$ is klt. 
\end{itemize}
\end{lem}
\begin{proof} Pick an ample line bundle $A$ on $X$ and a smooth metric $\phi_A$ on $A$ with positive curvature $\ddc\phi_A>0$, and quasi-psh metric $\phi_L$ on $L:=-K_Y$ with trivial multiplier ideal and such that $\ddc\phi_L\ge-m\pi^\star\ddc\phi_A$ for some $m\gg 1$. Pick $C>0$ such that $L+C\pi^\star A$ is big, and $\Q$-divisors $H,E$ such that $L+C\pi^\star A=H+E$, $H$ ample $E$ effective. Pick a smooth metric $\phi_H$ with positive curvature, denote by $\phi_E$ the canonical singular metric of $E$, and let $\tilde\phi_L:=\phi_H+\phi_E-C\pi^\star\phi_A$ the induced metric on $L$. 

By Demailly's regularization theorem, for each $\e>0$ we can find a quasi-psh metric $\phi_\e$ on $L$ with analytic singularities that is less singular than $\phi_L$ (and hence with trivial multiplier ideal), and such that 
\begin{align*}
\ddc\phi_\e\ge-\ddc(m\pi^\star\phi_A+\e\phi_H) & =-m\ddc\pi^\star\phi_A-\e\ddc(\tilde\phi_L-\phi_E+C\pi^\star\phi_A)\\
& \ge-(m+\e C)\pi^\star\ddc\phi_A-\e\ddc\tilde\phi_L. 
\end{align*}
Then $\tilde\phi_\e:=(1+\e)^{-1}(\phi_\e+\e\tilde\phi_L)$ is a quasi-psh metric on $L$ with analytic singularities such that 
$$
\ddc\tilde\phi_\e\ge-(1+\e)^{-1}(m+\e C)\pi^\star\ddc\phi_A,
$$
and whose multiplier ideal is trivial for $0<\e\ll 1$, by the openness theorem~\cite{GuanZhou}. Since $\phi_\e$ has analytic singularities, it corresponds to an ideal sheaf on $Y$, and it is then easy to get the desired boundary divisor $B$ (cf.~\cite[Proposition I.9.2.26]{PAG}). 
\end{proof}

We do not know whether a resolution of Fano type always exists. However: 

\begin{exam} If $\dim X=2$, then the minimal resolution $\pi\colon Y\to X$ is of Fano type. This follows from Lemma~\ref{lem:ft}, since the unique $\pi$-exceptional $\Q$-divisor $B$ such that $K_Y+B=\pi^\star K_X$ is effective by the Negativity Lemma (because $-B\equiv_\pi K_Y$ is $\pi$-nef), and $(Y,B)$ is klt since it is crepant birational to $X$. 
\end{exam} 

\begin{exam} Assume $\dim X=3$ and $X$ has Gorenstein quotient singularities. Then $X$ admits a crepant resolution $\pi\colon Y\to X$ (see~\cite[Theorem~1.2]{BKR}), which is thus of Fano type. 
\end{exam}

As a consequence of Example~\ref{exam:amppos}, we also have:

\begin{exam}\label{exam:Fanocone} Assume that $X$ has isolated singularities, each locally isomorphic to an affine cone over a Fano manifold. Then $X$ has log terminal singularities, and blowing up the singularities yields a resolution of singularities $\pi\colon Y\to X$ of Fano type. 
\end{exam}

We may now state the main results of this article: 

\begin{thm}\label{thm:opencoermab} Let $X$ be compact K\"ahler space with klt singularities, $T\subset\Autr(X)$ a compact torus, and pick smooth weights $v,w\in C^\infty(\ft^\vee)$ such that $v,w>0$ on the moment polytope $P_X$. Let further $\pi\colon Y\to X$ be a $T$-equivariant resolution of singularities of Fano type. 
\begin{itemize}
\item[(i)] If $\mab^\rel_X$ is coercive modulo $T_\C$ (see Definition~\ref{defi:coerG}), then so is $\mab^\rel_j$ for all $j$ large enough.
\item[(ii)] More generally, suppose given $\sigma,C\in\R$ such that 
$$
\mab^\rel_X\geq \sigma\, \dd_{1,T}(\cdot,0)+C\quad\text{on}\quad\cET_{X,0}.
$$
Then for any $\sigma'<\sigma$ there exists $C'\in\R$ such that 
$$
\mab^\rel_j\geq \sigma' \, d_{j,T}(\cdot,0)+C'\quad\text{on}\quad\cET_{j,0}
$$ 
for all $j$ large enough.
\end{itemize}
\end{thm}
Here we have set
$$
\cET_{X,0}:=\{\p\in \cET_X\: |\: \en(\p)=0\},\quad\cET_{j,0}=\{\f\in \cET_j\: |\, \en_j(\f)=0 \}.
$$
Theorem~\ref{thm:opencoermab} will be proved in~\S\ref{sec:proofmain} below. Combined with Theorem~\ref{thm:proper}, it implies: 

\begin{cor}\label{cor:openextr} Assume further that
\begin{itemize}
\item $X$ is smooth;
\item the torus $T$ is maximal in $\Autr(X)$;
\item $v$ is log-concave on $P$.
\end{itemize}
If the class $\{\om_X\}$ admits a $(v,w)$-extremal Kähler metric, then the K\"ahler class $\{\om_j\}$ contains an $(v,w)$-extremal K\"ahler metric for all $j$ large enough.
\end{cor}
Specializing to usual extremal Kähler metrics (\ie $(v,w)=(1,1)$), this extends~\cite[Corollary~2]{SSz20}, which considers the blowup $\pi\colon Y\to X$ along a submanifold of codimension at least three (see also~\cite{AP1,AP2,APS11,Hal} for blowups of points).

\begin{rmk}\label{rmk:NotMaxTorus} As in the above works, it is natural to ask for a version of Corollary~\ref{cor:openextr} when $T$ is not assumed to be maximal in $\Autr(X)$. By~\cite[Remark~7.7]{AJL}, $\mab^\rel_X$ should then be assumed to be coercive modulo (the identity component of) the centralizer $\Autr^T(X)$ of $T$ in $\Autr(X)$, and the conclusion should be (under appropriate assumptions) that $\mab^\rel_j$ is coercive modulo $\Autr^T(Y)$. However, under our approach relying on Theorem~\ref{thm:coeropen}, this seems only possible to achieve under the assumption that $\mab^\rel_X$ is already coercive modulo $\Autr^T(Y)\subset\Autr^T(X)$. By Lemma~\ref{lem:propersub}, the latter implies that $\Autr^T(X)/\Autr^T(Y)$ is compact, and hence trivial if for instance $\Autr^T(X)$ and $\Autr^T(Y)$ are both complex reductive, which holds if $X$ and $Y$ admit (weighted) extremal metrics by~\cite{Lah}. The conclusion is thus that $\pi$ should already be $\Autr^T(X)$-equivariant to begin with, and we are therefore not able to capture the more subtle conditions of the above works using our approach.
\end{rmk}

\begin{rmk}\label{rmk:sigmalsc2}
    As in Remark~\ref{rmk:sigmalsc1}, Theorem~\ref{thm:coeropen}~(ii) immediately implies the lower semicontinuity of coercivity thresholds.
\end{rmk}

%
%
%
%
\subsection{Smooth continuity of weighted functionals}
Until further notice we stick to the general setup of~\S\ref{sec:setup}, without yet assuming $\pi\colon Y\to X$ to be of Fano type. Since $\pi$ is an isomorphism outside Zariski closed subsets, any measure $\mu$ on $X$ that is nonpluripolar (\ie puts no mass on pluripolar sets) induces a nonpluripolar positive measure on $Y$ of same total mass, which we denote by $\pi^\star\mu$. As one easily checks, the map $\mu\mapsto\pi^\star\mu$ is continuous in the weak topology of measures. 

In particular, for any $\p\in\cET_X$ the nonpluripolar measure $\MA_{X,v}(\p)$ induces a nonpluripolar measure 
$$
\MA_{Y,v}(\pi^\star\p):=\pi^\star\MA_{X,v}(\p)
$$
on $Y$. When $\p\in\cH^T_X$ we simply have $\MA_{Y,v}(\pi^\star\p)=\MA_{\pi^\star\Om_X,v}(\pi^\star\p)$. By strong continuity of $\MA_{X,v}$ on $\cET_X$ (see Proposition~\ref{prop:MAvE1}), the map $\p\mapsto\MA_{Y,v}(\pi^\star\p)$ is continuous with respect to the strong topology of $\cET_X$ and the weak topology of measures on $Y$. 

Next denote by $D$ the unique $\pi$-exceptional $\Q$-divisor such that 
$$
K_Y=\pi^\star K_X+D, 
$$
and by $\phi_D$ the canonical singular metric on the $\Q$-line bundle $\cO_Y(D)=K_{Y/X}$ with curvature current $\ddc\phi_D=2\pi[D]$. Consider also the smooth metric 
$$
\phi_{X/Y}:=\tfrac 12\log\nu_Y-\pi^\star\tfrac 12\log\nu_X
$$
on $K_{Y/X}$ induced by the adapted measures $\nu_X$, $\nu_Y$. The function 
$$
\rho_{Y/X}:=\phi_D-\phi_{Y/X}\in L^1(Y)
$$
satisfies
\begin{equation}\label{equ:Jpi}
\pi^\star\nu_X=e^{2\rho_{Y/X}}\nu_Y. 
\end{equation}
After scaling the reference volume form $\nu_Y$, and hence the metric $\phi_{Y/X}$, we further assume without loss 
\begin{equation}\label{equ:normY}
\int_Y\rho_{Y/X}\,\MA_{Y,v}(0)=0. 
\end{equation}
Denote by 
$$
\Theta_{Y/X}:=\ddcT\phi_{Y/X}
$$
the equivariant curvature form of the smooth metric $\phi_{Y/X}$ on $K_{Y/X}$. Then 
\begin{equation}\label{equ:RicTh}
\Theta_{Y/X}=\pi^\star\Ric^T(\nu_X)-\Ric^T(\nu_Y) 
\end{equation}
and
\begin{equation}\label{equ:LPrho}
\Theta_{Y/X}+\ddcT\rho_{Y/X}=\ddcT\phi_D=2\pi[D]_T, 
\end{equation}
see~\eqref{equ:intequ}. For any $\p\in\cET_X$ we set
$$
\ent_{Y,v}(\pi^\star\p):=\tfrac12\Ent\left(\MA_{Y,v}(\pi^\star\p)|\nu_Y\right), 
$$
and 
$$
\enR_{Y,v}(\pi^\star\p):=\en_{\pi^\star\Om_X,v}^{-\Ric^T(\nu_Y)}(\pi^\star\p), 
$$
when $\p$ is further smooth (see~\S\ref{sec:wen}). Note that $\enR_{Y,v}$ cannot a priori be extended to $\pi^\star\cET_X$ (see however Lemma~\ref{lem:twenext} below, that will be put to use under the Fano type assumption to establish Theorem~\ref{thm:opencoermab}). 

\begin{lem}\label{lem:MY} For all $\p\in\cH^T_X$ we have 
$\ent_{X,v}(\p)+\enR_{X,v}(\p)=\ent_{Y,v}(\pi^\star\p)+\enR_{Y,v}(\pi^\star\p)$. 
\end{lem}

\begin{proof} On the one hand, \eqref{equ:RicTh} yields 
$\enR_{Y,v}(\pi^\star\p)=\enR_{X,v}(\p)+\en_{\pi^\star\Om_X,v}^{\Theta_{Y/X}}(\pi^\star\p)$. On the other hand, write $\MA_{X,v}(\p)=f\,\nu_X$ for a smooth positive density $f$, so that 
$$
\ent_{X,v}(\p)=\int_X \log f\,\MA_{X,v}(\p).
$$
By~\eqref{equ:Jpi} we get $\MA_{Y,v}(\pi^\star\p)=(\pi^\star f) e^{\rho_{Y/X}}\nu_Y$, and hence 
\begin{align*}
\ent_{Y,v}(\pi^\star\p) & =\int_Y\log\left((\pi^\star f )e^{\rho_{Y/X}}\right)\MA_{Y,v}(\pi^\star\p)\\
&=\ent_{X,v}(\p)+\int_Y\rho_{Y/X}\MA_{Y,v}(\pi^\star\p). 
\end{align*}
Thus
$$
\left(\ent_{Y,v}(\pi^\star\p)+\enR_{Y,v}(\pi^\star\p)\right)-\left(\ent_{X,v}(\p)+\enR_{X,v}(\p)\right) 
$$
$$
=\en_{\pi^\star\Om_X,v}^{\Theta_{Y/X}+\ddcT\rho_{Y/X}}(\pi^\star\p)=2\pi\en_{\pi^\star\Om_X,v}^{[D]_T}(\pi^\star\p), 
$$
by~\eqref{equ:LPrho}, which vanishes by Lemma~\ref{lem:LP} since $D$ is $\pi$-exceptional.  
\end{proof}

Our next goal is to show: 

\begin{prop}\label{prop:cvmab} For any sequence $\f_j\in\cH^T_j$ converging smoothly to $\pi^\star\p$ with $\p\in\cH^T_X$, we have $\mab^\rel_j(\f_j)\to\mab^\rel_X(\p)$ as $j\to\infty$. 
\end{prop}

\begin{lem}\label{lem:cvscal} Assume $f_j\in C^\infty(Y)$ converges smoothly to $\pi^\star f$ with $f\in C^\infty(X)$. Then 
$$
\int_Y f_j S_v(\Om_j)\MA_{j,v}(0)\to\int_X f S_v(\Om_X)\MA_{X,v}(0). 
$$
\end{lem}
\begin{proof} Set
$$
\rho_X:=\frac{1}{2}\log\left(\frac{\MA_{X,v}(0)}{\nu_X}\right),\quad\rho_j:=\frac{1}{2}\log\left(\frac{\MA_{j,v}(0)}{\nu_Y}\right). 
$$
Here $\rho_X$ is a quasi-psh function on $X$ (see Lemma~\ref{lem:adapted}), while $\rho_j$ is smooth on $Y$. Then
$$
\Ric_v^T(\Om_X)=\Ric^T(\nu_X)-\ddcT\rho_X,
$$
and hence 
\begin{align*}
S_v(\Om_X)\MA_{X,v}(0) & =\MA_{X,v}^{\Ric_v^T(\Om_X)}(0)\\
& =\MA_{X,v}^{\Ric^T(\nu_X)}(0)-\MA_{X,v}^{\ddcT\rho_X}(0).
\end{align*}
Similarly, 
\begin{align*}
S_v(\Om_j)\MA_{j,v}(0) & =\MA_{j,v}^{\Ric^T(\nu_Y)}(0)-\MA_{j,v}^{\ddcT\rho_j}(0)
\\
& =\MA_{j,v}^{\pi^\star\Ric^T(\nu_X)}(0)-\MA_{j,v}^{\Theta_{Y/X}+\ddcT\rho_j}(0), 
\end{align*}
by~\eqref{equ:RicTh}. On the one hand, $f_j\to\pi^\star f$ and $\Om_j\to\pi^\star\Om_X$ smoothly implies
$$
\int_Y f_j\MA_{j,v}^{\pi^\star\Ric^T(\nu_X)}(0)\to\int_Y\pi^\star f\MA_{\pi^\star\Om_X,v}^{\pi^\star\Ric^T(\nu_X)}(0)=\int_X f\MA_{X,v}^{\Ric^T(\nu_X)}(0)
$$
and
$$
\int_Y f_j\MA_{j,v}^{\Theta_{Y/X}}(0)\to\int_Y\pi^\star f\MA_{\pi^\star\Om_X,v}^{\Theta_{Y/X}}(0). 
$$
On the other hand, \eqref{equ:Jpi} yields $\rho_j\to\pi^\star\rho_X+\rho_{Y/X}$ pointwise on $Y$. Using $(1-\e_j)\pi^\star\om_X\le\om_j\le C\om_Y$, we further have $\pi^\star\rho_X+\rho_{Y/X}-o(1)\le\rho_j\le O(1)$, and hence $\rho_j\to\pi^\star\rho_X+\rho_{Y/X}$ in $L^1(Y)$, by dominated convergence. Using the symmetry property~\eqref{equ:twMAsymm} it follows that
$$
\int_Y f_j\MA_{j,v}^{\ddcT\rho_j}(0)=\int_Y \rho_j\MA_{j,v}^{\ddcT f_j}(0), 
$$
converges to 
\begin{align*}
\int_Y(\pi^\star\rho_X+\rho_{Y/X})\MA_{\pi^\star\Om_X,v}^{\ddcT\pi^\star f}(0) & =\int_Y \pi^\star f\MA_{\pi^\star\Om_X,v}^{\ddcT(\pi^\star\rho_X+\rho_{Y/X})}(0)\\
& =\int_X f\MA_{X,v}^{\ddcT\rho_X}(0)+\int_Y\pi^\star f\MA_{\pi^\star\Om_X,v}^{\ddcT\rho_{Y/X}}(0).
\end{align*}
Combining all this we get
$$
\int_Y f_j\MA_{j,v}^{\Theta_{Y/X}+\ddcT\rho_j}(0)\to\int_X f\MA_{X,v}^{\ddcT\rho_X}(0)+\int_Y \pi^\star f\MA_{\pi^\star\Om_X,v}^{\Theta_{Y/X}+\ddcT\rho_{Y/X}}(0)
$$
where the second integral vanishes by~\eqref{equ:LPrho} and Lemma~\ref{lem:LP}, since $D$ is $\pi$-exceptional.

We conclude, as desired,  
$$
\int_Y f_j S_v(\Om_j)\MA_{j,v}(0)\to\int_X f\MA_{X,v}^{\Ric^T(\nu_X)}(0)-\int_X f\MA_{X,v}^{\ddcT\rho_X}(0)=\int_X f S_v(\Om_X)\MA_{X,v}(0). 
$$

\end{proof}

\begin{lem}\label{lem:ellextcv} The extremal affine functions $\ell_j$ converge to $\ell_X$ in $\ft\oplus\R$. 
\end{lem} 
\begin{proof} By definition, $\ell_j\in\ft\oplus\R$ is the image of the linear form 
$$
[\ell\mapsto\int_Y\ell(m_j)S_v(\Om_j)\MA_{j,v}(0)]\in(\ft\oplus\R)^\star
$$
under the duality $(\ft\oplus\R)^\star\simto\ft\oplus\R$ induced by the weighted Futaki--Mabuchi pairing 
$$
\langle\ell,\ell'\rangle_j:=\int_Y\ell(m_j)\ell'(m_j)\MA_{j,vw}(0)
$$
Since $m_j\to\pi^\star m$ smoothly, $\langle\ell,\ell'\rangle_j$ converges to the weighted Futaki--Mabuchi pairing 
$$
\langle\ell,\ell'\rangle_X:=\int_X\ell(m_X)\ell'(m_X) \MA_{X,vw}(0),
$$
while Lemma~\ref{lem:cvscal} implies $\int_Y\ell(m_j)S_v(\Om_j)\MA_{j,v}(0)\to\int_X\ell(m_X)S_v(\Om_X)\MA_{X,v}(0)$. The result follows. 
\end{proof}

\begin{proof}[Proof of Proposition~\ref{prop:cvmab}] Assume $\f_j\in\cH^T_j$ converges smoothly to $\pi^\star\p$ with $\p\in\cH^T_X$. By~\eqref{equ:mabj}, we have 
$$
\mab_j^\rel(\f_j)=\ent_{j,v}(\f_j)+\enR_{j,v}(\f_j)+\en_{j,vw\ell_j}(\f_j). 
$$
Since $\f_j\to\pi^\star\p$ smoothly, we have  $\enR_{j,v}(\f_j)\to\enR_{Y,v}(\pi^\star\p)$, and $\MA_{j,v}(\f_j)\to\MA_{Y,v}(\pi^\star\p)$ smoothly, which implies 
$$
\ent_{j,v}(\f_j)=\tfrac12\Ent(\MA_{j,v}(\f_j)|\nu_Y)\to\ent_{Y,v}(\pi^\star\p)=\tfrac12\Ent(\MA_{Y,v}(\pi^\star\p)|\nu_Y).
$$
Since $\ell_j\to\ell_X$ (see Lemma~\ref{lem:ellextcv}), we further have 
$$
\en_{j,vw\ell_j}(\f_j)\to\en_{\pi^\star\Om_X,v w\ell_X}(\pi^\star\p)=\en_{X,v w\ell_X}(\p),
$$
and we conclude
$$
\mab_j^\rel(\f_j)\to\ent_{Y,v}(\pi^\star\p)+\enR_{Y,v}(\pi^\star\p)+\en_{X,v w\ell_X}(\p), 
$$
which is equal to
$$
\ent_{X,v}(\p)+\enR_{X,v}(\p)+\en_{X,v w\ell_X}(\p)=\mab^\rel_X(\p),
$$
thanks to Lemma~\ref{lem:MY}. 
\end{proof}

%
%
%
%
\subsection{Strong semicontinuity of weighted functionals}\label{sec:semicontfunc}
In what follows, we establish various semicontinuity results with respect to strong convergence in the sense of~\S\ref{sec:strong}. Recall that $\om_j\to\pi^\star\om_X$ smoothly with $\om_j\ge(1-\e_j)\pi^\star\om_X$, and that a sequence $\f_j\in\cET_j=\cE^1(Y,\om_j)^T$ converges strongly to $\pi^\star\p\in\cE^1(Y,\pi^\star\om_X)^T=\pi^\star\cET_X$ with $\p\in\cET_X$ iff 
$$
\f_j\to\pi^\star\p\text{ in }L^1(Y)\text{ and }\en_{\om_j}(\f_j)\to\en_{\pi^\star\om_X}(\pi^\star\p)=\en_{\om_X}(\p).
$$
We first establish:

\begin{lem}\label{lem:entlsc} Assume $\cET_j\ni\f_j\to\pi^\star\p$ strongly with $\p\in\cET_X$. Then $\MA_{j,v}(\f_j)\to\MA_{Y,v}(\pi^\star\p)=\pi^\star\MA_{X,v}(\p)$ weakly on $Y$. 
\end{lem} 
\begin{proof} Set $\mu_j:=\MA_{j,v}(\f_j)$, $\mu:=\MA_{Y,v}(\pi^\star\p)$, and note that $\mu_j$ is a positive measure whose total mass $\mu_j(Y)=\int_Y v(m_j)\om_j^n$ converges to $\mu(Y)=\int_X v(m_X)\om_X^n$. By weak compactness, we may thus assume that $\mu_j$ converges weakly a positive measure $\tilde\mu$, necessarily of same mass as $\mu$, and we then need to show $\tilde\mu=\mu$. 

Pick $f\in C^\infty(X)$. We claim that $\int_Y \pi^\star f\mu_j\to\int_Y\pi^\star f\,\mu$. Taking this for granted, we then get $\pi_\star\tilde\mu=\pi_\star\mu$, and hence $\tilde\mu=\mu$ outside the exceptional locus of $\pi$. Since $\mu$ is puts no mass on the exceptional locus and $\mu,\tilde\mu$ have same mass, it follows that $\tilde\mu$ also puts no mass on the exceptional locus, and hence $\tilde\mu=\mu$ on $Y$. 

To prove the claim, we may assume, after scaling $f$, that the latter is $\tfrac 12\om$-psh, and hence that $\pi^\star f$ is $\om_j$-psh for all $j$. By~\eqref{equ:MAvhold}, the functional $\cET_j\ni\f\mapsto F_j(\f):=\int_Y \pi^\star f\,\MA_{j,v}(\f)$ then satisfies a uniform H\"older estimate 
$$
\left|F_j(\f)-F_j(\f')\right|\le C \dd_{1,j}(\f,\f')^\a\max\{\dd_{1,j}(\f,0),\dd_{1,j}(\f',0)\}^{1-\a}. 
$$
For $\tau\in\cH^T_X$, 
$$
F_j((1-\e_j)\pi^\star\tau)=\int_Y\pi^\star f\,\MA_{j,v}((1-\e_j)\pi^\star\tau)
$$
further converges to 
$$
\int_X f\,\MA_{X,v}(\tau)=\int_Y \pi^\star f\,\MA_{Y,v}(\pi^\star\tau)=:F(\pi^\star\tau). 
$$ 
By Lemma~\ref{lem:strongcont}, we infer that $F_j(\f_j)=\int_Y \pi^\star f\,\mu_j$ converges to $F(\pi^\star\p)=\int_Y\pi^\star f\,\mu$, which proves the claim and concludes the proof. 
\end{proof}

Next we establish the strong continuity of the weighted Monge--Amp\`ere energy: 

\begin{lem}\label{lem:wencv} Consider a sequence $(g_j)$ of smooth functions on $\ft^\vee$, converging uniformly on compact sets to a smooth function $g$. Then:
\begin{itemize}
\item[(i)] for each sequence $\f_j\in\cET_j$ converging strongly to $\pi^\star\p$ with $\p\in\cET_X$, we have $\en_{j,g_j}(\f_j)\to\en_{X,g}(\p)$; 
\item[(ii)] there exists $A>0$ such that 
\begin{equation}\label{equ:wenbd}
\left|\en_{j,g_j}(\f)\right|\le A\dd_{1,j}(\f,0)
\end{equation}
for all $j$ and $\f\in\cET_j$. 
\end{itemize}
\end{lem} 
\begin{proof} Since $P_j\to P_X$, \eqref{equ:wenlip} yields a uniform Lipschitz estimate $|\en_{j,g_j}(\f)-\en_{j,g_j}(\p)|\le C\dd_{1,j}(\f,\p)$ for $\f,\p\in\cET_j$. This proves (ii), and also implies (i) by Lemma~\ref{lem:strongcont}, since any function in $\pi^\star\cET_X$ is a decreasing limit of functions in $\pi^\star\cH^T_X$, and $\en_{j,g_j}((1-\e_j)\pi^\star\p)\to\en_{\pi^\star\Om_X,g}(\pi^\star\p)=\en_{X,g}(\p)$ for $\p\in\cH^T_X$. 
\end{proof}

We now turn to the twisted weighted energy with respect to an equivariant current $\Theta=(\theta,m_\Theta)$ on $Y$. Recall that we have set $\en_{j,v}^\Theta(\f):=\en_{\Om_j,v}^\Theta(\f)$ for $\f\in\cH^T_j$. For $\f\in\cH^T_X$ we similarly set
$$
\en_{Y,v}^\Theta(\pi^\star\p):=\en_{\pi^\star\Om_X,v}^\Theta(\pi^\star\p). 
$$
Assume first that $\Theta$ is smooth. Since each $\Om_j$ is K\"ahler, Proposition~\ref{prop:wenE1} yields a unique continuous extension $\en_{j,v}^\Theta\colon\cET_j\to\R$, which further satisfies a H\"older estimate
\begin{equation*}
\left|\en_{j,v}^\Theta(\f)-\en_{j,v}^\Theta(\p)\right|\lesssim B_j \,\dd_{1,j}(\f,\p)^\a\max\{\dd_{1,j}(\f,0),\dd_{1,j}(\p,0)\}^{1-\a}
\end{equation*}
for all $\f,\p\in\cET_j$, where $\a:=2^{-n}$ and
$$
B_j:=\|\theta\|_{\om_j}\sup_{P_j}|v|+(\sup_X\|m_\Theta\|)(\sup_{P_j}\|v'\|). 
$$
In particular, 
\begin{equation}\label{equ:twenbd}
\left|\en_{j,v}^\Theta(\f)\right|\lesssim B_j \,\dd_{1,j}(\f,0).
\end{equation}


Since $P_j\to P_X$, we have 
$$
\sup_{P_j}|v|\to\sup_{P_X}|v|,\quad \sup_{P_j}\|v'\|\to\sup_{P_X}\|v'\|.
$$
If we further assume 
$$
\|\theta\|_{\pi^\star\om_X}:=\inf\{C\ge 0\mid\pm\theta\le C\pi^\star\om_X\}<\infty,
$$
then Lemma~\ref{lem:strongcont} and~\eqref{equ:twenbd} yield:

\begin{lem}\label{lem:twenext1} Let $\Theta=(\theta,m_\Theta)$ be a (smooth) equivariant form on $Y$ such that $\pm\theta\le C\pi^\star\om_X$ for some $C>0$. Then:
\begin{itemize}
\item[(i)] $\en_{Y,v}^{\Theta}\colon\pi^\star\cH_X^T\to\R$ admits a unique continuous extension $\en_{Y,v}^\Theta\colon\pi^\star\cET_X\to\R$, which satisfies
\begin{equation*}
\left|\en_{Y,v}^\Theta(\pi^\star\p)-\en_{Y,v}^\Theta(\pi^\star\p)\right|\lesssim B\,\dd_{1,X}(\f,\p)^\a\max\{\dd_{1,X}(\f,0),\dd_{1,X}(\p,0)\}^{1-\a}
\end{equation*}
for all $\f,\p\in\cET_X$, with
$B:=\|\theta\|_{\pi^\star\om_X}\sup_{P_X}|v|+(\sup_X\|m_\Theta\|)(\sup_{P_X}\|v'\|)$; 
\item[(ii)] for any strongly convergent sequence $\cET_j\ni\f_j\to\pi^\star\p\in\pi^\star\cET_X$, we have 
$$
\lim_j\en_{j,v}^\Theta(\f_j)=\en_{Y,v}^\Theta(\pi^\star\p); 
$$ 
\item[(iii)] there exists a uniform constant $A>0$ such that 
\begin{equation}\label{equ:twenjbd}
\left|\en_{j,v}^\Theta(\f)\right|\le A \dd_{1,j}(\f,0)
\end{equation}
for all $j$ and $\f\in\cET_j$. 
\end{itemize}
\end{lem}

The above condition $\pm\theta\le C\pi^\star\om_X$ with $C>0$ of course fails in general, \eg when $\theta$ is a K\"ahler form on $Y$. But as it turns out, a one-sided bound is enough to get a semicontinuous extension of the twisted weighted energy. For later reference we allow $\theta$ to be a current in the next result. 

\begin{lem}\label{lem:twenext} Let $\Theta=(\theta,m_\Theta)$ be an equivariant current with $\theta$ $\pi$-semipositive. Then the following properties hold:
\begin{itemize}
   \item[(i)] 
    the functionals $\en_{j,v}^\Theta\colon\cH_j^T\to\R$ and $\en_{Y,v}^\Theta\colon\pi^\star\cH_X^T\to\R$ admit unique usc extensions
    $$
   \en_{j,v}^\Theta\colon\cET_j\to\R\cup\{-\infty\},\quad \en_{Y,v}^\Theta\colon\pi^\star\cET_X\to\R\cup\{-\infty\}
    $$
    that are continuous along decreasing sequences; 
\item[(ii)] for any strongly convergent sequence $\cET_j\ni\f_j\to\pi^\star\p$ with $\p\in\cET_X$ we have 
    $$
    \limsup_j\en_{j,v}^\Theta(\f_j)\le\en_{Y,v}^\Theta(\pi^\star\p);
    $$
\item[(iii)] there exists a constant $A>0$ such that 
\begin{equation}\label{equ:twenjsup}
\en_{j,v}^\Theta(\f)\le A(\dd_{1,j}(\f,0)+1)
\end{equation}
for all $j$ and all $\f\in\cET_j$.
\end{itemize}
\end{lem}

\begin{proof} By Lemma~\ref{lem:twenext1}, the result holds with $\pi^\star\Om_X$ in place of $\Theta$. Since 
$$
\en_{Y,v}^\Theta=\en_{Y,v}^{\Theta+C\pi^\star\Om_X}-C\en_{Y,v}^{\pi^\star\Om_X},\quad\en_{j,v}^\Theta=\en_{j,v}^{\Theta+C\pi^\star\Om_X}-C\en_{j,v}^{\pi^\star\Om_X}
$$
where $\theta+C\pi^\star\om_X\ge 0$ for $C\gg 1$, we may thus assume without loss $\theta\ge 0$. Pick a quasi-psh $T$-invariant function $f$ such that $\Theta':=\Theta-\ddcT f$ is smooth. For each $j$, $\en_{j,v}^{\Theta'}$ and $\en_{j,v}^{\ddcT f}$ respectively admit a unique continuous extension $\en_{j,v}^{\Theta'}\colon\cET_j\to\R$ (see Proposition~\ref{prop:wenE1}) and a unique usc extension $\en_{j,v}^{\ddcT f}\colon\cET_j\to\R\cup\{-\infty\}$ that is continuous along decreasing sequences (see Proposition~\ref{prop:twenext}). Since $\Om_j$ is K\"ahler, $\en_{j,v}^{\Theta'}$ is continuous on $\cET_j$, and 
$\en_{j,v}^\Theta=\en_{j,v}^{\Theta'}+\en_{j,v}^{\ddcT f}$ thus admits a unique usc extension $\en_{j,v}^\Theta\colon\cET_j\to\R\cup\{-\infty\}$ that is continuous along decreasing sequences. 

In order to extend $\en_{Y,v}^\Theta\colon\pi^\star\cH_X^T\to\R$ to $\pi^\star\cET_X$, we will rely on Lemma~\ref{lem:weakext} with $\cF_j=\cE^{1,T}_j$, $\cR=\pi^\star\cH_X^T$, $\cF=\pi^\star\cE^{1,T}_X\simeq\cE^{1,T}_X$. We claim that 
$$
F_j:=\en_{j,v}^\Theta+A\en_{j,v}
$$
is monotone increasing on $\cET_j$ for a constant $A>0$ independent of $j$ (allowed to vary from line to line in what follows). Indeed, since $v>0$ on $P_X=\lim_j P_j$, we can choose $B>0$ independent of $j$ such that
$$
\langle v'(\a),\b\rangle+B v(\a)\ge 0
$$
for all $\a\in P_j$ and $\b\in K$. Here $K\subset\ft^\vee$ is the compact subset provided by Lemma~\ref{lem:enmono} applied to the equivariant form $\Theta'$. Since $\theta=\theta'+\ddc f\ge 0$, we conclude as desired that $\en_{j,v}^\Theta+A\en_{j,v}$ is monotone increasing. 

For any $\f\in\cH_X^T$, the smooth convergence $(1-\e_j)\pi^\star\p\to\pi^\star\p$ implies
$$
\en^\Theta_{j,v}\left((1-\e_j)\pi^\star\p\right)\to\en^\Theta_{Y,v}(\pi^\star\p),\quad\en_{j,v}((1-\e_j)\pi^\star\p)\to \en_{Y,v}(\pi^\star\p). 
$$
We can thus apply Lemma~\ref{lem:weakext} to $F:=\en_{Y,v}^\Theta+A\en_{Y,v}: \pi^\star\cH_X^T\to \R$, and get a unique extension to $F\colon\pi^\star\cET_X\to\R\cup\{-\infty\}$ that is continuous along decreasing sequences, weakly usc, monotone increasing, translation equivariant and such that
$$
\limsup_j F_j(\f_j)\leq F(\f)
$$
for any weakly convergent sequence $\cE^{1,T}_j\ni \f_j\to \f\in \pi^\star\cE^{1,T}_X$. Since $\en_{Y,v}$ is strongly continuous on $\pi^\star\cET_X$, setting $\en_{Y,v}^{\Theta}:=F-A\en_{Y,v}$ on $\pi^\star\cET_X$ yields the desired extension for which (i) and (ii) hold.

To prove (iii), we may assume $\f\in\cH_j^T$, by continuity of $\en_{j,v}^\Theta$ along decreasing sequences in $\cET_j$. Since $|\en_{j,v}(\f)|\le B \dd_{1,j}(\f,0)$ by~\eqref{equ:wenbd}, it is enough to prove the estimate for the monotone increasing functional $F_j=\en_{j,v}^{\Theta}+A\en_{j,v}$. The equivariance properties of $\en_{j,v}^\Theta$ and $\en_{j,v}$ yield 
$$
F_j(\f+c)=F_j(\f)+a_j c
$$
for $\f\in\cH_j^T$ and $c\in\R$, where 
$$
a_j:=\int_Y\left(\MA_{j,v}^{\Theta}(0)+C\MA_{j,v}(0)\right)\to\int_Y\left(\MA_{\pi^\star\Om_X,v}^{\Theta}(0)+C\MA_{\pi^\star\Om_X,v}(0)\right)
$$
as $j\to\infty$, and hence $|a_j|\le B$. Since $F_j$ is monotone increasing and vanishes at $0$, we have $F_j(\f-\sup\f)\le 0$, and hence 
$$
F_j(\f)\le a_j\sup\f\le B|\sup\f|\lesssim B (V_j^{-1} \dd_{1,j}(\f,0)+T_{\om_j}),
$$
by Lemma~\ref{lem:supd}. Finally, Lemma~\ref{lem:weakcomp} shows that $T_{\om_j}$ is bounded, and (iii) follows. 
\end{proof}

%

%
%
%
%
\subsection{Proof of main theorem}\label{sec:proofmain}
Assuming from now on $\pi\colon Y\to X$ to be of Fano type, we work towards the proof of Theorem~\ref{thm:opencoermab}. Pick a quasi-psh function $f$ on $Y$ such that $\hnu_Y=e^{-2f}\nu_Y$ has finite total mass and $\Ric(\hnu_Y)=\Ric(\nu_Y)+\ddc f$ is $\pi$-semipositive (see Definition~\ref{defi:Fanotype}). After averaging with respect to the Haar measure of $T$, we may further assume that $f$ is $T$-invariant, and normalize it by 
\begin{equation}\label{equ:normf}
\int_Y f\MA_{Y,v}(0)=0. 
\end{equation}
As a first key property we show: 
\begin{lem}\label{lem:hqpsh} We have $\hnu_Y=\pi^\star\left(e^{-2\rho}\nu_X\right)$ for a quasi-psh function $\rho$ on $X$ such that $\int_X\rho\MA_{X,v}(0)=0$. 
\end{lem}
\begin{proof} Set $\hro:=\rho_{Y/X}+f$, which satisfies $\int_Y\hro\,\pi^\star\MA_{X,v}(0)=0$ by~\eqref{equ:normY} and~\eqref{equ:normf}. By~\eqref{equ:Jpi} we have $\hnu_Y=e^{-2\hro}\pi^\star\nu_X$, and our goal is to show that $\hro=\pi^\star\rho$ for a quasi-psh function $\rho$ on $X$. Pick $C>0$ such that 
$$
\Ric(\nu_Y)+\ddc f\ge -C\pi^\star\om_X.
$$
By~\eqref{equ:RicTh} and~\eqref{equ:LPrho}, we have 
$$
\pi^\star\Ric(\nu_X)-\Ric(\nu_Y)+\ddc\rho_{Y/X}=2\pi[D],
$$
and hence 
\begin{equation}\label{equ:inD}
\pi^\star\eta_X+\ddc\hro\ge 2\pi[D]
\end{equation}
with $\eta_X:=\Ric(\nu_X)+C\om_X\in\cZ(X)$. Pick a small open subset $U\subset X$ such that $\eta_X=\ddc\tau$ for a smooth function $\tau\in C^\infty(U)$. Since $D$ is $\pi$-exceptional, \eqref{equ:inD} shows that $\pi^\star\tau+\hro$ induces a psh function $v$ on $U$ outside an analytic subset of codimension at least $2$. By normality of $U$, $v$ uniquely extends to a psh function $v$ on $U$, see~\cite{GR}. This shows that $\hro=\pi^\star\rho$ on $\pi^{-1}(U)$ for the quasi-psh function $\rho:=v-\tau$, and the result follows by uniqueness of such a function $\rho$. 
\end{proof}
Introduce next the equivariant current 
\begin{equation}\label{equ:hRic}
\Ric^T(\hnu_Y):=\Ric^T(\nu_Y)+\ddcT f. 
\end{equation}
Since $\Ric(\hnu_Y)$ is $\pi$-semipositive, Lemma~\ref{lem:twenext} shows that setting
$$
\henR_{j,v}(\f):=-\en_{j,v}^{\Ric^T(\hnu_Y)}(\f),\quad\henR_{Y,v}(\pi^\star\p):=-\en_{Y,v}^{\Ric^T(\hnu_Y)}(\pi^\star\p)
$$
for $\f\in\cET_j$, $\p\in\cET_X$ defines lsc functionals
$$
\henR_{j,v}\colon\cET_j\to\R\cup\{+\infty\},\quad\henR_{Y,v}\colon\pi^\star\cET_X\to\R\cup\{+\infty\}
$$
that are continuous along decreasing sequences. By~\eqref{equ:hRic} and Proposition~\ref{prop:twenext}, we have
\begin{equation}\label{equ:henR}
\henR_{j,v}(\f)=\enR_{j,v}(\f)-\int_Y f\MA_{j,v}(\f)+c_j,\quad\f\in\cET_j, 
\end{equation}
where 
$$
c_j:=\int_Y f\MA_{j,v}(0)\to\int_Y f\MA_{Y,v}(0)=0,
$$
see~\eqref{equ:normf}. On the other hand, Lemma~\ref{lem:hqpsh} yields
\begin{equation}\label{equ:henR2}
\henR_{Y,v}(\pi^\star\p)=\enR_{X,v}(\p)-\int_X\rho\MA_{X,v}(\p),\quad\p\in\cH^T_X,
\end{equation}
see Lemma~\ref{lem:3.14revisited}. Using this, we next establish the following uniform entropy growth. 

\begin{lem}\label{lem:entgrowth} 
There exists $\d,C>0$ such that
    $$
    \mab^\rel_j(\f)\geq \d\ent_j(\f)-C\left(\dd_{1,j}(\f,0)+1\right)
    $$
    for all $j$ large enough and $\f\in \cET_j$.
\end{lem}
\begin{proof} By Lemma~\ref{lem:twenext} we have $\henR_{j,v}(\f)\ge-C(\dd_{1,j}(\f,0)+1)$ for $\f\in\cET_j$, where $C>0$ denotes a uniform constant that is allowed to vary from line to line. Combined with~\eqref{equ:henR} this yields
\begin{equation}\label{equ:enRlow}
\enR_{j,v}(\f)\ge\int_Yf\MA_{j,v}(\f)-C(\dd_{1,j}(\f,0)+1).
\end{equation}
On the other hand, we can find $p>1$ such that $e^{-2p f}\in L^1(Y)$, by the openness theorem~\cite{GuanZhou}. For $\f\in\cET_j$, Lemma \ref{lem:entexp} yields
\begin{equation}\label{equ:entlow}
\frac 1p\ent_{j,v}(\f)=\frac{1}{2p}\Ent\left(\MA_{j,v}(\f)|e^{-2p f}\nu_Y\right)-\int_Y f\MA_{j,v}(\f)\ge -C-\int_Yf\MA_{j,v}(\f)  
\end{equation}
since the relative entropy is uniformly bounded below, see~\eqref{equ:entbd}. Finally, the smooth convergence $vw\ell_j\to vw\ell_X$ (see Lemma~\ref{lem:ellextcv}) yields 
$$
\left|\en_{j,vw\ell_j}(\f)\right|\le C\dd_{1,j}(\f,0). 
$$
Combining this with~\eqref{equ:enRlow} and~\eqref{equ:entlow}, we conclude  
\begin{align*}
\mab^\rel_j(\f) & =\ent_{j,v}(\f)+\enR_{j,v}(\f)+\en_{j,vw\ell_j}(\f) \\
& \ge (1-p^{-1})\ent_{j,v}(\f)-C\big(\dd_{1,j}(\f,0)+1\big). 
\end{align*}
Finally, Lemma \ref{lem:entvgr} yields $\ent_{j,v}(\f)\ge\d\ent_j(\f)-C$ for uniform constants $\d,C>0$, and the result follows. 
 \end{proof}

\begin{rmk} Dropping the Fano type condition on $\pi$, we can always find a quasi-psh function $f$ on $Y$ such that $\Ric(\nu_Y)+\ddc f$ is $\pi$-semipositive, simply because $c_1(Y)+C\{\pi^\star\omega_X\}$ is big for $C\gg 1$, and hence contains a closed positive current. By Skoda's integrability theorem we then have $e^{-p f}\in L^2(Y)$ for $0<p\ll 1$, and the proof of Lemma~\ref{lem:entgrowth} now yields, for any resolution of singularities $\pi\colon Y\to X$, a uniform entropy bound
$$
\mab^\rel_j(\f)\ge-C\left( \ent_j(\f)+\dd_{1,j}(\f,0)+1\right),\quad\f\in\cET_j. 
$$
\end{rmk}
Introduce now lsc functionals
$$
\hent_{j,v}\colon\cET_j\to\R\cup\{+\infty\},\quad\hent_{Y,v}\colon\pi^\star\cET_X\to\R\cup\{+\infty\}
$$
by setting 
$$
\hent_{j,v}(\f):=\tfrac12\Ent\left(\MA_{j,v}(\f)|\hnu_Y\right),\quad\hent_{Y,v}(\pi^\star\p):=\tfrac 12\Ent\left(\MA_{Y,v}(\pi^\star\p)|\hnu_Y\right)
$$
for $\f\in\cET_j$, $\p\in\cET_X$. Since $\hnu_Y=e^{-2f}\nu_Y=\pi^\star(e^{-2\rho}\nu_X)$ where $f$ and $\rho$ are usc, Lemma~\ref{lem:entexp} yields
\begin{equation}\label{equ:hent}
\ent_{j,v}(\f)=\hent_{j,v}(\f)-\int_Y f\MA_{j,v}(\f),\quad\ent_{X,v}(\p)=\hent_{Y,v}(\pi^\star\p)-\int_X\rho\MA_{X,v}(\p). 
\end{equation}
Using this we prove the following version of Lemma~\ref{lem:MY}. 

\begin{lem}\label{lem:MYext} For all $\p\in\cET_X$ we have 
$$
\hent_{Y,v}(\pi^\star\p)+\henR_{Y,v}(\pi^\star\p)\ge\ent_{X,v}(\p)+\enR_{X,v}(\p).
$$
\end{lem}
\begin{proof} When $\p\in \cH^T_X$, \eqref{equ:henR2} and~\eqref{equ:hent} yield 
$$
\hent_{Y,v}(\pi^\star\p)+\henR_{Y,v}(\pi^\star\p)=\ent_{X,v}(\p)+\enR_{X,v}(\p). 
$$
In the general case, write $\p\in\cET_X$ as the limit of a decreasing sequence $\p^k\in \cH^T_X$. By the previous step we get
 $$
 \enR_{X,v}(\p^k)-\henR_{Y,v}(\pi^\star\p^k)=\hent_{Y,v}(\pi^\star\p^k)-\ent_{X,v}(\p^k)=\int_X\rho\,\MA_{X,v}(\p^k), 
$$
using again~\eqref{equ:henR2}. Since $\p^k$ decreases to $\p$, we have $ \enR_{X,v}(\p^k)\to  \enR_{X,v}(\p)$, $\henR_{Y,v}(\pi^\star\p^k)\to \henR_{Y,v}(\pi^\star\p)$ (see Lemma~\ref{lem:twenext}), and $\MA_{X,v}(\p^k)\to \MA_{X,v}(\p)$ weakly as measures on $X$. As $\rho$ is usc, we infer
$$
 \enR_{X,v}(\p)-\henR_{Y,v}(\pi^\star\p)\leq \int_X\rho\,\MA_{X,v}(\p)= \hent_{Y,v}(\pi^\star\p)-\ent_{X,v}(\p), 
$$
which concludes the proof.
 \end{proof}

\begin{rmk} As we saw in the proof, equality holds in Lemma~\ref{lem:MYext} when $\p\in\cH^T_X$. More generally, equality holds if $\p\in\cET_X$ can be written as the limit of a sequence $\p_j\in\cH^T_X$ such that $\ent_{X,v}(\p_j)\to\ent_{X,v}(\p)$ (\eg when $X$ is smooth, see~\cite[Lemma~6.16]{AJL}). Indeed, we then have $\hent_{Y,v}(\pi^\star\p_j)+\henR_{Y,v}(\pi^\star\p_j)=\ent_{X,v}(\p_j)+\enR_{X,v}(\p_j)$, where the right-hand side converges to $\ent_{X,v}(\p)+\enR_{X,v}(\p)$, while $\ent_{Y,v}(\pi^\star\p)\le\liminf_j\ent_{Y,v}(\pi^\star\p_j)$, $\enR_{Y,v}(\pi^\star\p)\le\liminf_j\enR_{Y,v}(\pi^\star\p_j)$ since $\ent_{Y,v}$ and $\enR_{Y,v}$ are lsc. 
\end{rmk} 

We are now in a position to conclude the proof of Theorem~\ref{thm:opencoermab}. Indeed, consider the functionals 
$$
M_j\colon\cET_j\to\R\cup\{+\infty\},\quad M\colon\pi^\star\cET_X\to\R\cup\{+\infty\}
$$
respectively defined by
$$
M_j(\f):=\mab^\rel_j(\f)=\ent_{j,v}(\f)+\enR_{j,v}(\f)+\en_{j,vw\ell_j}(\f),\quad\f\in\cET_j, 
$$
and
$$
M(\pi^\star\p):=\mab^\rel_X(\p)=\ent_{X,v}(\p)+\enR_{X,v}(\p)+\en_{vw\ell_X}(\p),\quad\p\in\cET_X. 
$$
By Theorem~\ref{thm:coeropen}, it suffices to show that the conditions of~\S\ref{sec:opencoer} are satisfied for the group $G=T_\C$. Translation invariance and $T_\C$-invariance are automatic (see~\S\ref{sec:relmab}), and so is the properness of the $T_\C$-action on $\cET_j$, since $\om_j$ is K\"ahler and $Y$ is smooth (see Proposition~\ref{prop:isom}). The normalization condition $M_j(0)\to M(0)$ follows from Proposition~\ref{prop:cvmab}, the convexity of $M_j$ along psh geodesics in $\cET_j$ is proved in~\cite[Theorem~6.1]{AJL} (see Proposition~\ref{prop:relmabE1}), and the entropy growth condition on $M_j$ was proved in Lemma~\ref{lem:entgrowth}. 

It remains to take care of lower semicontinuity. Assume that $\f_j\in\cE^1_j$ converges strongly to $\pi^\star\p$ with $\p\in\cET_X$. For each $j$, \eqref{equ:henR} and~\eqref{equ:hent} yield
$$
M_j(\f_j)=\hent_{j,v}(\f_j)+\henR_{j,v}(\f_j)+\en_{j,vw\ell_j}(\f_j)+c_j
$$
with $c_j\to 0$. Since $\MA_{j,v}(\f_j)\to\MA_{Y,v}(\pi^\star\p)$ weakly (see Lemma~\ref{lem:entlsc}), the lower semicontinuity of the relative entropy $\Ent(\cdot|\hnu_Y)$ yields 
$$
\liminf_j\hent_{j,v}(\f_j)\ge\hent_{Y,v}(\pi^\star\p),
$$
while Lemma~\ref{lem:twenext} implies 
$$
\liminf_j\henR_{j,v}(\f_j)\ge\henR_{Y,v}(\pi^\star\p).
$$
Combined with Lemma~\ref{lem:MYext}, this yields
$$
\liminf_j\left(\hent_{j,v}(\f_j)+\henR_{j,v}(\f_j)\right)\ge\hent_{Y,v}(\pi^\star\p)+\henR_{Y,v}(\pi^\star\p)\ge \ent_{X,v}(\p)+\enR_{X,v}(\p).
$$
On the other hand, since $\ell_j\to\ell_X$ (see Lemma~\ref{lem:ellextcv}), Lemma~\ref{lem:wencv} yields 
$$
\lim_j\en_{j,vw\ell_j}(\f_j)=\en_{X,vw\ell_X}(\p). 
$$
We thus get the desired lower semicontinuity property
$$
\liminf_j M_j(\f_j)\ge M(\pi^\star\p), 
$$
which concludes the proof of Theorem~\ref{thm:opencoermab}. 

\appendix
%
%
\section{Relative entropy}\label{sec:ent}
Let $X$ be an arbitrary compact (Hausdorff) topological space, with a fixed positive Radon measure $\nu$. Recall that the \emph{relative entropy} of any positive Radon measure $\mu$ on $X$ with respect to $\nu$ is defined as
$$
\Ent(\mu|\nu)=\int_X f\log f\,d\nu\in\R\cup\{+\infty\}
$$
if $d\mu=f d\nu$, $f\in L^1(\nu)$, and $\Ent(\mu|\nu)=+\infty$ otherwise. It can be expressed as a Legendre transform 
\begin{equation}\label{equ:entleg}
\Ent(\mu|\nu)=\sup_{g\in C^0(X)}\left\{\int g\,d\mu-\mu(X)\log\int e^g d\nu\right\}+\mu(X)\log\mu(X), 
\end{equation}
which shows that
$$
\Ent(\cdot|\nu)\colon\cM\to\R\cup\{+\infty\}
$$
is convex and lsc on the space $\cM$ of positive Radon measures (equipped with the weak topology), and satisfies 
\begin{equation}\label{equ:entbd}
\Ent(\mu|\nu)\ge\mu(X)\log\frac{\mu(X)}{\nu(X)}.
\end{equation}
As in~\cite[Lemma 2.11]{BBEGZ}, we note that~\eqref{equ:entleg} remains valid when $g$ is merely lsc and bounded below (resp.~usc and bounded above), by monotone approximation by continuous functions. 


\begin{lem}\label{lem:entcst} Assume $\mu,\mu'$ are positive Radon measures such that $\mu'\le A\mu$ with $A>0$. Then 
$$
\Ent(\mu'|\nu)\le A\Ent(\mu|\nu)+C
$$
with $C>0$ only depending on $A,\nu(X),\mu(X)$. 
\end{lem}
\begin{proof} Write $d\mu=f\,d\nu$ with $f\in L^1(\nu)$ and $d\mu'=g d\mu$ with $g\in L^1(\mu)$. Then $0\le g\le A$, and hence
$$
\Ent(\mu'|\nu)=\int_X (fg)\log(fg)\,d\nu=\int_X (fg)\log f\,d\nu+\int_X g\log g\,d\mu. 
$$
$$
\le A\int_{\{f\ge 1\}} f\log f\,d\nu+A\log A\,\mu(X)\le A\Ent(\mu|\nu)+A\nu(X)+A\log A\,\mu(X),
$$
using $x\log x\ge -1$ for $x\in (0,1)$. 
\end{proof}

\begin{lem}\label{lem:entexp} Consider another positive Radon measure $\nu'$ on $X$ such that $d\nu'=e^\rho d\nu$ with $\rho\colon X\to\R\cup\{+\infty\}$ lsc. For any positive Radon measure $\mu$, we then have 
\begin{equation}\label{equ:entexp}
\Ent(\mu|\nu)=\Ent(\mu|\nu')+\int \rho\,d\mu.
\end{equation}
\end{lem}
\begin{proof} After scaling $\mu$, we may assume without loss $\mu(X)=1$. Pick $g\in C^0(X)$. Applying~\eqref{equ:entleg} to the lsc function $g+\rho$ yields
$$
\int(g+\rho)\,d\mu\le\log\int e^{g+\rho}\,d\nu+\Ent(\mu|\nu)=\log\int e^g\,d\nu'+\Ent(\mu|\nu). 
$$
Thus
$$
\int g\,d\mu-\log\int e^g\,d\nu'+\int \rho\,d\mu\le \Ent(\mu|\nu),
$$
and hence 
$$
\Ent(\mu|\nu')+\int \rho\,d\mu\le \Ent(\mu|\nu),
$$
by~\eqref{equ:entleg}. Similarly, \eqref{equ:entleg} applied to the usc function $g-\rho$ yields 
$$
\int(g-\rho)\,d\mu\le\log\int e^{g-\rho}\,d\nu'+\Ent(\mu|\nu')=\log\int e^g\,d\nu+\Ent(\mu|\nu'). 
$$
Adding $\int \rho\,d\mu$ to both sides, we get
$$
\int g\,d\mu-\log\int e^g\,d\nu\le \Ent(\mu|\nu')+\int \rho\,d\mu,
$$
and~\eqref{equ:entleg} yields
$\Ent(\mu|\nu)=\Ent(\mu|\nu')+\int \rho\,\mu$. 

\end{proof}

We will also rely on the following basic result. 

\begin{lem}\label{lem:ent} Pick a convergent sequence $f_j\to f$ in $L^1(\nu)$, and assume that 
$$
C_p:=\sup_j\int_X \exp(p|f_j|)\,d\nu<+\infty
$$
for each $p>0$. For each $C>0$, we then have $\int|f_j-f|\,d\mu\to 0$ uniformly with respect to all positive measures $\mu$ such that $\Ent(\mu|\nu)\le C$. 
\end{lem}
\begin{proof} Note first that $\int_X \exp(p|f|)\,d\nu\le C_p$. Indeed, after passing to a subsequence we have $f_j\to f$ $\nu$-a.e., which yields the estimate by Fatou's lemma. Since 
$$
\exp(p|f_j-f|)\le \exp(2p\max\{|f_j|,|f|\})\le \exp(2p|f_j|)+\exp(2p|f|),
$$
we may thus assume without loss $f=0$. 

Since $\int \exp(|f_j|)\,d\nu$ is bounded, $(f_j)$ is bounded in $L^2(\nu)$, and we claim $\|f_j\|_{L^2(\nu)}\to 0$. 
Let $L$ be a limit point of $\int|f_j|^2\,d\nu$. We need to show $L=0$. After passing to a subsequence, we may assume $\int|f_j|^2\,d\nu\to L$, and also $f_j\to 0$ $\nu$-a.e, since $f_j\to 0$ in $L^1(\nu)$. As $\int\exp(|f_j|)\,d\nu$ is bounded, the sequence $|f_j|^2$ is uniformly integrable, and $f_j\to 0$ $\nu$-a.e.\ thus implies $\int|f_j|^2\,d\nu\to 0$, \ie $L=0$. 

Next, pick a positive measure $\mu$ such that $\Ent(\mu|\nu)\le C$, and write $d\mu=g\,d\nu$ with $g\in L^1(\nu)$. Arguing as in the proof of~\cite[Theorem~2.17]{BBEGZ}, we introduce the convex function $\chi(x)=(x+1)\log(x+1)-x$ on $\R_+$, with Legendre conjugate 
$$
\chi^\star(y)=\sup_{x\ge 0}\{xy-\chi(x)\}=e^y-y-1\le y e^y
$$
on $\R_+$. Since $\chi(x)-x\log x$ is bounded above on $\R_+$, we have $\int \chi(g)\,d\nu\le C'$ where $C'$ only depends on $C$. For each $p>0$ we have 
$$
|f_j| g=p^{-1}\left(p|f_j|\right) g\le p^{-1} \chi^\star(p|f_j|)+p^{-1}\chi(g)\le |f_j| \exp(p|f_j|)+p^{-1}\chi(g), 
$$
and hence 
$$
\int|f_j|\,d\mu=\int |f_j|g\,d\nu\le \int |f_j| \exp(p|f_j|)d\nu+p^{-1}C'\le \|f_j\|_{L^2(\nu)}C_{2p}^{1/2}+p^{-1} C',
$$
by Cauchy--Schwarz. For any $\e>0$, we can choose $p\gg 1$ such that $p^{-1}C'\le\e$. For all $j$ large enough we then have $\|f_j\|_{L^2(\nu)}C_{2p}^{1/2}\le\e$, and hence $\int |f_j|\,\mu\le 2\e$. The result follows. 
\end{proof}

%
%
%

%
%
%
%
%
%
\end{document}